\documentclass[10pt]{article}
\usepackage[latin1]{inputenc}

\usepackage{enumitem}
\usepackage{soul}
\usepackage{amsthm}
\usepackage{amsmath,amssymb,wasysym,amsbsy}
\usepackage{bm}
\usepackage{graphicx}
\usepackage{subfigure}
\usepackage{multicol,array,lscape}
\usepackage{nonfloat}
\usepackage[small,bf]{caption2} 
\usepackage{breqn}
\usepackage[dvipsnames,svgnames,x11names]{xcolor}
\usepackage{pst-all}
\usepackage{pst-3dplot}
\usepackage{authblk}
\usepackage{cancel} % para tachar f?rmulas
\usepackage{soul} % para tachar texto
\usepackage[%
  colorlinks=true,%
	linkcolor	=blue,%
	citecolor =blue,%
  urlcolor=blue%
]{hyperref}
\usepackage[left=3cm,right=3cm, top=3.3cm, bottom=3cm]{geometry}
\newtheorem{theo}{Theorem}[section]
\newtheorem{lemm}[theo]{Lemma}

\newtheorem{coro}[theo]{Corolary}
\newtheorem{prop}[theo]{Proposition}

\newtheorem{defi}[theo]{Definition}
\newtheorem{remark}[theo]{Remark}

\numberwithin{equation}{section}
\title{Convergence and positivity of finite element methods for a haptotaxis model of tumoral invasion}
\author[a]{Viviana Ni\~no-Celis}
\author[a]{Diego A. Rueda-G\'omez}
\author[a]{\'Elder J. Villamizar-Roa\thanks{Corresponding author.
\href{mailto:jvillami@uis.edu.co}{jvillami@uis.edu.co}  (E. J. Villamizar-Roa).}}
\affil[a]{Universidad Industrial de Santander, Escuela de Matem\'{a}ticas, A.A. 678, Bucaramanga, Colombia.}

\DeclareRobustCommand{\uvec}[1]{{%
  \ifcsname uvec#1\endcsname
     \csname uvec#1\endcsname
   \else
    \bm{\hat{\mathbf{#1}}}%
   \fi
}}
\date{}
\begin{document}
\maketitle
%%%%%%%%%%%%%% ABSTRACT %%%%%%%%%%%%%% ABSTRACT %%%%%%%%%%%%%%%%%%%%%%%%%%%%%%%
\begin{abstract}
In this paper, we consider a mathematical model for the invasion of host tissue by tumour cells in a $d$-dimensional bounded domain, $d\leq 3$. This model consists of a system of differential equations describing the evolution of cancer cell density, the extracellular matrix protein density and the matrix degrading enzyme concentration.
We develop two fully discrete schemes for approximating the solutions based on the Finite Element (FE) method. For the first numerical scheme, we use a splitting technique to deal with the 
haptotaxis term, leading to introduce an equivalent system with a new variable given by the gradient of extracellular matrix. This scheme is well-posed and preserves the 
non-negativity of extracellular matrix and the degrading enzyme.  We analyze error estimates and convergence towards regular solutions.  The second numerical scheme is based on an equivalent formulation in which the  cancer cell density equation is expressed in a divergence form through a suitable change of variables. This second numerical scheme preserves the non-negativity of all the discrete variables. Finally, we present some numerical simulations in agreement with the theoretical analysis. 

\vspace{0.3cm}

\noindent{\bf Keywords.} Haptotaxis, tumoral invasion, finite elements, convergence rates, error estimates, positivity. \vspace{0.3cm}

\noindent{\bf AMS subject classifications.}  {35Q92; 92C50; 92C15; 65M12; 65M15; 65M60}
\end{abstract}

%%%%%%%%%%%%%  INTRODUCTION %%%%%%%%%%%%%%% INTRODUCTION %%%%%%%%%%%%%%%%%%%%%%
\section{Introduction}
{Tissue invasion represents one of the most critical steps in cancer metastasis, which is characterized essentially by four hallmark features, namely, the cancer cell adhesion to the surrounding tissue or extracellular matrix, the secretion of the matrix degrading enzymes and the degradation of extracellular matrix, the migration of the cancer cells, and the proliferation of tumor cells. In particular, the degradation of the extracellular matrix by the degrading enzymes produces gradients of non-diffusible molecules within extracellular matrix which direct the movement of invasive cells. This mobility mechanism is known as {\it Haptotaxis}.\\ 

In order to describe the cancer invasion mechanism, a variety of mathematical models have been proposed, see for instance \cite{Anderson, Bellomo,Chaplain,Chaplain1,Gatenby,Gerisch,Lachowicz,Lachowicz1,Perumpanani,Szymanska,Szymanska1}. Gatenby and Gawlinski \cite{Gatenby} considered a reaction-diffusion model to examine the tumor invasion of the surrounding tissue suggesting that the cancer cells produce lactic acid toxic which 
alter the microenvironment of the normal tissue, generating space for tumor cells to proliferate and invade the surrounding tissue. Perumpanani and Byrne \cite{Perumpanani} suggest that  other two mechanisms for invasion of the surrounding tissue are the proteases and the haptotactic movement of the cancer cells. The proteases production depends on the tumor cell density and the collagen gel concentration \cite{Tao}. In Anderson {\it et al.} \cite{Anderson}, the authors presented two mathematical models to describe the invasion of extracellular matrix by tumour cells. These models consider the interaction between three variables, namely, the extracellular matrix, the tumour cells and the matrix degrading enzymes. The first model focuses on the macro-scale structure (cell population level) and consider the tumour as a single mass; meanwhile, the second model focusses on the micro-scale (individual cell) level, which uses a discrete technique  to model the migration and invasion at the level of individual cells, in order to examine the implications of the metastatic spread. Anderson and Chaplain \cite{Chaplain} also proposed a mathematical model to describe the interactions between the tumor and the surrounding tissue. The cancer cells produce degrading enzymes to debase the extracellular matrix and originates the movement. An extension of the model proposed in \cite{Chaplain} was presented by Chaplain and Lolas \cite{Chaplain1}. Lachowicz \cite{Lachowicz} also proposed two mathematical models of tissue invasion of tumours, which are defined at micro and meso-scale levels of description. Mathematical relationships among these possible descriptions are formulated.
More recently, several models have incorporated new ingredients in the modelling of tumour invasion, including the cell-cell adhesion and cell-matrix adhesion \cite{Chaplain2,Gerisch}, the competition for space \cite{Szymanska}, the influence of heat shock proteins \cite{Szymanska1}, and so on.\\

The generic mathematical model proposed by Anderson {\it et al} \cite{Anderson} to describe the interaction between the cancer cell density (denoted by $u$), the extracellular matrix protein density (denoted by $v$), and the degrading enzyme concentration (denoted by $m$), is given by the following system of differential equations:
\begin{equation}\label{KNS0}
\left\{
\begin{array}{lc}
\partial_{t} u=D_{u}\Delta {u}-\nabla\cdot(\chi(v)u\nabla v) + F_1(u,v,m),&\\
\partial_{t}v=F_2(v,m),  &\\
\partial_{t}m=D_{m}\Delta m+F_3(u,v,m),&
\end{array}
\right.
\end{equation}
in $\Omega\times (0,T),$ $\Omega\subset\mathbb{R}^d,$ $d\leq 3,$ and $0<T\leq \infty.$  The parameters $D_{u}, D_{m}$ represent the diffusion coefficients of the cancer cells and the degrading enzyme,  and the nonlinear term  $-\nabla \cdot(\chi (v)u \nabla v)$ represents diffusion by haptotaxis. The function $\chi(\cdot)$ is called the sensitivity function which describes the sensitivity of the cancer cells to the gradient of the extracellular matrix, meanwhile functions $F_1,F_2$ and $F_3$ represent possible interactions between the variables. Depending on the kind of interaction between $u,v,$ and $m,$ several submodels of (\ref{KNS0}) have been considered recently; although these models are simplifications of real biophysical context, their solutions display complex dynamics and their mathematical analysis is challenging. In this sense, an interesting particular case of the generic model (\ref{KNS0}) is given by the following system  \cite{Marciniak}:
\begin{equation}\label{KNS}
\left\{
\begin{array}{lc}
\partial_{t} u=D_{u}\Delta {u}-\nabla\cdot(\chi(v)u\nabla v) + \mu_{u}u(1-u-v),&\\
\partial_{t}v=-\alpha mv,  &\\
\partial_{t}m=D_{m}\Delta m-\rho_{m}m+ \mu_{m}uv.&
\end{array}
\right.
\end{equation}
The term $\mu_{u}u(1-u-v)$ represents the proliferation of cancer cells which follows a logistic growth law accounting for the competition for space; the term $-\rho_{m}m+\mu_{m}uv$ indicates that there exists an interaction between cancer cells and the extracellular matrix in the production of degrading enzimes,  and a self-degradation of enzymes,
with some proportionality rates $\mu_{m}\geq 0$ y $\rho_{m}\geq 0$, 
respectively. Finally, the ODE (\ref{KNS})$_2$ describes the dynamic of the extracellular matrix; it is assumed that the extracellular matrix is degraded upon contact with the  degrading enzyme secreted by the cancer cells at the rate $\alpha>0,$ 
and  there is no spatial transport of the extracellular matrix.
System (\ref{KNS}) is completed with the following initial and boundary data:
\begin{equation}\label{initialdata}
\left\{
\begin{array}{lc}
\left[u(0,x),v(0,x),m(0,x)\right]=\left[u_0(x),v_0(x),m_0(x)\right],\ x\in\Omega,\\[.3cm]
D_u\frac{\partial u(x,t)}{\partial \boldsymbol{\nu}}-\chi(v)u\frac{\partial v(x,t)}{\partial \boldsymbol{\nu}}=D_m\frac{\partial m(x,t)}{\partial \boldsymbol{\nu}}=0, \quad x \in\partial\Omega,\quad t\in(0,T),
\end{array}
\right.
\end{equation}
where $\boldsymbol{\nu}$ denotes the unit outward normal vector to the boundary. A particular case of system (\ref{KNS0}) obtained considering linear kinetics of the extracellular matrix, that is, $F_3=-\rho_{m}m+ \mu_{m}v$, the sensitivity function being a positive constant, and considering only spatial transport of cancer cells, that is, $F_1=0$, was analyzed by Morales-Rodrigo in \cite{Rodrigo}. There, by using the Schauder fixed point theorem, were proved the existence and uniqueness of local solutions in the class of H\"older spaces. A simplified system of (\ref{KNS})-(\ref{initialdata}) of two equations with $F_1=0$ was also considered by Corrias, Perthame and Zaag \cite{Corrias}. They analyzed the existence of global solutions in the framework of $L^p$-spaces, under smallness of initial data. Szyma\'nska {\it et al.} \cite{Szymanska} considered a model with nonlocal (integral) cell kinetics, and proved the existence of global solutions without imposing any smallness conditions on the initial data. The complete system (\ref{KNS})-(\ref{initialdata}) was analyzed by Marciniak-Czochra and Ptashnyk in \cite{Marciniak}. The authors proved the existence and uniform boundedness of global solutions by showing a priori estimates for the supremum norm and using the method of bounded invariant rectangles applied to the reformulated system in divergence form with a diagonal diffusion matrix.\\

Although the qualitative analysis of (\ref{KNS})-(\ref{initialdata}) is quite aceptable, from a numerical point of view there is a significant gap. Indeed, as far as we know, the literature related to the numerical analysis of haptotaxis systems is scarce. We only known some numerical simulations in order to investigate the pattern formation and predict numerically the nonlinear dynamic of the some haptotaxis systems, 
mainly focused on the one-dimensional case, see for instance \cite{Anderson,Chaplain2,Gatenby,Hillen,Marciniak,Perumpanani,Szymanska1,Zhigun}.\\ 

Taking into account the lack of numerical analysis to approximate the solutions of haptotaxis models, the aim of this paper is to propose two numerical schemes to approximate the strong solutions of (\ref{KNS})-(\ref{initialdata}), and develop the underlying numerical analysis. The  main difficulties to deal with the numerical analysis of (\ref{KNS})-(\ref{initialdata}) come from the strong coupling nonlinear  term $-\nabla\cdot(\chi(v)u\nabla v).$ Indeed, it is not clear how to perform a convergence order analysis in a FE scheme based  on the classical variational formulation, since using the $v-$equation (\ref{KNS})$_2$, it is not possible to control this nonlinear term. Thus, in order to overcome this difficulty, we use a splitting technique to deal with the 
haptotaxis term in the cancer cell density equation, leading to introduce an equivalent system with a new variable given by the gradient of extracellular matrix. This idea allows us to propose a first fully discrete numerical scheme based on the Finite Element (FE) method, which is well-posed and preserves the  non-negativity of the discrete variables of the extracellular matrix and the degrading enzyme. For this scheme we analyze error estimates and convergence towards regular solutions. On the other hand, 
based on an equivalent formulation proposed in \cite{Marciniak} to prove the existence and boundedness of global solutions, in which the density of cancer cell density equation 
is expressed in a divergence form through a suitable change of variables, we propose a second numerical scheme which is well-posed and preserve the non-negativity for all the discrete variables. Some numerical simulations valide the theoretical analysis and show that, in general, both numerical schemes have a similar behavior. As far as we know, this paper is almost the only existing one dedicated to the analysis of numerical schemes for this haptotaxis problem.\\ 

The layout of this paper is as follows: In Section 2, we recall some existence and uniqueness results of (\ref{KNS})-(\ref{initialdata}) in the continuous case. We also define an equivalent formulation of (\ref{KNS})-(\ref{initialdata}), which will be used to construct the first numerical approximation. In Section 3, we define the first numerical scheme for approximating the solutions of (\ref{KNS})-(\ref{initialdata}). That numerical scheme is constructed by using FE approximations in space and finite differences in time; we first develop some preliminaries, and establish the properties of well-posedness and positivity. We end Section 3 establishing the second numerical scheme which is motivated by the equivalent weak formulation given in \cite{Marciniak}, which behaves well from the point of view of the positivity for the discrete cancer cell density. 
In Section 4, we obtain some uniform estimates and subsequently we develop the convergence analysis. In Section 5, we provide some numerical simulations in agreement with the {theoretical} results. 

{
\section{The continuous problem}
In this section, we establish a variational formulation  of (\ref{KNS})-(\ref{initialdata}) and two equivalent formulations of (\ref{KNS})-(\ref{initialdata}) which will be used to construct the numerical schemes. After establishing the definition of weak solution, we recall  some existence results for $d$-dimensional bounded domains ($d\leq 3$) obtained in \cite{Marciniak}. We start recalling some basic notations that will be used through this paper.  We use the standard Sobolev and Lebesgue spaces $W^{k,p}(\Omega)$ and $L^p(\Omega),$ with respective norms $\Vert \cdot\Vert_{W^{k,p}}$ and $\Vert \cdot\Vert_{L^p}.$ In particular, we denote $W^{k,2}(\Omega)=H^k(\Omega).$ The $L^2(\Omega)$-inner product will be represented by $(\cdot,\cdot).$ Corresponding Sobolev spaces of vector valued functions will be denoted by ${\bf W}^{k,p} (\Omega),$ ${\bf L}^{p} (\Omega),$ and so on.  
% Also, we will consider the space
%\begin{eqnarray*}%\label{spaceH}
%\mathbf{H}^{1}_{\nu}(\Omega)&:=&\{\mathbf{u}\in \mathbf{H}^{1}(\Omega): \mathbf{u}\cdot \boldsymbol{\nu}=0 \mbox{ on } \partial\Omega\},
%\end{eqnarray*}
%Frequently,  we will use the following classical interpolation inequality
%\begin{equation}\label{in3D}
%\Vert u\Vert_{L^3}\leq C\Vert u\Vert_{L^2}^{1/2}\Vert u\Vert_{L^6}^{1/2} \  \ \forall u\in H^1(\Omega) \ \ \mbox{ (in 3D domains)}.
%\end{equation}
It is important to mention that the letters $C,C_i,K_i$ will denote different positive
constants independent of discrete parameters which may change from line to line (or even within the same line). Now we are in position to recall the definition of weak solution of (\ref{KNS})-(\ref{initialdata}). From now on, $\Omega$ is a bounded domain of $\mathbb{R}^d,$ $d\leq 3,$ with boundary $\partial\Omega$ smooth enough. 
\begin{defi}\label{weak1}{\bf (Weak solution of (\ref{KNS})-(\ref{initialdata}))}
	A weak solution of (\ref{KNS})-(\ref{initialdata}) is a triple $[u,v,m]$ of functions satisfying $u, v, m \in L^{2}(0,T;H^{1}(\Omega))$, $u, v \in L^{\infty}(0,T;L^\infty(\Omega))$ and $u_{t}, v_{t}, m_{t} \in L^{2}(0,T;L^2(\Omega))$ such that
	\begin{flalign*}
	& \ \  \int_{0}^{T}\int_{\Omega}(\partial_t u\varphi_{1}+D_{u}\nabla u\cdot\nabla\varphi_{1}-\chi(v)u\nabla v\cdot \nabla \varphi_{1})dxdt=\mu_{u}\int_{0}^{T}\int_{\Omega} u(1-u-v)\varphi_{1}dxdt,\\
	& \ \ \int_{0}^{T}\int_{\Omega}(\partial_tv\varphi_{2}+\alpha mv \varphi_{2})dxdt=0,  \\
	& \ \  \int_{0}^{T}\int_{\Omega}(\partial_tm\varphi_{3}+D_{m}\nabla m\cdot\nabla \varphi_{3} +\rho_{m}m \varphi_{3})dxdt=\mu_{m}\int_{0}^{T}\int_{\Omega}uv\varphi_{3} dxdt,
	\end{flalign*}
for all $\varphi_{1} \in L^{2}(0,T;H^{1}(\Omega))$, $\varphi_{2} \in L^{2}(0,T;L^2(\Omega))$, $\varphi_{3} \in L^{2}(0,T;H^{1}(\Omega))$ and $u,v,m$ satisfy initial conditions (\ref{initialdata}), i.e. $u\rightarrow u_{0}, v\rightarrow v_{0}, m \rightarrow m_{0}$ in $L^{2}(\Omega)$ as $t\rightarrow 0$.
\end{defi}
In order to get the existence and boundedness of global weak solutions of {(\ref{KNS})}-{(\ref{initialdata})}, in {\cite{Marciniak} (see also \cite{Corrias}), the authors considered an equivalent system where the first equation in {(\ref{KNS})} is expressed in a divergence form. Explicitly, defining the auxiliary variable $s=\frac{u}{\phi(v)},$ where $\phi(v)=\exp(\frac{1}{D_{u}}\int_{0}^{v}\chi(v')dv'),$ the system {(\ref{KNS})} is rewritten as follows: 
\begin{equation}\label{Chemoweak3}
\begin{split}
\left \{
\begin{array}{lccl}
\phi(v)\partial_{t}s=D_{u}\nabla\cdot(\phi(v)\nabla s)+s\phi(v)\left(\alpha\frac{\chi (v)}{D_{u}}vm+ \mu_{u}-\mu_{u}s\phi(v)-\mu{_u}v\right), \\
\partial_{t}v=-\alpha mv,\\ 
\partial_{t}m=D_{m}\Delta m-\rho_{m}m+\mu_{m}s\phi(v)v,
\end{array}
\right.
\end{split}
\end{equation}
with initial and boundary conditions

\begin{equation}\label{initialdata3}
\left\{
\begin{array}{lc}
\left[s(0,x),v(0,x),m(0,x)\right]=\left[s_0(x)=\frac{u_{0}(x)}{\phi(v_{0}(x))},v_0(x),m_0(x)\right],\ x\in\Omega,\\[.3cm]
D_u\frac{\phi(v)\partial s(x,t)}{\partial \boldsymbol{\nu}}=D_m\frac{\partial m(x,t)}{\partial \boldsymbol{\nu}}=0, \quad x \in\partial\Omega,\quad t\in(0,T).
\end{array}
\right.
\end{equation}
The notion of weak solution considered for (\ref{Chemoweak3})-(\ref{initialdata3}) is the following one:}

\begin{defi}\label{weak2}{\bf (Weak solution of (\ref{Chemoweak3})-(\ref{initialdata3}))} 
	The triple $(s,v,m)$ is called a weak solution of the model (\ref{Chemoweak3})-(\ref{initialdata3}) if $s, v, m \in L^{2}(0,T;H^{1}(\Omega))$, $v \in L^{\infty}(0,T; L^\infty(\Omega))$ and $s_{t}, v_{t}, m_{t} \in L^{2}(0,T;L^2(\Omega))$ such that
	\begin{flalign*}
	& \ \  \int_{0}^{T}\int_{\Omega} (\phi(v)\partial_{t}s\varphi_{1}+D_{u}\phi(v)\nabla s\cdot \nabla \varphi_{1}) dxdt  \nonumber\\
	& \ \ \hspace{1cm} =\frac{\alpha}{D_{u}}\int_{0}^{T}\int_{\Omega} \chi(v)s\phi(v)vm\varphi_{1} dxdt+\mu_{u}\int_{0}^{T}\int_{\Omega}s\phi(v)(1-s\phi(v)-v)\varphi_{1} dxdt, \\ 
	& \ \ \int_{0}^{T}\int_{\Omega}(\partial_{t}v\varphi_{2}+\alpha mv\varphi_{2})dxdt=0,  \\
	& \ \  \int_{0}^{T}\int_{\Omega}(\partial_{t}m\varphi_{3}+D_{m}\nabla m\cdot\nabla\varphi_{3}+\rho_{m}m\varphi_{3})dxdt=\mu_{m}\int_{0}^{T}\int_{\Omega}s\phi(v)v\varphi_{3} dxdt, 
	\end{flalign*}
for all $\varphi_{1} \in L^{2}(0,T;H^{1}(\Omega))$, $\varphi_{2} \in L^{2}(0,T;L^2(\Omega))$, $\varphi_{3} \in L^{2}(0,T;H^{1}(\Omega)) $ and $s,v,m$ satisfy initial conditions (\ref{initialdata3}),{ i.e. $s\rightarrow s_{0}, v\rightarrow v_{0}, m \rightarrow m_{0}$ in $L^{2}(\Omega)$ as $t\rightarrow 0$.} 
\end{defi}
As pointed out in \cite{Marciniak}, if $u,v\in L^\infty(0,T;L^\infty(\Omega)),$ the existence of weak solutions solutions of system of {(\ref{Chemoweak3})}-{(\ref{initialdata3})} (in the sense of Definition \ref{weak2}) is equivalent to the existence of weak solutions of {(\ref{KNS})}-{(\ref{initialdata})} (in the sense of Definition \ref{weak1}).
In \cite{Marciniak}, by using the Schauder fixed point theorem, the following local existence of solutions of {(\ref{Chemoweak3})}-{(\ref{initialdata3})} was proved.
\begin{theo}\label{exis1}(\cite[ Theorem 3.1.]{Marciniak})
{For $s_{0} \geq 0$, $m_{0} \geq 0$, $v_{0} \geq 0$, $s_{0}$, $v_{0}$, $m_{0} \in H^{1}(\Omega)$, $v_{0} \in L^{\infty}(\Omega)$ and a continuous and positive $\chi$, there exists a local in time, non-negative weak solution of system {(\ref{Chemoweak3})}-{(\ref{initialdata3})} in the sense of Definition \ref{weak2}}
\end{theo}
Next, by using the ``bounded invariant rectangles'' method, the following global existence and boundedness theorem for {(\ref{KNS})}-{(\ref{initialdata})} was proved in \cite{Marciniak}. 
\begin{theo}\label{exis2}(\cite[Theorems 3.2 and 3.3.]{Marciniak})
{For non-negative and bounded initial data $u_{0}$, $v_{0}$, $m_{0} \in H^{1}(\Omega)$, and a continuous and positive function $\chi$, there exists a global solution of system {(\ref{KNS})}-{(\ref{initialdata})}, in the sense of Definition \ref{weak1}, and it is uniformly bounded. In addition, if $\chi$ is locally Lipschitz-continuous, and the initial data satisfy $u_0,v_0,m_0\in L^\infty(\Omega),$ $\nabla v_0, \nabla m_0\in L^q(\Omega),$ $q\geq d,$ $\nabla u_0\in L^2(\Omega),$ the weak solution is unique, and $v\in L^\infty(0,T;W^{1,q}(\Omega)),$ $m\in L^q(0,T;W^{1,q}(\Omega)).$}
\end{theo}

For regular initial data, the weak solutions are regular; more exactly, the following regularity result can be established.
\begin{theo}\label{TRG}
Under the assumptions of Theorem \ref{exis2}, if $\chi\in C^1$ and $u_0,m_0,v_0\in C^2(\overline{\Omega}),$ then the weak solution provided by Theorem \ref{exis1} is classical. 
\end{theo}
\begin{proof}
The proof is essentially in \cite{Marciniak}, Section 5.2. For convenience of the reader, we detail it here. From the regularity of weak solutions provided by Theorem \ref{exis2}, it holds, in particular,  that $\nabla m\in L^q(0,T;L^q(\Omega)).$ Then, differentiating the $v$-equation in (\ref{KNS}) with respect to $x_i,$ $i=1,...,n,$ testing the
obtained equation by $\vert v_{x_i}\vert^{q-2}v_{x_i}$ and integrating in time one gets $v\in L^{\infty}(0,T;W^{1,q}(\Omega))$ (cf. Lemma 5.3 in \cite{Marciniak}); moreover,
\begin{eqnarray*}
\sup_{0\leq t\leq T}\Vert \nabla v(t)\Vert^q_{L^q}\leq C\Vert \nabla v_0\Vert^q_{L^q}+C\Vert v\Vert_{L^\infty(L^\infty)}\int_0^T\int_\Omega\vert \nabla m\vert^q dxdt.
\end{eqnarray*}
From the regularity of weak solutions provided by Theorem \ref{exis2}, it holds, in particular,  that $-\rho_{m}m+ \mu_{m}uv\in L^q(0,T;L^q(\Omega)).$ Thus, from the parabolic regularity applied to the $m$-equation in (\ref{KNS}) (see \cite{Feireisl}, Theorem 10.22, p. 344), one has $m\in L^q(0,T;W^{2,q}(\Omega)).$ 
Differentiating the $v$-equation in (\ref{Chemoweak3}) with respect to $x_i$ and $x_j$ and testing the
obtained equation by $\vert v_{x_ix_j}\vert^{q-2}v_{x_ix_j}$ we can obtain that $v \in L^q(W^{2,q}(\Omega))$ (cf. \cite{Marciniak}, Lemma 5.4), and the following estimate holds:
\begin{eqnarray*}
\sup_{0\leq t\leq T}\Vert D^2v(t)\Vert_{L^q}^q\leq \Vert D^2v_0\Vert_{L^q}^q+C\Vert v\Vert_{L^\infty(L^\infty)}\int_0^T\int_\Omega\left(\vert \nabla m\vert^{2q}+\vert \nabla v\vert^{2q}+\vert D^2 m\vert^{q}\right)dxdt.
\end{eqnarray*}
Since $v\in L^\infty(0,T;W^{2,q}(\Omega))$ and $u\in L^2(0,T;H^1(\Omega)),$ from the equality $s=\frac{u}{\phi(v)},$ one has that $\nabla s\in L^2(0,T;L^2(\Omega)).$ Thus, rewriting  the $s$-equation in (\ref{Chemoweak3}) as
\begin{eqnarray}\label{est10}
\partial_{t}s-D_{u}\Delta s=D_u\frac{\phi'(v)}{\phi(v)}\nabla v\nabla s+s\left(\alpha\frac{\chi (v)}{D_{u}}vm+ \mu_{u}-\mu_{u}s\phi(v)-\mu{_u}v\right),
\end{eqnarray}
one gets that the right hand side of (\ref{est10}) belongs to $L^2(0,T;L^2(\Omega)).$ Consequently, by parabolic regularity (see \cite{Feireisl}, Theorem 10.22, p. 344) one deduces $s\in L^2(0,T;H^2(\Omega))\cap L^\infty(0,T;H^1(\Omega));$ thus, $\nabla s\in L^2(0,T;L^6(\Omega))\cap L^\infty(0,T;L^2(\Omega)).$ By Lebesgue interpolation one gets
$\nabla s\in L^q(0,T;L^{\frac{6q}{3q-4}}(\Omega)).$ Applying parabolic regularity again, one gets $s\in L^q(0,T;W^{2,\frac{6q}{3q-4}}(\Omega)),$ that is, $\nabla s\in L^q(0,T;W^{1,\frac{6q}{3q-4}}(\Omega)).$ If $q\leq 4,$ the last regularity implies, by Sobolev embeddings,  that $\nabla s\in L^q(0,T;L^q(\Omega)).$ Otherwise, if $q>4,$ one has $\nabla s\in L^q(0,T; L^{\frac{6q}{q-4}}(\Omega)).$ Then, applying parabolic regularity once more, one gets $s\in L^q(0,T;W^{2,\frac{6q}{q-4}}(\Omega)),$ that is, $\nabla s\in L^q(0,T;W^{1,\frac{6q}{q-4}}(\Omega)),$ which implies that $\nabla s\in L^q(0,T;L^q(\Omega)).$ Now, differentiating equations (\ref{Chemoweak3})$_1$ with respect to $t,$ testing the obtained equation by $\vert \partial_ts\vert^{q-2}\partial_ts$ and taking into account that $\nabla s\in L^q(0,T;L^q(\Omega)),$ it holds that $\partial_ts\in L^\infty(0,T;L^q(\Omega))$ (cf. \cite{Marciniak}, Lemma 5.5).  Moreover, the following estimate is true
\begin{equation}
\sup_{0\leq t\leq T}\Vert \partial_ts(t)\Vert^q_{L^q}\leq C\left(\Vert [v,m,s]\Vert_{L^\infty(L^\infty)},\Vert s_0\Vert_{W^{2,q}}\right)+ C\Vert [v,m]\Vert_{L^\infty(L^\infty)}\int_0^T\int_\Omega\vert \nabla s\vert^qdxdt.\label{est1}
\end{equation}
Analogously, differentiating (\ref{Chemoweak3})$_3$ with respect to $t,$ testing the obtained equation by $\vert \partial_tm\vert^{q-2}\partial_tm,$ and using (\ref{est1}), it holds that $\partial_tm\in L^\infty(0,T;L^q(\Omega))$ and
\begin{equation}
\sup_{0\leq t\leq T}\Vert \partial_tm(t)\Vert^q_{L^q}\leq C\left(\Vert [v,m,s]\Vert_{L^\infty(L^\infty)},\Vert m_0\Vert_{W^{2,q}}\right)+C\Vert [v,m]\Vert_{L^\infty(L^\infty)}\int_0^T\int_\Omega\vert \partial_ts\vert^qdxdt.\label{est2}
\end{equation} 
Taking into account that $u=\phi(v)s,$ from (\ref{est1}) it holds that $\partial_tu\in L^\infty(0,T;L^q(\Omega)).$ On the other hand, isolating the terms $\Delta m$ and $\Delta s$ in (\ref{Chemoweak3}) and (\ref{est10}) respectively, and using (\ref{est1})-(\ref{est2}), it is straightforward to prove that $s,m\in L^\infty(0,T;W^{2,q}(\Omega))$ (cf. \cite{Marciniak}, Lemma 5.6). Furthermore, the following estimate holds:
\begin{equation*}
\sup_{0\leq t\leq T}\Vert \Delta s(t)\Vert^q_{L^q}\leq C\Vert [v,m,s]\Vert_{L^\infty(L^\infty)}\left( 1+\sup_{0\leq t\leq T}\Vert \nabla v(t)\Vert_{L^{2q}}^{2q}+\sup_{0\leq t\leq T}\Vert \partial_ts(t)\Vert^q_{L^q}\right),
\end{equation*}
\begin{equation*}
\sup_{0\leq t\leq T}\Vert \Delta m(t)\Vert^q_{L^q}\leq C\Vert [v,m,s]\Vert_{L^\infty(L^\infty)}\left( 1+\sup_{0\leq t\leq T}\Vert \partial_tm(t)\Vert^q_{L^q}\right).
\end{equation*}
Since $m,s\in L^\infty(0,T;W^{2,q}(\Omega))$ and $\partial_tm,\partial_ts\in L^\infty(0,T;L^q(\Omega))$, then, in particular one has that $m,s\in C([0,T];C(\overline{\Omega})).$  
Notice that the right hand side of (\ref{Chemoweak3})$_3$ belongs to $C([0,T];C(\overline{\Omega})).$ Thus, since $m_0\in C^{2}(\overline{\Omega}),$ from the parabolic regularity (cf. \cite{Feireisl}, Theorem 10.23.) one gets $m\in C([0,T];C^{2}(\overline{\Omega}))$ and $\partial_tm\in C([0,T];C(\overline{\Omega})).$ Now, since 
$$ D_u\frac{\phi'(v)}{\phi(v)}\nabla v+\left(\alpha\frac{\chi (v)}{D_{u}}vm+ \mu_{u}-\mu_{u}s\phi(v)-\mu{_u}v\right)\in C([0,T];C(\overline{\Omega})),$$
and $s_0\in C^{2}(\overline{\Omega}),$ from the parabolic regularity (cf. \cite{Ladys}) one has $s\in C([0,T];C^{2}(\overline{\Omega}))$ and $\partial_ts\in C([0,T];C(\overline{\Omega})).$ Since $s=\frac{u}{\phi(v)}$ one can conclude that
\begin{eqnarray*}
u\in C([0,T];C^{2}(\overline{\Omega})),\ \partial_{t}u\in C([0,T];C(\overline{\Omega})),
\end{eqnarray*}
which conclude that the global solution is classical.
\end{proof}

A strong difficulty to deal with system {(\ref{KNS})}-{(\ref{initialdata})} comes from the second order nonlinear term in the cancer cell density equation. Thus, in order to control it numerically, following the ideas in \cite{Abel}, we introduce the new variable $\boldsymbol{\sigma}=\nabla v,$ which allows us to propose a FE scheme for which we can analyze convergence rates. Specifically, we consider the following variational formulation:
%he variational form of $\boldsymbol{\sigma}=\nabla v$ is given by:
%\begin{eqnarray}\label{s1}
%({\boldsymbol{\sigma}},\bar{\boldsymbol{\sigma}})+(v,\nabla\cdot \bar{\boldsymbol{\sigma}})=0,\ \forall\ \bar{\boldsymbol{\sigma}}\in \mathbf{H}^{1}_{\nu}(\Omega). 
%\end{eqnarray}
%Then, differentiating in time equation (\ref{s1}), replacing $\partial_{t}v=-\alpha mv,$ and integrating by parts the resulting equation, we get a evolution equation for the variable ${\boldsymbol{\sigma}}$ which leads to the following weak formulation of {(\ref{KNS})}-{(\ref{initialdata})}:
\begin{equation}\label{split}
\begin{split}
\left \{
\begin{array}{lccl}
 (\partial_{t}m,\bar{m})+D_{m}(\nabla m,\nabla \bar{m})+\rho_{m}(m,\bar{m})=\mu_{m}(uv,\bar{m}),\\
\partial_tv=-\alpha mv, \\
(\partial_{t}u,\bar{u})+D_{u}(\nabla u,\nabla\bar{u}) = (\chi(v)u{\boldsymbol\sigma},\nabla \bar{u}) +\mu_{u}(u - u^2-uv,\bar{u}), \\
(\partial_{t}{\boldsymbol{\sigma}},\bar{\boldsymbol{\sigma}}) + \alpha(m{\boldsymbol{\sigma}},\bar{\boldsymbol{\sigma}})=-\alpha (v\nabla m,\bar{\boldsymbol{\sigma}}), 
\end{array}
\right.
\end{split}
\end{equation}
for all $[\bar{m},\bar{u},\bar{\boldsymbol{\sigma}}]\in H^1(\Omega)\times H^1(\Omega)\times\mathbf{H}^{1}(\Omega)$, where the equation (\ref{split})$_4$ was obtained by applying the gradient operator to equation (\ref{KNS})$_2$. Then, the following result holds:
\begin{lemm}
If a triple $[m,v,u]$ is a classical solution of {(\ref{KNS})}-{(\ref{initialdata})}, then $[m,v,u,{\boldsymbol{\sigma}}]$ is a smooth solution of (\ref{split}). 
Reciprocally, if $[m,v,u,{\boldsymbol{\sigma}}]$ is a smooth solution of (\ref{split}), then $[m,v,u]$ is a classical solution of {(\ref{KNS})}-{(\ref{initialdata})}.
\end{lemm}
\begin{proof}
If $[m,v,u]$ is a classical solution of {(\ref{KNS})}-{(\ref{initialdata})}, then defining ${\boldsymbol{\sigma}}=\nabla v,$ previous procedure to get (\ref{split}) shows that $[m,v,u,{\boldsymbol{\sigma}}]$ is a smooth solution of (\ref{split}). Reciprocally, assume that $[m,v,u,{\boldsymbol{\sigma}}]$ is a smooth enough solution of (\ref{split}). Then, computing the gradient in the ODE for $v,$ and subtracting the result from the ${\boldsymbol{\sigma}}$-equation (\ref{split})$_4,$ one gets
\begin{equation*}
\begin{split}
\left \{
\begin{array}{lccl}
\partial_{t}({\boldsymbol{\sigma}}-\nabla v)=-\alpha m({\boldsymbol{\sigma}}-\nabla v),\\
({\boldsymbol{\sigma}}-\nabla v)(0)=0,
\end{array}
\right.
\end{split}
\end{equation*}
which implies that $\nabla v={\boldsymbol{\sigma}}.$ Finally, replacing $\nabla v={\boldsymbol{\sigma}}$ in (\ref{split})$_3,$ one can conclude that $[m,v,u]$ is a classical solution of {(\ref{KNS})}-{(\ref{initialdata})}.
\end{proof}
}

\section{Definition of the numerical schemes}\label{NScheme}
 In this section, we construct  two numerical schemes approaching the weak solutions of the haptotaxis for  invasion system (\ref{KNS}) with initial and boundary data (\ref{initialdata}). We use a mixed approximation by applying finite element approximations in space and finite differences in time. For simplicity, we assume a uniform partition of $[0,T]$ with time step
$\Delta t=T/N : (t_n = n\Delta t)_{n=0}^{n=N}.$  

\subsection{Scheme \textbf{UVM$\sigma$}}
We construct the first scheme by considering the system (\ref{split}), in which the auxiliary variable ${\boldsymbol\sigma}=\nabla v$ is introduced. Then, for the space discretization, we consider conforming FE spaces:
$\mathcal{X}_{m} \times  \mathcal{X}_{v} \times \mathcal{X}_{u}\times \mathcal{X}_{\boldsymbol{\sigma}}  \subset H^{1}(\Omega)^{3}\times \mathbf{H}^1(\Omega),$ corresponding to a family of shape-regular and quasi-uniform  triangulations of $\overline{\Omega}$, $\{\mathcal{T}_h\}_{h>0}$,  made up of simplexes $K$ (triangles in 2D and tetrahedra in 3D), such that $\overline{\Omega}= \cup_{K\in \mathcal{T}_h} K$, where $h = \max_{K\in \mathcal{T}_h} h_K$, with $h_K$ being the diameter of $K$. A possibility for choosing  the discretization is to approximate the spaces $[\mathcal{X}_{m}, \mathcal{X}_v, \mathcal{X}_u, \mathcal{X}_{\boldsymbol{\sigma}}]$ by $\mathbb{P}_{r_1}\times\mathbb{P}_{r_2}\times \mathbb{P}_{r_3}\times\mathbb{P}_{r_4}$- continuous FE, with $r_1=1$ and $r_i\geq 1$ ($i=2,3,4$).

\begin{remark}
The condition $r_1=1$ is needed to obtain an adequate formulation for $m$ in order to guarantee the non-negativity  for the discrete solution (see (\ref{scheme1})$_1$,  Remark \ref{eqh2} and Lemma \ref{pospos}).
\end{remark}

\subsubsection{Interpolation Operators}\label{INOP1}
From now on, we consider the  interpolation operator $\mathbb{P}_u:H^1(\Omega)\rightarrow \mathcal{X}_u,$
such that for all $u\in {H}^1(\Omega)$, $\mathbb{P}_u u \in \mathcal{X}_u,$ satisfies
\begin{equation}\label{Interp1New}
(\nabla
(\mathbb{P}_u u - u), \nabla\bar{u})  + (\mathbb{P}_u u - u, \bar{u})=0,\quad \forall \bar{u}\in \mathcal{X}_u.
\end{equation}
It is not difficult to see that the interpolation operator $\mathbb{P}_u$ is well defined as consequence of the Lax-Milgram Theorem. Moreover, it is well known that the following interpolation error holds: 
\begin{equation}\label{aprox01} 
\| u-\mathbb{P}_u u\|_{L^2}+ h \|  u-\mathbb{P}_u u  \|_{H^1} \leq C h^{r_3+1} \|u\|_{H^{r_3+1}}, \ \ \forall u \in H^{r_3+1}(\Omega).
\end{equation}
Also, the following stability properties hold 
\begin{equation}\label{aprox01-aNN-new}
\| \mathbb{P}_u u\|_{H^{1}} \leq  \| u\|_{H^1} \qquad \mbox {and} \qquad \| \mathbb{P}_u u\|_{W^{1,6}} \leq C \| u \|_{H^2}.
\end{equation}
Inequality (\ref{aprox01-aNN-new})$_1$ can be deduced from (\ref{Interp1New}), and (\ref{aprox01-aNN-new})$_2$ can be obtained from (\ref{aprox01}) using the inverse inequality 
$$\Vert u_h \Vert_{W^{1,6}}\leq Ch^{-p}\Vert u_h\Vert_{H^1} \ \ \mbox {for all} \ u_h \in \mathcal{X}_u,$$ 
with $p=2/3$ (in the 2D case) and $p=1$ (in the 3D case), and comparing $\mathbb{P}_u$ with an average interpolation of Clement or Scott-Zhang type (which are stable in $W^{1,6}$-norm).  Moreover, we consider interpolation operators $\mathbb{P}_m$, $\mathbb{P}_v$ and $\mathbb{P}_{\boldsymbol{\sigma}}$ such that: $\mathbb{P}_m m_0 \geq 0 $ (if $m_0\geq 0$), $\mathbb{P}_v v_0 \geq 0 $ (if $v_0\geq 0$), and the following approximation properties
\begin{equation}\label{aprox01nn} 
\left\{
\begin{array}{lc}
\| m-\mathbb{P}_m m\|_{L^2}+ h \| m -\mathbb{P}_m m  \|_{H^1} \leq K h^{r_1+1} \|m\|_{H^{r_1+1}}, & \forall m \in H^{r_1+1}(\Omega),\\[.2cm]
\| v-\mathbb{P}_v v\|_{L^2}+ h \|  v-\mathbb{P}_v v  \|_{H^1} \leq K h^{r_2+1} \|v\|_{H^{r_2+1}}, & \forall v \in H^{r_2+1}(\Omega),\\[.2cm]
\Vert {\boldsymbol\sigma}- \mathbb{P}_{\boldsymbol \sigma} {\boldsymbol\sigma}  \Vert_{L^2} + h \Vert {\boldsymbol\sigma} - \mathbb{P}_{\boldsymbol \sigma} {\boldsymbol\sigma} \Vert_{H^1} \leq Ch^{r_4 + 1} \Vert  {\boldsymbol\sigma}\Vert_{H^{r_4+1}}, &  \forall  {\boldsymbol\sigma}\in \mathbf{H}^{r_4 +1}(\Omega),
\end{array}
\right.
\end{equation}
and stability properties 
\begin{equation}\label{aprox01-aNN}
\| [\mathbb{P}_m m, \mathbb{P}_v v, \mathbb{P}_{\boldsymbol\sigma} {\boldsymbol{\sigma}}] \|_{H^{1}} \leq  \| [m,v, {\boldsymbol{\sigma}}] \|_{H^1},
\end{equation}
\begin{equation}\label{aprox01-a}
\| [\mathbb{P}_m m, \mathbb{P}_v v, \mathbb{P}_{\boldsymbol\sigma} {\boldsymbol{\sigma}}] \|_{W^{1,6}} \leq C \| [m,v, {\boldsymbol{\sigma}}] \|_{H^2},
\end{equation}
hold. An example of interpolation operators satisfying the these properties are the nodal interpolation operator or average interpolators of Clement or Scott-Zhang type.

We denote the nodal interpolation operator by $I_h: C(\overline{\Omega})\rightarrow \mathcal{X}_m$, and we introduce the discrete semi-inner product on $C(\overline{\Omega})$ (which is an inner product in $\mathcal{X}_m$) and its induced discrete seminorm (norm in $\mathcal{X}_m$):
\begin{equation}\label{mlump}
(u_1,u_2)^h:=\int_\Omega I_h (u_1 u_2), \   \vert u \vert_h=\sqrt{(u,u)^h}.
\end{equation}
\begin{remark}\label{eqh2}
In $\mathcal{X}_m$, the norms $\vert \cdot\vert_h$ and $\Vert \cdot\Vert_{L^2}$ are equivalent uniformly with respect to $h$ (see \cite{PB}). Moreover, the following property holds for all $u_1, u_2\in \mathcal{X}_m$:
\begin{equation}\label{MassL}
\vert (u_1,u_2)^h - (u_1,u_2)\vert \leq C h \Vert u_1\Vert_{L^2} \Vert \nabla u_2\Vert_{L^2}.
\end{equation}
\end{remark}

\subsubsection{Definition of the scheme}
Considering the weak formulation (\ref{split}), we consider the following first order in time, linear and decoupled numerical scheme (from now on, Scheme \textbf{UVM$\sigma$}):
\\

\textbf{Initialization:} Let $[m_{h}^{0},v_{h}^{0},u_{h}^{0},{\boldsymbol{\sigma}}^0_h]=[\mathbb{P}_m m_0, \mathbb{P}_v v_0, \mathbb{P}_u u_0, \mathbb{P}_{\boldsymbol{\sigma}} {\boldsymbol{\sigma}}_0 ] \in \mathcal{X}_{m} \times \mathcal{X}_{v}\times \mathcal{X}_{u}\times \mathcal{X}_{\boldsymbol{\sigma}}$.\\

\textbf{Time step $n$:} Given the vector $[m_{h}^{n-1},v_{h}^{n-1},u_h^{n-1},{\boldsymbol{\sigma}}^{n-1}_h]\in \mathcal{X}_{m}  \times \mathcal{X}_{v}\times \mathcal{X}_{u} \times \mathcal{X}_{\boldsymbol{\sigma}}$, compute $[m_{h}^{n},v_{h}^{n},u_h^n,{\boldsymbol{\sigma}}^n_h] \in \mathcal{X}_{m}  \times \mathcal{X}_{v}\times \mathcal{X}_{u}\times \mathcal{X}_{\boldsymbol{\sigma}}$ such that 
\begin{flalign}\label{scheme1} 
&1) \ \   (\delta_{t}m^{n}_{h},\bar{m})^{h}+D_{m}(\nabla m^{n}_{h},\nabla \bar{m})+\rho_{m}(m^{n}_{h},\bar{m})^{h}=\mu_{m}([u^{n-1}_{h}]_{+}v^{n-1}_{h},\bar{m}),     \nonumber\\
&2) \ \ \delta_{t}v^{n}_{h} =-\alpha m^{n}_{h}v^{n}_{h},   \nonumber\\
&3) \ \  (\delta_{t}u^{n}_{h},\bar{u})+D_{u}(\nabla u^{n}_{h},\nabla\bar{u}) = (\chi(v^{n}_{h})u^{n-1}_{h}{\boldsymbol\sigma}^{n-1}_{h},\nabla \bar{u}) +\mu_{u}(u^{n-1}_{h} - (u^{n-1}_{h})^2-u^{n}_hv^n_h,\bar{u}), \\
&4)\ \ (\delta_{t}{\boldsymbol{\sigma}}^{n}_{h},\bar{\boldsymbol{\sigma}}) + \alpha(m^{n}_{h}{\boldsymbol{\sigma}}^{n}_{h},\bar{\boldsymbol{\sigma}})=-\alpha (v^n_h\nabla m^n_h,\bar{\boldsymbol{\sigma}}),  \nonumber
\end{flalign}
for all $[\bar{m},\bar{u},\bar{\boldsymbol{\sigma}}] \in\mathcal{X}_m\times \mathcal{X}_{u} \times \mathcal{X}_{\boldsymbol{\sigma}}$; 
 where, in general, we denote $\delta_t z^n_h=\frac{z^n_h-z^{n-1}_h}{\Delta t}$ and $z_+=\max\{z,0\}\geq 0$.\\

\subsubsection{Positivity and Well-posedness}\label{SSPW}
We will prove the well-posedness of the scheme \textbf{UVM$\sigma$}, and non-negativity of the variables $v^n_h$ and $m^n_h$. From now on, we denote in general $a_-=\min\{a,0\}\leq 0$. We remark that it is not possible to prove the positivity of $u_{h}^{n}$ in the scheme  \textbf{UVM$\sigma$}.  
\begin{lemm} {\bf (Positivity of $v^n_h$ and $m^n_h$)}\label{pospos}
Let $([u_{h}^{n},v_{h}^{n},{\boldsymbol{\sigma}}^n_h,m_{h}^{n}])_{n\in\mathbb{N}}$ the sequence defined in scheme \textbf{UVM$\sigma$}. If $v_{h}^{n-1},m_{h}^{n-1}\geq 0,$ then $v_{h}^{n},m_{h}^{n}\geq 0.$
\end{lemm}
\begin{proof} 
Testing (\ref{scheme1})$_1$ by $\bar{m}=I_h([m^n_h]_{-})\in \mathcal{X}_m$ one has
\begin{eqnarray}\label{positi}
&(\delta_{t}m^{n}_{h},I_{h}([m^{n}_{h}]_{-}))^{h}&\!\!\!\!+ D_{m}(\nabla m^{n}_{h},\nabla I_{h}([m^{n}_{h}]_{-}))\nonumber\\
&&\!\!\!\!+ \rho_{m} (m^{n}_{h},I_{h}([m^{n}_{h}]_{-}))^h=\mu_{m} ([u^{n-1}_{h}]_{+}v^{n-1}_{h},I_{h}([m^{n}_{h}]_{-})).
\end{eqnarray}
From the definition of the nodal interpolation operator $I_h,$ the semi-inner product $(\cdot,\cdot)^h$ (given in (\ref{mlump})), using that $(I_h(m))^2 \leq I_h(m^2)$ for all $m\in C(\bar{\Omega})$,  and taking into account that $m_{h}^{n-1}\geq 0$, one gets
\begin{equation}
(\delta_{t}m^{n}_{h},I_{h}([m^{n}_{h}]_{-}))^{h}=\frac{1}{\Delta t}\int_{\Omega}^{ }I_{h}([m^{n}_{h}]_{-}^2)dx-\frac{1}{\Delta t}\int_{\Omega}^{ }I_{h}(m^{n-1}_{h}[m^{n}_{h}]_{-})dx\geq \frac{1}{\Delta t}\|I_h([m^{n}_{h}]_{-})\|^{2}_{L^2}\label{e1}
\end{equation}
and 
\begin{eqnarray}\label{e2}
\rho_{m}(m^{n}_{h},I_h([m^n_h]_{-}))^{h}=\rho_{m} \int_{\Omega}^{ } I_{h}(([m^{n}_{h}]_{-})^2)dx\geq \rho_{m} \Vert I_h([m^{n}_{h}]_{-})\Vert_{L^2}^{2}.
\end{eqnarray}
Also, recalling that $m^{n}_{h}=I_{h}([m^{n}_{h}]_{+})+I_{h}([m^{n}_{h}]_{-}),$ and using Proposition 2.5 of \cite{GJV},  one has
\begin{eqnarray}
D_{m}(\nabla m^{n}_{h},\nabla I_{h}([m^{n}_{h}]_{-}))&\!\!\!\!=\!\!\!\!&D_{m}(\nabla I_{h}([m^{n}_{h}]_{+}),\nabla I_{h}([m^{n}_{h}]_{-}))+D_{m}(\nabla I_{h}([m^{n}_{h}]_{-}),\nabla I_{h}([m^{n}_{h}]_{-}))\nonumber\\
&\!\!\!\!\geq\!\!\!\! & D_{m}\|\nabla I_{h}([m^{n}_{h}]_{-})\|^{2}_{L^{2}}.\label{e3}
\end{eqnarray}
Then, from (\ref{positi})-(\ref{e3}), using that $v^{n-1}_h\geq 0$ one arrives at 
\begin{eqnarray*}
\left(\frac{1}{\Delta t}+\rho_{m} \right)\|I_h([m^{n}_{h}]_{-})\|^{2}_{L^2}+D_{m}\|\nabla I_{h}([m^{n}_{h}]_{-})\|^{2}_{L^{2}}\leq \mu_m\int_\Omega [u^{n-1}_{h}]_{+}v^{n-1}_{h}I_h([m^{n}_{h}]_{-}) \ dx\leq 0,
\end{eqnarray*}
which implies that
$[m^{n}_{h}]_{-}=0$, and thus $m^{n}_{h} \geq 0$. Finally, from (\ref{scheme1})$_{2}$ and taking into account that  $v_{h}^{n-1},m_{h}^{n}\geq 0,$ it holds that $$(1+\alpha\Delta t m^{n}_h)>0\ \mbox{and}\ v^{n}_h=\frac{1}{(1+\alpha \Delta tm^{n}_{h})}v^{n-1}_{h}\geq 0.$$
\end{proof}
\begin{prop}{\bf(Well-posedness)}\label{wp}
There exists a unique $[v^n_h,u^{n}_{h},m^{n}_{h},{\boldsymbol{\sigma}}^{n}_{h}] \in \mathcal{X}_{v} \times \mathcal{X}_{u} \times  \mathcal{X}_{m} \times \mathcal{X}_{{\boldsymbol{\sigma}}}$ solution of the scheme \textbf{UVM$\sigma$}.	
\end{prop}
\begin{proof}
First, in order to show that there exists a unique solution of $m^{n}_{h}\in \mathcal{X}_{m}$ of (\ref{scheme1})$_1$, it suffices to prove the uniqueness (since (\ref{scheme1})$_1$ is linear). To this aim, suppose that there exist $m^{n}_{h,1}, m^{n}_{h,2}\in \mathcal{X}_{m}$ two possible solutions; then, denoting $m^{n}_{h}=m^{n}_{h,1}-m^{n}_{h,2},$ subtracting the two equations (\ref{scheme1})$_1$ satisfied by $m^{n}_{h,1}$ and $m^{n}_{h,2},$ one gets
\begin{equation}\label{um}
(m^{n}_{h},\bar{m})^h+\Delta t D_{m}(\nabla m^{n}_{h},\nabla \bar{m})+  \rho_{m}\Delta t(m^{n}_{h},\bar{m})^h=0,\ \forall\bar{m} \in\mathcal{X}_{m}.
\end{equation}
Thus, taking $\bar{m}=m^{n}_{h}$ in (\ref{um}) and using Remark \ref{eqh2}, one has
\begin{equation*}
(1+\rho_m\Delta t)\|m^{n}_{h}\|_{L^2}^{2}+D_{m}\Delta t\|\nabla m^{n}_{h} \|_{L^{2}}^{2}=0,
\end{equation*}
which implies that $m^{n}_{h}=0,$ or equivalently, $m^{n}_{h,1}=m^{n}_{h,2}.$	Now, knowing $v^{n-1}_h$ and $m^n_h$, it is clear that there exists a unique $v^{n}_h\in \mathcal{X}_v$ solution of (\ref{scheme1})$_2.$ Finally, given $[u^{n-1}_h,{\boldsymbol{\sigma}}^{n-1}_h]$ and knowing the existence and uniqueness of $v^n_h\in \mathcal{X}_v$ and $m^n_h\in\mathcal{X}_m$, we have that there exists a unique $[u^n_h,{\boldsymbol{\sigma}}^{n}_{h}]\in \mathcal{X}_u\times  \mathcal{X}_{{\boldsymbol{\sigma}}}$ solution of (\ref{scheme1})$_{3,4}$. In fact, suppose that exist $[u^n_{h,1},{\boldsymbol{\sigma}}^{n}_{h,1}], [u^n_{h,2},{\boldsymbol{\sigma}}^{n}_{h,2}]\in \mathcal{X}_u\times \mathcal{X}_{{\boldsymbol{\sigma}}}$ two possible solutions of ({\ref{scheme1}})$_{3,4}$.  Then, denoting $u^{n}_{h}=u^{n}_{h,1}-u^{n}_{h,2},$ ${\boldsymbol{\sigma}}^{n}_{h}={\boldsymbol{\sigma}}^{n}_{h,1}-{\boldsymbol{\sigma}}^{n}_{h,2},$ subtracting (\ref{scheme1})$_{3,4}$ satisfied by $[u^n_{h,1},{\boldsymbol{\sigma}}^{n}_{h,1}]$ and  $[u^n_{h,2},{\boldsymbol{\sigma}}^{n}_{h,2}]$, one gets
\begin{equation}\label{usigma}
({\boldsymbol{\sigma}}^{n}_{h},\bar{\boldsymbol{\sigma}})+ \alpha\Delta t (m^{n}_{h}{\boldsymbol{\sigma}}^{n}_{h},\bar{\boldsymbol{\sigma}}) =0,\ \ \forall{\bar{\boldsymbol{\sigma}}}\in \mathcal{X}_{{\boldsymbol{\sigma}}},
\end{equation}
\begin{equation}\label{usigma2}
(u^{n}_{h}, \bar{u})+D_{u}\Delta t(\nabla u^{n}_{h},\nabla \bar{u}) + \mu_u \Delta t (u^n_h v^n_h,\bar{u})=0,\ \ \forall\bar{u} \in \mathcal{X}_{u}.
\end{equation}
Thus, taking $[\bar{u},{\bar{\boldsymbol{\sigma}}}]=[u^n_h,{\boldsymbol{\sigma}}^n_h]$ in (\ref{usigma})-(\ref{usigma2}) and using that $m^n_h,v^n_h\geq 0$ (see Lemma \ref{pospos}), one can conclude that $[u^n_h,{\boldsymbol{\sigma}}^{n}_{h}]=[0,{\bf 0}]$, and taking into account that (\ref{scheme1})$_{3,4}$ is an algebraic linear system, one deduce the existence and uniqueness of $[u^n_h,{\boldsymbol{\sigma}}^n_h]$ solution of (\ref{scheme1})$_{3,4}$.
\end{proof}

\subsection{Scheme \textbf{UVMs}}
In this subsection, we propose another numerical scheme which guarantees the positivity for all discrete variables, whose construction is motivated by the Definition \ref{weak2}, and where the auxiliary variable $s=u/\phi(v)$ is considered. The spatial discretization is assumed as in the scheme \textbf{UVM$\sigma$}; but in this case, instead to the FE-space $\mathcal{X}_{\boldsymbol\sigma}$, we consider the FE-space for the auxiliary variable $s$, denoted by $\mathcal{X}_s$, generated by $\mathbb{P}_1$-continuous. This last constraint is necessary to guarantee the positivity of the discrete variable $s_h$, and therefore, the positivity of $u_h$. \\

Then, we consider the following first order in time, linear and decoupled numerical scheme (from now on, Scheme \textbf{UVMs}):\\

\textbf{Initialization:} Let $[s_{h}^{0},v_{h}^{0},m_{h}^{0}]=[\mathbb{P}_{s}s_{0},\mathbb{P}_{v}v_{0},\mathbb{P}_{m}m_{0}] \in\mathcal{X}_{s} \times \mathcal{X}_{v}\times  \mathcal{X}_{m},$ being $\mathbb{P}_s:H^1(\Omega)\rightarrow \mathcal{X}_s$ an interpolation operator.\\

\textbf{Time step $n$:} Given the vector $[s^{n-1}_{h},v^{n-1}_{h},m^{n-1}_{h}] \in \mathcal{X}_{s} \times \mathcal{X}_{v}\times  \mathcal{X}_{m},$ compute $[s_{h}^{n},v_{h}^{n},m_{h}^{n}] \in \mathcal{X}_{s}  \times \mathcal{X}_{v} \times  \mathcal{X}_{m}$ such that 
\begin{flalign}\label{schemes} 
&1) \ \  (\phi(v^{n}_{h})\delta_{t}s^{n}_{h},\bar{s})^{h}+D_{u}(\phi (v^{n}_{h})\nabla s^{n}_{h},\nabla\bar{s}) =\frac{\alpha}{D_{u}}\left (s^{n-1}_{h}\phi (v^{n}_{h})\chi(v^{n}_{h})v^{n}_{h}m^{n}_{h},\bar{s}  \right ) \nonumber \\
& \ \ \hspace{3.5cm} +\mu_{u}(s^{n-1}_{h}\phi(v^{n}_{h}),\bar{s}) -\mu_u(s^{n-1}_{h}s^{n}_{h}\phi(v^{n}_{h})^{2},\bar{s})^{h}-\mu_{u}(s^{n}_{h}\phi(v^{n}_{h})v^{n}_{h},\bar{s})^{h}, \nonumber\\
&2) \ \ \ \delta_{t}v^{n}_{h} =-\alpha m^{n}_{h}v^{n}_{h},\\
&3) \ \  (\delta_{t}m^{n}_{h},\bar{m})^{h}+D_{m}(\nabla m^{n}_{h},\nabla \bar{m})+\rho_{m}(m^{n}_{h},\bar{m})^{h}=\mu_{m}(s^{n-1}_{h}\phi(v^{n-1}_{h})v^{n-1}_{h},\bar{m}),\nonumber
\end{flalign}
for all $[\bar{s},\bar{v},\bar{m}] \in \mathcal{X}_{s}  \times \mathcal{X}_{v}\times  \mathcal{X}_{m}$. Recall that, in general, we denote $\delta_t z^n_h=\frac{z^n_h-z^{n-1}_h}{\Delta t}$, and the semi-inner product $(\cdot,\cdot)^h$ was defined in (\ref{mlump}).\\

We can recover $u^n_h$ a posteriori, from the relation $u^n_h=\phi(v^n_h)s^n_h$.\\

The numerical scheme \textbf{UVMs} is well-posed and preserves the positivity in all unknowns. This is the content of next proposition.
\begin{prop}{\bf(Well-posedness and positivity of the scheme \textbf{UVMs})}
There exists a unique solution $[s^n_h,v^{n}_{h},m^{n}_{h}] \in \mathcal{X}_{s} \times \mathcal{X}_{v} \times  \mathcal{X}_{m}$ of the scheme \textbf{UVMs}. Moreover, if $s_{h}^{n-1},v^{n-1}_h,m_{h}^{n-1}\geq 0,$ then $s_{h}^{n},v^n_h,m_{h}^{n}\geq 0.$	
\end{prop}
\begin{proof}
First we prove the positivity of the possible solutions of (\ref{schemes}). Following the proof of Lemma \ref{pospos}, one has that $m^n_h\geq0 $ and $v^n_h\geq 0$, taking into account that $\mu_{m}(s^{n-1}_{h}\phi(v^{n-1}_{h})v^{n-1}_{h},I_h([m_h^n]_{-})) \leq 0$ since ($s^{n-1}_h,v^{n-1}_h\geq 0$ and $\phi(v)\geq 1$ for all $v\geq0$). Now, testing (\ref{schemes})$_1$ by $\bar{s}=I_h([s^n_h]_{-})\in \mathcal{X}_s$ one has
\begin{eqnarray}\label{positis}
 (\phi(v^{n}_{h})\delta_{t}s^{n}_{h},I_h([s^n_h]_{-}))^{h}+D_{u}(\phi (v^{n}_{h})\nabla s^{n}_{h},\nabla I_h([s^n_h]_{-})) =\frac{\alpha}{D_{u}}\left (s^{n-1}_{h}\phi (v^{n}_{h})\chi(v^{n}_{h})v^{n}_{h}m^{n}_{h},I_h([s^n_h]_{-}) \right )\nonumber\\
+\mu_{u}(s^{n-1}_{h}\phi(v^{n}_{h}),I_h([s^n_h]_{-}))  -\mu_u(s^{n-1}_{h}s^{n}_{h}\phi(v^{n}_{h})^{2},I_h([s^n_h]_{-}))^{h}-\mu_{u}(s^{n}_{h}\phi(v^{n}_{h})v^{n}_{h},I_h([s^n_h]_{-}))^{h}.\ \ \ 
\end{eqnarray}
From the definition of the nodal interpolation operator $I_h,$ the semi-inner product $(\cdot,\cdot)^h$ (given in (\ref{mlump})), using that $(I_h(s))^2 \leq I_h(s^2)$ for all $s\in C(\bar{\Omega})$,  and taking into account that $s_{h}^{n-1}\geq 0$ and $\phi(v^{n}_{h})\geq 1$ (since $v^n_h\geq0$), one obtains
\begin{eqnarray}
(\phi(v^{n}_{h})\delta_{t}s^{n}_{h},I_h([s^n_h]_{-}))^{h}&\!\!\!=\!\!\!&\frac{1}{\Delta t}\int_{\Omega}^{ }I_{h}(\phi(v^{n}_{h})[s^{n}_{h}]_{-}^2)dx-\frac{1}{\Delta t}\int_{\Omega}^{ }I_{h}(\phi(v^{n}_{h})s^{n-1}_{h}[s^{n}_{h}]_{-})dx\nonumber\\
&\!\!\!\geq\!\!\! & \frac{1}{\Delta t}\|I_h(\sqrt{\phi(v^{n}_{h})}[s^{n}_{h}]_{-})\|^{2}_{L^2}.\label{e1z}
\end{eqnarray}
Also, recalling that $s^{n}_{h}=I_{h}([s^{n}_{h}]_{+})+I_{h}([s^{n}_{h}]_{-}),$ and using Proposition 2.5 of \cite{GJV}, one gets
\begin{eqnarray*}
D_{u}(\nabla s^{n}_{h},\nabla I_{h}([s^{n}_{h}]_{-}))&\!\!\!\!=\!\!\!\!&D_{u}(\nabla I_{h}([s^{n}_{h}]_{+}),\nabla I_{h}([s^{n}_{h}]_{-}))+D_{u}(\nabla I_{h}([s^{n}_{h}]_{-}),\nabla I_{h}([s^{n}_{h}]_{-}))\nonumber\\
&\!\!\!\!\geq\!\!\!\! & D_{u}\|\nabla I_{h}([s^{n}_{h}]_{-})\|^{2}_{L^{2}}\label{e3z}
\end{eqnarray*}
and, since $\phi (v^{n}_{h})\geq 1$, one can conclude
\begin{eqnarray}
D_{u}(\phi (v^{n}_{h})\nabla s^{n}_{h},\nabla I_h([s^n_h]_{-}))\geq D_{u}\|\nabla I_{h}([s^{n}_{h}]_{-})\|^{2}_{L^{2}}.\label{e3zb}
\end{eqnarray}
On the other hand, using that $s^{n-1}_h,\phi (v^{n}_{h}),\chi(v^{n}_{h}),v^{n}_{h},m^{n}_{h}\geq 0$, one gets
\begin{eqnarray}\label{ez4}
\frac{\alpha}{D_{u}}\left (s^{n-1}_{h}\phi (v^{n}_{h})\chi(v^{n}_{h})v^{n}_{h}m^{n}_{h},I_h([s^n_h]_{-})\right)+ \mu_{u}(s^{n-1}_{h}\phi(v^{n}_{h}),I_h([s^n_h]_{-})) \leq 0, 
\end{eqnarray}
\begin{eqnarray}\label{ez5}
 -\mu_u(s^{n-1}_{h}s^{n}_{h}\phi(v^{n}_{h})^{2},I_h([s^n_h]_{-}))^{h}=-\mu_u\int_\Omega I_h (s^{n-1}_{h}\phi(v^{n}_{h})^{2}([s^n_h]_{-})^2)\leq 0,
\end{eqnarray}
\begin{eqnarray}\label{ez6}
-\mu_{u}(s^{n}_{h}\phi(v^{n}_{h})v^{n}_{h},I_h([s^n_h]_{-}))^{h}=-\mu_u\int_\Omega I_h (\phi(v^{n}_{h})v^{n}_{h}([s^n_h]_{-})^2)\leq 0.
\end{eqnarray}
Then, from (\ref{positis})-(\ref{ez6}), one arrives at 
\begin{eqnarray*}
\frac{1}{\Delta t}\|I_h(\sqrt{\phi(v^{n}_{h})}[s^{n}_{h}]_{-})\|^{2}_{L^2}+\|\nabla I_{h}([s^{n}_{h}]_{-})\|^{2}_{L^{2}}\leq 0,
\end{eqnarray*}
which implies that
$[s^{n}_{h}]_{-}=0$, and thus $s^{n}_{h} \geq 0$.

Now we prove the well-posedness. First, given $s^{n-1}_h,v^{n-1}_h, m^{n-1}_h $, the existence and uniqueness of $m^{n}_{h}\in \mathcal{X}_{m}$ solution of (\ref{schemes})$_3$ can be proved as in Proposition \ref{wp}; and, knowing $v^{n-1}_h$ and $m^n_h$, it is clear that there exists a unique $v^{n}_h\in \mathcal{X}_v$ solution of (\ref{schemes})$_2.$ Finally, given $s^{n-1}_h,m^{n-1}_h$ and knowing the existence and uniqueness of $[m^n_h,v^n_h]\in \mathcal{X}_m\times\mathcal{X}_v$, one has that there exists a unique $s^n_h\in \mathcal{X}_s$ solution of (\ref{schemes})$_{1}$. In fact, suppose that exist $s^n_{h,1},s^n_{h,2}\in \mathcal{X}_s$ two possible solutions of ({\ref{schemes}})$_{1}$.  Then, denoting $s^{n}_{h}=s^{n}_{h,1}-s^{n}_{h,2},$ subtracting (\ref{schemes})$_{1}$ satisfied by $s^n_{h,1}$ and  $s^n_{h,2}$, one gets
\begin{eqnarray}\label{ez9}
\frac{1}{\Delta t}(\phi(v^{n}_{h})s^{n}_{h},\bar{s})^{h}+D_{u}(\phi (v^{n}_{h})\nabla s^{n}_{h},\nabla\bar{s})=-\mu_u(s^{n-1}_{h}s^{n}_{h}\phi(v^{n}_{h})^{2},\bar{s})^{h}-\mu_{u}(s^{n}_{h}\phi(v^{n}_{h})v^{n}_{h},\bar{s})^{h}.
\end{eqnarray}
Taking $\bar{s}=s^{n}_{h}$ in (\ref{ez9}), recalling that $\phi(v^{n}_{h})\geq 1,$ $v^n_h,s^{n-1}_h,s^n_h\geq 0,$ and using Remark \ref{eqh2} for $\mathcal{X}_s$ instead of $\mathcal{X}_m$, it holds
\begin{eqnarray*}
\Vert s^n_h\Vert_{L^2}^2+\Delta t D_u\Vert \nabla s^{n}_{h}\Vert^2_{L^2}\leq 0,
\end{eqnarray*}
which implies that $s^n_h=0,$ that is, $s^n_{h,1}=s^n_{h,2}.$
\end{proof}

\begin{remark}{\bf (Positivity of $u^n_h$)} Notice that, taking into account that $s^n_h\geq 0$ and $\phi(v^n_h)\geq 1$ (since $v^n_h\geq 0$), one deduces that $u^n_h\geq 0$. 
\end{remark}

\section{Uniform estimates and convergence}\label{S5}
In this section, we focus on the numerical analysis of the scheme \textbf{UVM$\sigma$}, obtaining some uniform estimates for any solution of  (\ref{scheme1}) that will be used in the convergence analysis. With this aim, we make the following inductive hypothesis:  there exists a positive constant $K>0$, independent of $n$, such that
\begin{equation}\label{IndHyp}
\Vert [u^{n-1}_h,{\boldsymbol{\sigma}}^{n-1}_h]\Vert_{L^2 \times L^4}\leq K, \qquad \forall n\geq 1.
\end{equation}
After the convergence analysis we verify the validity of (\ref{IndHyp}) by following an inductive procedure. Induction hypotheses in the convergence analysis of numerical schemes approaching nonlinear PDEs have been considered by several authors (see for instance, \cite{Abel,Zhan-Zhu} and some references therein). However, it is worthwhile to remark that the inductive hypothesis (\ref{IndHyp}) includes less restrictive spaces than in the previous works mentioned. In fact, in \cite{Zhan-Zhu} the authors use an inductive hypotheses of kind $\Vert {\boldsymbol{\sigma}}^{n}_h\Vert_{W^{1,\infty}}\leq K$ to deal with a numerical scheme to approximate a 2D-Keller-Segel system; and recently, in \cite{Abel}, the authors assume an inductive hypotheses of kind $\Vert [{\boldsymbol{\sigma}}^{n-1}_h,c^{m-1}_h\Vert_{H^{1}}\leq K$ to carry out a convergence analysis of a chemotaxis-Navier-Stokes system in three dimensional domains and $\Vert {\boldsymbol{\sigma}}^{n-1}_h\Vert_{H^{1}}\leq K$ in two dimensional domains. \\

 Additionally, we will use the following discrete Gronwall lemmas:
  \begin{lemm}\label{Diego11} (\cite[p. 655]{HLe})
 	Assume that $\Delta t,\beta,B>0$ and $b^k, d^k\ge 0$ satisfy:
 	\begin{equation*}\label{e-Diego11-1}
 	(1+\beta \Delta t)d^{k+1} - d^k + \Delta t b^{k+1}\leq B \Delta t,  \quad \forall k \ge 0.
 	\end{equation*}
 	Then, it holds
 	\begin{equation*}\label{e-Diego11-2}
 	d^{k} + \Delta t \, \sum_{i=1}^{k} (1+\beta \Delta t)^{-(k+1-i)} b^{i} \le d^0 + \beta^{-1}B, \quad \forall k \ge 1.
 	\end{equation*}
 \end{lemm}
 
 \begin{lemm}\label{Diego2} (\cite[p. 369]{Hey})
 Assume that $\Delta t>0$ and $B,b^k, d^k, g^k, h^k\ge 0$ satisfy:
 	\begin{equation*}\label{e-Diego2-1}
 	d^{k+1}+ \Delta t\sum_{i=0}^{k} b^{i+1}  \le \Delta t \sum_{i=0}^{k} g^i \, d^i +\Delta t\sum_{i=0}^{k} h^i + B,  \quad \forall k \ge 0.
 	\end{equation*}
 	Then, it holds
 	\begin{equation*}\label{e-Diego2-2}
 	d^{k+1} + \Delta t \, \sum_{i=0}^{k} b^{i+1} \le \exp \left( \Delta t \, \sum_{i=0}^{k} g^i\right)
 	\, \left( \Delta t \, \sum_{i=0}^{k} h^i + B \right), \quad \forall k \ge 0.
 	\end{equation*}
 \end{lemm}
 
\subsection{Uniform estimates}
In order to develop the convergence analysis, some uniform estimates (in weak and strong norms) for the discrete variables $v^n_h$ and $m^n_h$ are needed; these are natural estimates coming from the same analysis as in the continuous problem.
\begin{lemm}{\bf (Uniform estimate for $v_h^n$)}\label{uev} 
If $v^n_h$ is any solution of (\ref{scheme1})$_2$, then 
\begin{equation*}
\Vert v^n_h\Vert_{L^\infty}\leq K_0 \qquad \forall n\geq 0. 
\end{equation*}
\end{lemm}
\begin{proof}
From (\ref{scheme1})$_2$, taking into account that $m^n_h,v^n_h\geq 0$ (see Lemma \ref{pospos}), one has that $
v^n_h -v^{n-1}_h \leq0,$ and thus, adding from $n=1$ to $n=r$, one arrives at
$$
v^r_h\leq v^0_h \qquad \forall r\geq 1,
$$
from which one deduces that $\sup_{\Omega} v^n_h\leq \sup_{\Omega} v^0_h:=K_0$ for all $n\geq 0$.
\end{proof}
\begin{lemm}{\bf (Uniform weak estimates for $m_h^n$)}\label{uem1}
Assume the inductive hypothesis (\ref{IndHyp}). If $m^n_h$ is any solution of (\ref{scheme1})$_1$, then $m^n_h$ is bounded in $l^\infty(L^2) \cap l^2(H^1)$.
\end{lemm}
\begin{proof}
Testing (\ref{scheme1})$_1$ by $\bar{m}= m^n_h$, using the H\"older and Young inequalities, taking into account Remark \ref{eqh2}, Lemma \ref{uev} and using the fact that $\Vert z_+\Vert_{L^2}\leq \Vert z \Vert_{L^2}$ and the inductive hypothesis (\ref{IndHyp}), one has 
\begin{eqnarray}\label{nh1}
&\displaystyle\frac{1}{2} \delta_t&\!\!\!\!\!\! \Vert m^n_h\Vert^2_{L^2} + \frac{\Delta t}{2} \Vert \delta_t m^n_h\Vert_{L^2}^2 + D_m \Vert \nabla m^n_h\Vert^2_{L^2}+\rho_m \Vert m^n_h\Vert^2_{L^2}\leq  \mu_m\Vert [u^{n-1}_h]_+\Vert_{L^2}\Vert v^{n-1}_h\Vert_{L^\infty}\Vert m^n_h\Vert_{L^2}\nonumber\\
&&\!\!\! \leq \frac{\rho_m}{4} \Vert m^n_h\Vert^2_{L^2} + \frac{C\mu_m^2}{\rho_m} \Vert u^{n-1}_h\Vert_{L^2}^2 \Vert v^{n-1}_h\Vert_{L^\infty}^2\leq \frac{\rho_m}{4} \Vert m^n_h\Vert^2_{L^2} + C.
\end{eqnarray} 
 Then, multiplying (\ref{nh1}) by $2\Delta t$ and denoting $\lambda_m=\min\{D_m,\rho_m\}$, one arrives at
 \begin{equation*}
\Big(1+\frac{\Delta t}{2} \Big) \Vert m^n_h\Vert^2_{L^2} - \Vert m^{n-1}_h\Vert^2_{L^2} + \Delta t \lambda_m \Vert m^n_h\Vert^2_{H^1}\leq \Delta t C,
 \end{equation*}
and applying Lemma \ref{Diego11} the proof is concluded.
\end{proof}
Now, in next lemma, some uniform strong estimates are proved for $m^n_h$, which will also be necessary in the convergence analysis.
\begin{lemm}{\bf (Uniform strong estimates for $m_h^n$)}\label{uem2}
	Assume the inductive hypothesis (\ref{IndHyp}).  If $m^n_h$ is any solution of (\ref{scheme1})$_1$, then the following estimate holds
	\begin{equation}\label{semm}
	\Vert m^n_h\Vert_{H^1}^2 + \Delta t \sum_{k=1}^{n}\Vert \delta_t  m^k_h\Vert^2_{L^2}\leq C, \qquad \forall n\geq 1,
	\end{equation}
with the constant $C>0$ depending on the data $(\mu_m, D_m, \rho_m, m_0,v_0,T,K)$, but independent of $(\Delta t,h)$ and $n$.
\end{lemm}
\begin{proof}
Testing (\ref{scheme1})$_1$ by $\bar{m}=\delta_t m^n_h \in \mathcal{X}_m$, and proceeding as in (\ref{nh1}) one obtains
\begin{eqnarray}\label{nh1m2}
&\displaystyle\frac{\lambda_m}{2}  \delta_t&\!\!\!\!\!\! \Vert m^n_h\Vert^2_{H^1} + \frac{\Delta t \lambda_m}{2} \Vert \delta_t m^n_h\Vert_{H^1}^2 +  \Vert \delta_t m^n_h\Vert^2_{L^2}\leq \frac{1}{2} \Vert \delta_t m^n_h\Vert^2_{L^2} + C.
\end{eqnarray} 
Then, multiplying (\ref{nh1m2}) by $\Delta t$ and adding from $n=1$ to $n=r$, (\ref{semm}) is concluded.
\end{proof}

%%%%%%%%%%%%%%%%%%%%%%%%%%%%%%%%%%%%%%%%%%%%%%%%%%%NEW11111111%%%%%%%%%%%%%%%%%
\subsection{Error estimates in weak norms}\label{ESWN}
In this section we derive error estimates for any solution $[m^n_h,v^n_h,u^n_h,{\boldsymbol\sigma}^n_h]$ of the scheme \textbf{UVM$\sigma$}, with respect to a smooth enough solution $[m, v,u,{\boldsymbol\sigma}]$ of (\ref{split}). We will denote by   $C,C_i,K_i$ to different positive constants possibly  depending on the continuous solution $(m,v,u,{\boldsymbol \sigma}=\nabla v)$, but independent of the discrete parameters $(k,h)$ and the time step $n$.

 We start by establishing the following notation for the errors at time $t=t_{n}$: $e_m^n=m^n-m^n_h$, $e_v^n=v^n-v^n_h$, $e_u^n=u^n-u^n_h$  and $e_{\boldsymbol\sigma}^n={\boldsymbol\sigma}^n-{\boldsymbol\sigma}^n_h,$ where, in general, $z^{n}$ denote the value of $z$ at time $t_{n}$. Taking the difference between scheme (\ref{scheme1}) and (\ref{split}) at $t=t_{n}$ we get that $[e_m^n,e^n_v,e^n_u,e_{{\boldsymbol \sigma}}^n]$ satisfies:
\begin{eqnarray}
&&(\delta_t e^n_m,\bar{m}) + D_m(\nabla e_m^n, \nabla \bar{m})+\rho_m(e_m^n,\bar{m}) =(\omega_m^n,\bar{m}) + (\delta_t  m^n_h + \rho_m m^n_h,\bar{m})^h - (\delta_t  m^n_h + \rho_m m^n_h,\bar{m})\nonumber\\
&&\hspace{0.5cm}  + \mu_m((u^{n}\!-u^{n-1})v^n + u^{n-1}(v^n - v^{n-1}) + ([u^{n-1}]_{+} - [u^{n-1}_h]_{+})v^{n-1}_h + u^{n-1} e^{n-1}_v, \bar{m}),\label{errm}\\
\nonumber\\
&&\delta_t e_v^n= - \alpha m^n e_v^n - \alpha v^n_h e^n_m + \omega^n_v, \label{errv}\\
\nonumber\\
&&(\delta_t e_u^n, \bar{u}) + D_u(\nabla e_u^n, \nabla \bar{u})=(\omega_u^n,\bar{u})+(\chi(v^n)[(u^n-u^{n-1}){\boldsymbol{\sigma}}^n + u^{n-1}({\boldsymbol{\sigma}}^n- {\boldsymbol{\sigma}}^{n-1})],\nabla \bar{u})\nonumber\\
&&\ \ +((\chi(v^n)-\chi(v^n_h))u^{n-1}{\boldsymbol{\sigma}}^{n-1} + \chi(v^n_h) u^{n-1}e^{n-1}_{\boldsymbol{\sigma}}+ \chi(v^n_h) e^{n-1}_u {\boldsymbol{\sigma}}^{n-1}_h,\nabla \bar{u}) +\mu_u(u^n-u^{n-1}+e^{n-1}_u,\bar{u})\nonumber\\
&&\ \ - \mu_u((u^n)^2-(u^{n-1})^2+e^{n-1}_u(u^{n-1}+u^{n-1}_h),\bar{u}) - \mu_u(e^{n}_u v^{n}_h+u^{n}e^n_v,\bar{u}),\label{erru}\\
\nonumber\\
&&(\delta_t e_{\boldsymbol\sigma}^n,\bar{\boldsymbol\sigma}) = (\omega^n_{\boldsymbol{\sigma}},\bar{\boldsymbol\sigma})-\alpha (m^n_h e_{\boldsymbol\sigma}^n,\bar{\boldsymbol\sigma}) - \alpha ({\boldsymbol\sigma}^n e^n_m,\bar{\boldsymbol\sigma}) -\alpha (v^n_h \nabla e_m^n,\bar{\boldsymbol\sigma}) - \alpha (e^n_v \nabla m^n,\bar{\boldsymbol\sigma}),\label{errs}
\end{eqnarray}
for all $[\bar{m},\bar{u},\bar{\boldsymbol \sigma}]\in \mathcal{X}_m\times \mathcal{X}_u\times \mathcal{X}_{\boldsymbol \sigma}$, where $\omega_m^n,\omega_v^n,\omega^n_u,\omega^n_{\boldsymbol{\sigma}}$ are the consistency errors associated to the scheme \textbf{UVM$\sigma$}, that is, $\omega_m^n=\delta_t m^n - (m_t)^n$ and so on. \\

With the help of the interpolation operators $\mathbb{P}_m, \mathbb{P}_{v}, \mathbb{P}_{u}, \mathbb{P}_{\boldsymbol{\sigma}},$ defined in Subsection \ref{INOP1}, we decompose the total errors $e_m^n, e_v^n, e_u^n, e_{\boldsymbol{\sigma}}^n,$ as the sum of interpolation and discrete errors as follows:
\begin{eqnarray}\label{u1a}
e_m^n &=&(m^n -\mathbb{P}_m m^n) + ( \mathbb{P}_m m^n - m^n_h )
= \theta^n_m+\xi^n_m,\\
e_v^n&=&(v^n -\mathbb{P}_v v^n) + ( \mathbb{P}_v v^n - v^n_h )=\theta^n_v+\xi^n_v,\label{u1av}\\
e_u^n&=&(u^n -\mathbb{P}_{u} u^n) + ( \mathbb{P}_{u} u^n - u^n_h )=\theta^n_{u}+\xi^n_{u},\label{u1au}\\
e_{\boldsymbol{\sigma}}^n&=&({\boldsymbol{\sigma}}^n -\mathbb{P}_{\boldsymbol{\sigma}} {\boldsymbol{\sigma}}^n) + ( \mathbb{P}_{\boldsymbol{\sigma}} {\boldsymbol{\sigma}}^n - {\boldsymbol{\sigma}}^n_h )=\theta^n_{\boldsymbol{\sigma}}+\xi^n_{\boldsymbol{\sigma}},\label{u1asig}
\end{eqnarray}
where, in general, $\theta^m_{z}$ and $\xi^m_{z}$ denote the interpolation and discrete errors (for the variable $z$), respectively. Then, taking into account (\ref{errm})-(\ref{errs}), (\ref{u1a})-(\ref{u1asig}) and the definition of the interpolation operators given in Subsection \ref{INOP1}, one gets
\begin{eqnarray}\label{errn-int}
&(\delta_t &\!\!\!\!\! \xi_m^n,\bar{m}) + D_m(\nabla \xi_m^n, \nabla \bar{m})+\rho_m (\xi_m^n,\bar{m}) =(\omega_m^n- \delta_t \theta_m^n,\bar{m})+ (\delta_t  m^n_h,\bar{m})^h - (\delta_t  m^n_h,\bar{m}) \nonumber\\
&&\!\!\!\!\!\!\!\!\!\! + \rho_m(m^n_h,\bar{m})^h - \rho_m(m^n_h,\bar{m}) -D_m(\nabla \theta_m^n, \nabla \bar{m})-\rho_m (\theta_m^n,\bar{m}) + \mu_m((u^{n}\!-u^{n-1})v^n,\bar{m}) \nonumber\\
&&\!\!\!\!\!\!\!\!\!\!+\mu_m(u^{n-1}(v^n - v^{n-1}) \!+\! ([u^{n-1}]_{+}\! -\! [u^{n-1}_h]_{+})v^{n-1}_h\! + u^{n-1} (\xi^{n-1}_v\!+ \theta^{n-1}_v), \bar{m}),
\end{eqnarray}
\begin{equation}\label{errv-int}
\delta_t \xi^n_v = \omega_v^n- \delta_t \theta^n_v-\alpha m^n(\xi^n_v + \theta^n_v) -\alpha v^n_h(\xi^n_m + \theta^n_m), 
\end{equation}
\begin{eqnarray}\label{erru-int}
&(\delta_t \xi_u^n,&\!\!\!\!\bar{u}) + D_u(\nabla \xi_u^n, \nabla \bar{u})=(\omega_u^n - \delta_t \theta_u^n,\bar{u}) + D_u(\theta_u^n, \bar{u})+(\chi(v^n)(u^n-u^{n-1}){\boldsymbol{\sigma}}^n ,\nabla \bar{u})\nonumber\\
&&\!\!\!\!\!\!\!\!\!\!\! +(\chi(v^n)u^{n-1}({\boldsymbol{\sigma}}^n- {\boldsymbol{\sigma}}^{n-1}) +(\chi(v^n)-\chi(v^n_h))u^{n-1}{\boldsymbol{\sigma}}^{n-1} + \chi(v^n_h) u^{n-1}(\xi^{n-1}_{\boldsymbol{\sigma}}+\theta^{n-1}_{\boldsymbol{\sigma}}),\nabla \bar{u})\nonumber\\
&&\!\!\!\!\!\!\!\!\!\!\! + (\chi(v^n_h) (\xi^{n-1}_u+\theta^{n-1}_u) {\boldsymbol{\sigma}}^{n-1}_h,\nabla \bar{u}))+\mu_u(u^n-u^{n-1}+\xi^{n-1}_u+\theta^{n-1}_u,\bar{u})\nonumber\\
&&\!\!\!\!\!\!\!\!\!\!\!  - \mu_u((u^n)^2-(u^{n-1})^2+(\xi^{n-1}_u+\theta^{n-1}_u)(u^{n-1}+u^{n-1}_h) + (\xi^{n}_u+\theta^{n}_u) v^{n}_h+u^{n}(\xi^n_v+\theta^{n}_v),\bar{u}),
\end{eqnarray}
\begin{eqnarray}\label{errs-int}
&(\delta_t \xi_{\boldsymbol\sigma}^n&\!\!\!\!\!,\bar{\boldsymbol \sigma})  = (\omega_{\boldsymbol\sigma}^n,\bar{\boldsymbol \sigma})-(\delta_t \theta_{\boldsymbol\sigma}^n,\bar{\boldsymbol \sigma}) -\alpha(m^{n}_h (\xi^{n}_{\boldsymbol{\sigma}}+\theta^{n}_{\boldsymbol{\sigma}}),\bar{\boldsymbol\sigma})
\nonumber\\
&&\!\!\!\!\!\! -\alpha( {\boldsymbol{\sigma}}^n (\xi^{n}_{m}+\theta^{n}_{m})+ v^n_h (\nabla \xi_m^n + \nabla \theta^n_m) + (\xi^n_v+\theta^n_v) \nabla m^n,\bar{\boldsymbol\sigma}).
\end{eqnarray}
\vspace{0.1 cm}

{\underline{{\it 1. Error estimate for} $m$}}\\

Taking $\bar{m}=\xi_m^n$ in (\ref{errn-int}) one gets 
\begin{eqnarray}\label{errn-int2}
&\displaystyle\frac{1}{2}&\!\!\!\!\!\delta_t  \Vert \xi_m^n\Vert_{L^2}^2 + \frac{\Delta t}{2} \Vert \delta_t \xi_m^n\Vert_{L^2}^2 + \lambda_m \Vert\xi_m^n\Vert_{H^1}^2 \leq (\omega_m^n,\xi_m^n)- \left(\delta_t \theta_m^n,\xi_m^n\right) + (\delta_t  m^n_h,\bar{m})^h - (\delta_t  m^n_h,\xi_m^n)  \nonumber\\
&&\!\!\!\!\!\! + \rho_m(m^n_h,\xi_m^n)^h - \rho_m(m^n_h,\xi_m^n) -D_m(\nabla \theta_m^n, \nabla \xi_m^n)-\rho_m (\theta_m^n,\xi_m^n) + \mu_m((u^{n}\!-u^{n-1})v^n, \xi_m^n)    \nonumber\\
&&\!\!\!\!\!\!+ \mu_m(u^{n-1}(v^n - v^{n-1}) + ([u^{n-1}]_{+} - [u^{n-1}_h]_{+})v^{n-1}_h + u^{n-1} (\xi^{n-1}_v+ \theta^{n-1}_v), \xi_m^n)  =\sum_{k=1}^{10} I_k, 
\end{eqnarray}
(recall that  $\lambda_m=\min\{D_m,\rho_m\}$). Then, using the H\"older and Young inequalities, (\ref{aprox01}) and (\ref{aprox01nn})$_{1-2}$, the terms on the right hand side of (\ref{errn-int2}) are bounded in the following way:
\begin{equation}\label{Ea1a}
I_1\leq \displaystyle\frac{\lambda_m}{10} \Vert \xi_m^n\Vert_{H^1}^2 + \frac{C}{\lambda_m}\Vert \omega_m^n\Vert_{(H^1)'}^2\leq \displaystyle\frac{\lambda_m}{10}\Vert\xi_m^n\Vert_{H^1}^2+\frac{C\Delta t}{\lambda_m} \int_{t_{n-1}}^{t_n}\Vert m_{tt}(t)\Vert_{(H^1)'}^2 dt,
\end{equation}
\begin{eqnarray}\label{Ea1b}
&I_2&\!\!\!\leq  \Vert \xi_m^n\Vert_{L^2} \Vert(\mathcal{I} - \mathbb{P}_m) \delta_t m^n \Vert_{L^2} \leq \displaystyle\frac{\lambda_m}{10} \Vert \xi_m^n\Vert_{H^1}^2+\frac{C h^{2(r_1+1)}}{\lambda_m}\Vert \delta_t m^n\Vert_{H^{r_1+1}}^2\nonumber\\
&&\!\!\!\leq \displaystyle\frac{\lambda_m}{10} \Vert \xi_m^n\Vert_{H^1}^2+\displaystyle\frac{C h^{2(r_1+1)}}{\lambda_m \Delta t }\int_{t_{n-1}}^{t_n}\Vert m_t \Vert_{H^{r_1+1}}^2  dt,
\end{eqnarray}
\begin{eqnarray}\label{Ea1d}
& I_7+I_8&\!\!\! \leq D_m\Vert \nabla \theta^n_m\Vert_{L^2}  \Vert \nabla \xi^n_m\Vert_{L^2}   + \rho_m\Vert \theta^n_m\Vert_{L^2}   \Vert\xi^n_m\Vert_{L^2}   \nonumber\\
&&\!\!\! \leq \displaystyle\frac{\lambda_m}{10} \Vert \xi_m^n\Vert_{H^1}^2 +\frac{C}{\lambda_m}(D_m^2 h^{2r_1} +\rho_m^2  h^{2(r_1+1)})  \|m^{n}\|_{ H^{r_1+1}}^2,
\end{eqnarray}
\begin{eqnarray}\label{Ea1e}
& I_9&\!\!\!\!\!+I_{10} \leq \mu_m\Vert [u^n\! -\! u^{n-1}\!, v^n \!- \!v^{n-1}]\Vert_{(H^1)'}  \Vert[v^{n},u^{n-1}]\Vert_{L^\infty} \Vert \xi_m^n\Vert_{H^1} \nonumber\\
&&+\mu_m\Vert [\xi^{n-1}_u, \theta^{n-1}_u,\xi^{n-1}_v, \theta^{n-1}_v]\Vert_{L^2}  \Vert[u^{n-1},v^{n-1}_h]\Vert_{L^\infty} \Vert \xi_m^n\Vert_{L^2}  \nonumber\\
&&\!\!\! \leq \displaystyle\frac{\lambda_m}{10} \Vert \xi_m^n\Vert_{H^1}^2 +\frac{C\mu_m^2}{\lambda_m} (\Vert [u^n\! -\! u^{n-1}\!, v^n \!- \!v^{n-1}]\Vert_{(H^1)'}^2 +\Vert \xi^{n-1}_u\Vert_{L^2}^2+ \Vert \xi^{n-1}_v\Vert_{L^2}^2)  \Vert[v^{n},u^{n-1},v^{n-1}_h]\Vert_{L^\infty}^2\nonumber\\
&&+\displaystyle\frac{C\mu_m^2}{\lambda_m}(h^{2(r_2+1)} \|v^{n-1}\|_{ H^{r_2+1}}^2 + h^{2(r_3+1)} \|u^{n-1}\|_{ H^{r_3+1}}^2) \Vert[v^{n},u^{n-1},v^{n-1}_h]\Vert_{L^\infty}^2.
\end{eqnarray}
Moreover, taking into account the property (\ref{MassL}), one gets
\begin{eqnarray}\label{Ea1c}
&I_3+I_4+I_5+I_6&\!\!\!\!\! \leq C h^{r_1} \Vert \delta_t m^n_h\Vert_{L^2} \Vert \nabla \xi^n_m\Vert_{L^2} + C\rho_m h^{r_1} \Vert  m^n_h\Vert_{L^2} \Vert \nabla \xi^n_m\Vert_{L^2} \nonumber\\
&&\!\!\!\!\!\leq \frac{\lambda_m}{10} \Vert \xi^n_m\Vert_{H^1}^2 + \frac{C}{\lambda_m} h^{2r_1} ( \Vert \delta_t m^n_h\Vert_{L^2}^2 + \rho_m^2 \Vert  m^n_h\Vert_{L^2}^2).
\end{eqnarray}
Therefore, from (\ref{errn-int2})-(\ref{Ea1c}), one arrives at
\begin{eqnarray}\label{errnfin}
&\displaystyle\frac{1}{2}&\!\!\!\!\!\delta_t  \Vert \xi_m^n\Vert_{L^2}^2 + \frac{\Delta t}{2} \Vert \delta_t \xi_m^n\Vert_{L^2}^2 + \frac{\lambda_m}{2} \Vert\xi_m^n\Vert_{H^1}^2  \leq   C\int_{t_{n-1}}^{t_n}\left(\frac{h^{2(r_1+1)}}{\Delta t }\Vert m_t \Vert_{H^{r_1+1}}^2+ \Delta t \Vert m_{tt}(t)\Vert_{(H^1)'}^2\right)  dt \nonumber\\
&&\!\!\!\!\!\!\!\!\! +C (\Vert u^n\! -\! u^{n-1}\Vert_{(H^1)'}^2 + \Vert v^n \!- \!v^{n-1}\Vert_{(H^1)'}^2 +\Vert \xi^{n-1}_u\Vert_{L^2}^2+ \Vert \xi^{n-1}_v\Vert_{L^2}^2)  \Vert[v^{n},u^{n-1},v^{n-1}_h]\Vert_{L^\infty}^2\nonumber\\
&&\!\!\!\!\!\!\!\!\!+C(h^{2(r_2+1)} \|v^{n-1}\|_{ H^{r_2+1}}^2 + h^{2(r_3+1)} \|u^{n-1}\|_{ H^{r_3+1}}^2) \Vert[v^{n},u^{n-1},v^{n-1}_h]\Vert_{L^\infty}^2\nonumber\\
&&\!\!\!\!\!\!\!\!\!+ C ( h^{2r_1} +  h^{2(r_1+1)})  \|m^{n}\|_{ H^{r_1+1}}^2+ Ch^{2r_1} ( \Vert \delta_t m^n_h\Vert_{L^2}^2 +  \Vert  m^n_h\Vert_{L^2}^2).
\end{eqnarray}
\vspace{0.1 cm}

{\underline{{\it 2. Error estimate for} $v$}}\\

Testing (\ref{errv-int}) by $\xi^n_v\in \mathcal{X}_v$, using the H\"older and Young inequalities and (\ref{aprox01nn})$_{1-2}$, one has
\begin{eqnarray}\label{errc-int2}
&\displaystyle\frac{1}{2}&\!\!\!\!\!\delta_t  \Vert \xi_v^n\Vert_{L^2}^2 + \frac{\Delta t}{2} \Vert \delta_t \xi_v^n\Vert_{L^2}^2 + \alpha \int_{\Omega} m^n (\xi^n_v)^2 \ dx  =(\omega_v^n- \delta_t \theta^n_v-\alpha m^n \theta^n_v -\alpha v^n_h(\xi^n_m + \theta^n_m), \xi_v^n)\nonumber\\
&&\!\!\!\!\! \leq (\Vert \omega_v^n\Vert_{L^2} + \Vert \delta_t \theta^n_v\Vert_{L^2} +\alpha \Vert [m^n,v^n_h]\Vert_{L^\infty} (\Vert \theta_v^n\Vert_{L^2} + \Vert \xi_m^n\Vert_{L^2} + \Vert \theta_m^n\Vert_{L^2})) \Vert \xi_v^{n}\Vert_{L^2}\nonumber\\
&&\!\!\!\!\! \leq C \Vert \xi_v^{n}\Vert_{L^2}^2 + C (\Vert \omega_v^n\Vert_{L^2}^2 + \Vert \delta_t \theta^n_v\Vert_{L^2}^2 +\alpha^2 \Vert [m^n,v^n_h]\Vert_{L^\infty}^2 ( \Vert \theta_v^n\Vert_{L^2}^2 + \Vert \xi_m^n\Vert_{L^2}^2 + \Vert \theta_m^n\Vert_{L^2}^2))\nonumber\\
&&\!\!\!\!\! \leq C \Vert \xi_v^{n}\Vert_{L^2}^2 + C\int_{t_{n-1}}^{t_n}\left(\frac{h^{2(r_2+1)}}{\Delta t }\Vert v_t \Vert_{H^{r_2+1}}^2+ \Delta t \Vert v_{tt}(t)\Vert_{L^2}^2\right)  dt +C  \Vert [m^n,v^n_h]\Vert_{L^\infty}^2  \Vert \xi_m^n\Vert_{L^2}^2\nonumber\\
&& +C  \Vert [m^n,v^n_h]\Vert_{L^\infty}^2 ( h^{2(r_2+1)}\|v^{n}\|_{ H^{r_2+1}}^2  + h^{2(r_1+1)}\|m^{n}\|_{ H^{r_1+1}}^2).
\end{eqnarray}

\vspace{0.5 cm}
{\underline{{\it 3. Error estimate for} $u$}}\\

Taking $\bar{u}=\xi_{u}^n$ in (\ref{erru-int}), one gets
\begin{eqnarray}\label{erru-int2}
&\displaystyle\frac{1}{2}&\!\!\!\!\!\delta_t  \Vert \xi_u^n\Vert_{L^2}^2+\displaystyle\frac{\Delta t}{2}\Vert\delta_t  \xi_u^n\Vert_{L^2}^2 + D_{u}\Vert\nabla \xi_u^n\Vert_{L^2}^2 + \mu_u\int_{\Omega} v^n_h (\xi_u^n)^2 \ dx= (\omega_u^n - \delta_t \theta_u^n + D_u\theta_u^n, \xi_u^n)\nonumber\\
&&\!\!\!\!\!+(\chi(v^n)(u^n-u^{n-1}){\boldsymbol{\sigma}}^n +\chi(v^n)u^{n-1}({\boldsymbol{\sigma}}^n- {\boldsymbol{\sigma}}^{n-1}) +(\chi(v^n)-\chi(v^n_h))u^{n-1}{\boldsymbol{\sigma}}^{n-1} ,\nabla \xi_u^n)\nonumber\\
&&\!\!\!\!\! + (\chi(v^n_h) u^{n-1}(\xi^{n-1}_{\boldsymbol{\sigma}}+\theta^{n-1}_{\boldsymbol{\sigma}}),\nabla \xi_u^n)+ (\chi(v^n_h) (\xi^{n-1}_u+\theta^{n-1}_u) {\boldsymbol{\sigma}}^{n-1}_h,\nabla \xi_u^n)\nonumber\\
&&\!\!\!\!\!+\mu_u(u^n-u^{n-1}+\xi^{n-1}_u+\theta^{n-1}_u,\xi_u^n)  - \mu_u((u^n)^2-(u^{n-1})^2+(\xi^{n-1}_u+\theta^{n-1}_u)(u^{n-1}+u^{n-1}_h),\xi_u^n)\nonumber\\
&&\!\!\!\!\! - \mu_u(\theta^{n}_u v^{n}_h+u^{n}(\xi^n_v+\theta^{n}_v),\xi_u^n) =\sum_{k=1}^{7} J_k.
\end{eqnarray}
Then, using the H\"older and Young inequalities, (\ref{aprox01}) and (\ref{aprox01nn})$_{2-3}$,  the terms on the right hand side of (\ref{erru-int2}) are bounded as follows:
\begin{eqnarray}\label{Ea1au}
&J_1&\!\!\!\leq C\Vert \xi^n_u\Vert_{H^1} \Vert \omega^n_u\Vert_{(H^1)'} + C\Vert \xi^n_u\Vert_{L^2} (D_u\Vert \theta^n_u\Vert_{L^2}+\Vert \delta_t \theta^n_u\Vert_{L^2}) \nonumber\\
&&\!\!\!  \leq \displaystyle\frac{D_u}{8} \Vert \nabla \xi_{u}^n\Vert_{L^2}^2 + CD_u \Vert \xi_{u}^n\Vert_{L^2}^2 + \frac{C}{D_u} \int_{t_{n-1}}^{t_n}\!\! \left[\Delta t \Vert u_{tt}(t)\Vert_{(H^1)'}^2 + \frac{h^{2(r_3+1)}}{\Delta t }\Vert u_t \Vert_{H^{r_3+1}}^2\right]dt\nonumber\\
&& + CD_uh^{2(r_3+1)} \|u^{n}\|_{ H^{r_3+1}}^2,
\end{eqnarray}
\begin{eqnarray}\label{Ea1bu}
&J_2+J_3&\!\!\!\leq \displaystyle\frac{D_u}{8} \Vert \nabla \xi_{u}^n\Vert_{L^2}^2+ \frac{C}{D_u} \Vert[u^n - u^{n-1},{\boldsymbol{\sigma}}^n- {\boldsymbol{\sigma}}^{n-1}] \Vert_{L^2}^2\Vert[{\boldsymbol{\sigma}}^n,u^{n-1}] \Vert_{L^\infty}^2 \Vert\chi(v^n)\Vert_{L^\infty}^2\nonumber\\
&&\hspace{-1.3 cm}+  \frac{C}{D_u} ( \Vert\xi^{n}_v\Vert_{L^2}^2 + \Vert\xi^{n-1}_{\boldsymbol{\sigma}} \Vert_{L^2}^2 + h^{2(r_2+1)} \|v^{n}\|_{ H^{r_2+1}}^2 + h^{2(r_4+1)} \|{\boldsymbol{\sigma}}^{n-1}\|_{ H^{r_4+1}}^2)\Vert u^{n-1} \Vert_{L^\infty}^2\Vert[\chi(v^n_h),{\boldsymbol{\sigma}}^{n-1}] \Vert_{L^\infty}^2,
\end{eqnarray}
\begin{eqnarray}\label{Ea1du}
&J_5+J_7&\!\!\!  \leq C \Vert \xi_{u}^n\Vert_{L^2}^2 + C\mu_u^2 ( \Vert u^n \!- \!u^{n-1}\Vert_{L^2}^2+h^{2(r_3+1)}\! \|u^{n-1}\|_{H^{r_3+1}}^2+ \Vert \xi^{n-1}_{u}\Vert_{L^2}^2)\nonumber\\
&&\hspace{-0.6 cm} + C\mu_u^2  \Vert [v^n_h,u^n]\Vert_{L^\infty}^2 (h^{2(r_3+1)}\! \|u^{n}\|_{H^{r_3+1}}^2 + \Vert \xi^{n}_{v}\Vert_{L^2}^2+h^{2(r_2+1)}\! \|v^{n}\|_{H^{r_2+1}}^2).
\end{eqnarray}
Moreover, using the 3D interpolation inequalities
\begin{equation*}
\Vert u\Vert_{L^4} \leq  C\Vert u\Vert_{L^2}^{1/4}\Vert u\Vert_{H^1}^{3/4} \ \ \mbox{and} \ \
\Vert u\Vert_{L^3}\leq \Vert u\Vert_{L^2}^{1/2}\Vert u\Vert_{L^6}^{1/2} \  \ \mbox{for all} \ \ u\in H^1(\Omega),
\end{equation*}
as well as the H\"older and Young inequalities, (\ref{aprox01}), (\ref{aprox01-aNN-new})$_2$ and (\ref{aprox01-a}), one has
\begin{eqnarray}\label{Ea1cu}
& J_4&\!\!\! =  (\chi(v^n_h)\xi^{n-1}_u {\boldsymbol{\sigma}}^{n-1}_h,\nabla \xi_u^n) - (\chi(v^n_h)\theta^{n-1}_u \xi_{\boldsymbol{\sigma}}^{n-1},\nabla \xi_u^n)+ (\chi(v^n_h)\theta^{n-1}_u \mathbb{P}_{\boldsymbol{\sigma}}{\boldsymbol{\sigma}}^{n-1},\nabla \xi_u^n) \nonumber\\
&&\!\!\!\! \leq C(\Vert \xi^{n-1}_{u}\Vert_{L^2}^{1/4}\Vert \xi^{n-1}_{u}\Vert_{H^1}^{3/4} \Vert{\boldsymbol{\sigma}}^{n-1}_h\Vert_{L^4}+ \Vert \xi_{\boldsymbol\sigma}^{n-1}\Vert_{L^2}   \Vert \theta_u^{n-1}\Vert_{L^\infty} ) \Vert \nabla\xi_u^n\Vert_{L^2}\Vert \chi(v^n_h)\Vert_{L^\infty}  \nonumber\\
&&+C \Vert \mathbb{P}_{\boldsymbol{\sigma}} \boldsymbol{\sigma}^{n-1}\Vert_{L^\infty}  \Vert \theta_u^{n-1}\Vert_{L^2} \Vert \nabla\xi_u^n\Vert_{L^2}\Vert \chi(v^n_h)\Vert_{L^\infty}  \nonumber\\
&&\!\!\!\! \leq \displaystyle\frac{D_u}{8} \Vert \nabla \xi_{u}^n\Vert_{L^2}^2 +  \displaystyle\frac{D_u}{8} \Vert \nabla \xi^{n-1}_{u}\Vert_{L^2}^{2}+C D_u \Vert  \xi^{n-1}_{u}\Vert_{L^2}^{2}+ \frac{C}{D_u^7}  \Vert \xi^{n-1}_{u}\Vert_{L^2}^{2} \Vert\chi(v^n_h)\Vert_{L^\infty}^8 \Vert{\boldsymbol{\sigma}}^{n-1}_h\Vert_{L^4}^8 \nonumber\\
&&+\frac{C}{D_u}\Vert \chi(v^n_h)\Vert_{L^\infty}^2(\Vert u^{n-1}\Vert_{H^2}^2 \Vert \xi_{\boldsymbol{\sigma}}^{n-1}\Vert_{L^2}^2    +  h^{2(r_3 +1)} \Vert {\boldsymbol{\sigma}}^{n-1}\Vert_{H^2}^2   \Vert u^{n-1}\Vert_{H^{r_3 + 1}}^2),
\end{eqnarray}
\begin{eqnarray}\label{Ea1eu}
&J_6&\!\!\! \leq C \Vert \xi_{u}^n\Vert_{L^2}^2+ C\mu_u^2(\Vert u^n \!- \!u^{n-1}\!\Vert_{L^2}^2 + \Vert \xi^{n-1}_{u}\Vert_{L^2}^2 + h^{2(r_3+1)}\! \|u^{n-1}\|_{H^{r_3+1}}^2) \Vert [u^n + u^{n-1},u^{n-1}]\Vert_{L^\infty}^2 \nonumber\\
&& +C\mu_u^2\Vert \theta^{n-1}_u\Vert_{L^\infty}^2 \Vert \xi^{n-1}_u\Vert_{L^2}^2 + C\mu_u^2 h^{2(r_3 +1)} \Vert \mathbb{P}_u u^{n-1}\Vert_{L^\infty}^2 \Vert u^{n-1}\Vert_{H^{r_3 + 1}}^2 +\displaystyle\frac{D_u}{8} \Vert \nabla \xi_{u}^n\Vert_{L^2}^2\nonumber\\
&&+C D_u\Vert \xi_{u}^n\Vert_{L^2}^2 +  \displaystyle\frac{D_u}{8} \Vert \nabla \xi^{n-1}_{u}\Vert_{L^2}^{2}+C D_u \Vert  \xi^{n-1}_{u}\Vert_{L^2}^{2}+\frac{C\mu_u^4}{D_u^3} \Vert  \xi^{n-1}_{u}\Vert_{L^2}^{2} \Vert  u^{n-1}_{h}\Vert_{L^2}^{4}.
\end{eqnarray}
Therefore, from (\ref{erru-int2})-(\ref{Ea1eu}), using the inductive hypothesis (\ref{IndHyp}) and taking into account that $v^n_h\geq 0$, one arrives at
\begin{eqnarray}\label{errufin}
&\displaystyle\frac{1}{2}&\!\!\!\!\!\delta_t  \Vert \xi_u^n\Vert_{L^2}^2+\displaystyle\frac{\Delta t}{2}\Vert\delta_t  \xi_u^n\Vert_{L^2}^2 + \frac{D_{u}}{2}\Vert\nabla \xi_u^n\Vert_{L^2}^2 -  \frac{D_{u}}{4}\Vert\nabla \xi_u^{n-1}\Vert_{L^2}^2 \leq C \Vert \xi_{u}^n\Vert_{L^2}^2+ C \Vert \xi_{u}^{n-1}\Vert_{L^2}^2  \nonumber\\
&&+ C h^{2(r_3+1)} \|u^{n}\|_{ H^{r_3+1}}^2+ C\int_{t_{n-1}}^{t_n}\!\! \left[\Delta t \Vert u_{tt}(t)\Vert_{(H^1)'}^2 + \frac{h^{2(r_3+1)}}{\Delta t }\Vert u_t \Vert_{H^{r_3+1}}^2\right]dt \nonumber\\
&&+ C \Vert[u^n - u^{n-1},{\boldsymbol{\sigma}}^n- {\boldsymbol{\sigma}}^{n-1}] \Vert_{L^2}^2\Vert[{\boldsymbol{\sigma}}^n,u^{n-1}] \Vert_{L^\infty}^2 \Vert\chi(v^n)\Vert_{L^\infty}^2\nonumber\\
&&+  C ( \Vert\xi^{n}_v\Vert_{L^2}^2 + \Vert\xi^{n-1}_{\boldsymbol{\sigma}} \Vert_{L^2}^2 + h^{2(r_2+1)} \|v^{n}\|_{ H^{r_2+1}}^2 + h^{2(r_4+1)} \|{\boldsymbol{\sigma}}^{n-1}\|_{ H^{r_4+1}}^2)\Vert u^{n-1} \Vert_{L^\infty}^2\Vert[\chi(v^n_h),{\boldsymbol{\sigma}}^{n-1}] \Vert_{L^\infty}^2 \nonumber\\
&&+ C \Vert \xi^{n-1}_{u}\Vert_{L^2}^{2} \Vert\chi(v^n_h)\Vert_{L^\infty}^8 +C\Vert \chi(v^n_h)\Vert_{L^\infty}^2(\Vert u^{n-1}\Vert_{H^2}^2 \Vert \xi_{\boldsymbol{\sigma}}^{n-1}\Vert_{L^2}^2    +  h^{2(r_3 +1)} \Vert {\boldsymbol{\sigma}}^{n-1}\Vert_{H^2}^2   \Vert u^{n-1}\Vert_{H^{r_3 + 1}}^2)\nonumber\\
&&+C ( \Vert u^n \!- \!u^{n-1}\Vert_{L^2}^2+h^{2(r_3+1)}\! \|u^{n-1}\|_{H^{r_3+1}}^2+ \Vert \xi^{n-1}_{u}\Vert_{L^2}^2)\nonumber\\
&& + C  \Vert [v^n_h,u^n]\Vert_{L^\infty}^2 (h^{2(r_3+1)}\! \|u^{n}\|_{H^{r_3+1}}^2+ \Vert \xi^{n}_{v}\Vert_{L^2}^2+h^{2(r_2+1)}\! \|v^{n}\|_{H^{r_2+1}}^2)\nonumber\\
&&+ C(\Vert u^n \!- \!u^{n-1}\!\Vert_{L^2}^2 + \Vert \xi^{n-1}_{u}\Vert_{L^2}^2 +h^{2(r_3+1)}\! \|u^{n-1}\|_{H^{r_3+1}}^2) \Vert [u^n + u^{n-1},u^{n-1}]\Vert_{L^\infty}^2 \nonumber\\
&& +C\Vert \theta^{n-1}_u\Vert_{L^\infty}^2 \Vert \xi^{n-1}_u\Vert_{L^2}^2 + C h^{2(r_3 +1)} \Vert \mathbb{P}_u u^{n-1}\Vert_{L^\infty}^2 \Vert u^{n-1}\Vert_{H^{r_3 + 1}}^2 .
\end{eqnarray}
\vspace{0.5 cm}

{\underline{{\it 4. Error estimate for } $\boldsymbol{\sigma}$}}\\

Taking $\bar{\boldsymbol{\sigma}}=\xi_{\boldsymbol\sigma}^n$ in (\ref{errs-int}), using the H\"older and Young inequalities and (\ref{aprox01nn})$_{1-3}$, one arrives at
\begin{eqnarray}\label{errs-int2}
&\displaystyle\frac{1}{2}&\!\!\!\!\!\delta_t \Vert\xi_{\boldsymbol\sigma}^n\Vert_{L^2}^2+\displaystyle\frac{\Delta t}{2}\Vert\delta_t \xi_{\boldsymbol\sigma}^n\Vert_{L^2}^2 +\alpha\int_\Omega  m^{n}_h (\xi^{n}_{\boldsymbol{\sigma}})^2 \ dx= (\omega_{\boldsymbol\sigma}^n- \delta_t \theta_{\boldsymbol\sigma}^n,\xi_{\boldsymbol\sigma}^n) 
\nonumber\\
&&-\alpha(m^{n}_h \theta^{n}_{\boldsymbol{\sigma}},\xi_{\boldsymbol\sigma}^n)-\alpha( {\boldsymbol{\sigma}}^n (\xi^{n}_{m}+\theta^{n}_{m})+ v^n_h (\nabla \xi_m^n + \nabla \theta^n_m) + (\xi^n_v+\theta^n_v) \nabla m^n,\xi_{\boldsymbol\sigma}^n)\nonumber\\
&&\!\!\!\!\!\! \leq C \Vert \xi_{\boldsymbol{\sigma}}^{n}\Vert_{L^2}^2 + C\!\int_{t_{n-1}}^{t_n}\!\!\!\left(\frac{h^{2(r_4+1)}}{\Delta t }\Vert {\boldsymbol{\sigma}}_t \Vert_{H^{r_4+1}}^2+ \Delta t \Vert {\boldsymbol{\sigma}}_{tt}(t)\Vert_{L^2}^2\right)  dt  + C\alpha^2 \Vert \theta^n_{\boldsymbol{\sigma}}\Vert_{L^\infty}^2 \Vert \xi^n_m\Vert_{L^2}^2  \nonumber\\
&& +  C\alpha^2 h^{2(r_4+1)} \Vert\mathbb{P}_{m} m^n \Vert_{L^\infty}^2 \|\boldsymbol{\sigma}^{n}\|_{ H^{r_4+1}}^2+C \alpha^2 \Vert\boldsymbol{\sigma}^n\Vert_{L^\infty}^2 (\Vert \xi_{m}^{n}\Vert_{L^2}^2  + h^{2(r_1+1)}\|m^{n}\|_{ H^{r_1+1}}^2) + \frac{\lambda_m}{4} \Vert \xi^n_m\Vert_{H^1}^2\nonumber\\
&&  +C \alpha^2 \Vert[ v^{n}_h,\nabla m^n]\Vert_{L^\infty}^2 \Big(\frac{1}{\lambda_m}\Vert \xi_{\boldsymbol{\sigma}}^{n}\Vert_{L^2}^2  + h^{2r_1}   \|m^{n}\|_{ H^{r_1+1}}^2  +  \Vert \xi_{v}^{n}\Vert_{L^2}^2 + h^{2(r_2+1)}\|v^{n}\|_{ H^{r_2+1}}^2\Big). 
\end{eqnarray}
\vspace{0.5 cm}

{\underline{\it 5. Estimate for the terms $\Vert u^n - u^{n-1}\Vert_{(H^1)'}$, $\Vert v^n - v^{n-1}\Vert_{(H^1)'}$, $\Vert u^n - u^{n-1}\Vert_{L^2}$ and $\Vert{\boldsymbol{\sigma}^n} - \boldsymbol{\sigma}^{n-1}]\Vert_{L^2}$}}\\

Observe that the following estimates hold
\begin{equation}\label{mm1}
\displaystyle\Delta t \sum_{n=1}^r \Vert [u^n - u^{n-1},v^n-v^{n-1}]\Vert_{(H^1)'}^2\leq C(\Delta t)^4 \Vert [u_{tt},v_{tt}]\Vert^2_{L^2((H^1)')} + C(\Delta t)^2 \Vert [u_{t},v_{t}]\Vert^2_{L^2((H^1)')},
\end{equation}
\begin{equation}\label{mm1-2}
\displaystyle\Delta t \sum_{n=1}^r \Vert [u^n - u^{n-1},{\boldsymbol{\sigma}^n} - \boldsymbol{\sigma}^{n-1}]\Vert_{L^2}^2\leq C(\Delta t)^4 \Vert [u_{tt},{\boldsymbol{\sigma}}_{tt}]\Vert^2_{L^2(L^2)} + C(\Delta t)^2 \Vert [u_{t},{\boldsymbol{\sigma}}_{t}]\Vert^2_{L^2(L^2)}.
\end{equation}
Indeed, 
$$
\Vert \omega_u^n\Vert_{(H^1)'}=\Vert \delta_t u^n - (u_t)^n\Vert_{(H^1)'}=\Big\Vert \frac{1}{\Delta t} (u^n- u^{n-1}) - (u_t)^n\Big\Vert_{(H^1)'}\leq C(\Delta t)^{1/2} \Big(\int_{t_{n-1}}^{t_n}\Vert u_{tt}(t)\Vert_{(H^1)'}^2 dt\Big)^{1/2}, 
$$
where the last inequality was obtained as in (\ref{Ea1au}). Therefore, one can deduce 
$$
\Delta t \sum_{n=1}^r \Vert u^n - u^{n-1}\Vert_{(H^1)'}^2\leq C(\Delta t)^4 \Vert u_{tt}\Vert^2_{L^2((H^1)')} + C(\Delta t)^2 \Vert u_t\Vert^2_{L^2((H^1)')}.
$$
Analogously, we obtain the estimate for $v$ given in (\ref{mm1}) and the estimate for $[u,{\boldsymbol{\sigma}}]$ in the $L^2$-norm  given in (\ref{mm1-2}).

\vspace{0.7 cm}
Then, we can prove the following result:
 \begin{theo}\label{theo1N}
Assume (\ref{IndHyp}). Let  $[m^n_h,v^n_h,u^n_h,\boldsymbol{\sigma}^{n}_h]$ be any solution of the scheme \textbf{UVM$\sigma$} and consider a su\-ffi\-cien\-tly regular solution $[m, c, u, {\boldsymbol\sigma}]$ of (\ref{split}). There exists a constant $C$ (depending on the data of the problem (\ref{split})) such that  If $\Delta t C<\frac{1}{2}$, the following estimate for the discrete errors holds
\begin{equation}\label{EEtheo1}
\|[\xi^n_m,\xi^n_v,\xi^n_{u},\xi^n_{\boldsymbol{\sigma}}]\|_{l^{\infty}(L^2)} + \|[\xi^n_m,\xi^n_{u}]\|_{l^{2}(H^1)} \leq C(T) \Big(\Delta t +\max\{h^{r_1},h^{r_2+1},h^{r_3+1},h^{r_4+1}\}\Big).
\end{equation}
\end{theo}
\begin{proof}
Adding (\ref{errnfin}), (\ref{errc-int2}), (\ref{errufin}) and (\ref{errs-int2}), multiplying the resulting expression by $\Delta t$, adding from $k=1$ to $k=n$, taking into account that $m^n, m^n_h\geq 0$, using Lemmas \ref{uev}-\ref{uem2}, estimates (\ref{mm1})-(\ref{mm1-2}) and the regularity for the exact solution given in Theroem \ref{TRG}, and recalling that $[\xi^0_m,\xi^0_v,\xi^0_{u},\xi^0_{\boldsymbol{\sigma}}]=[0,0, 0,{\bf 0}]$, one has
\begin{eqnarray}\label{EEW1}
&\Vert [\xi^n_m, \xi^n_v, \xi^n_u,\xi^n_{\boldsymbol{\sigma}}]\Vert_{L^2}^2&\!\!\!\! + \Delta t  \sum_{k=1}^n \Big(\frac{D_{u}}{2}\Vert\nabla \xi_u^k\Vert_{L^2}^2 + \frac{\lambda_m}{2}  \Vert\xi_m^k\Vert_{H^1}^2\Big)\leq C_1((\Delta t)^2+(\Delta t)^4) \nonumber\\
&&\hspace{-2.5 cm} +C_2(h^{2(r_1+1)}+ h^{2(r_2+1)} + h^{2(r_3+1)}+ h^{2(r_4+1)}+ h^{2r_1}) + C_3\Delta t \sum_{k=1}^n \Vert [\xi^{k-1}_m, \xi^{k-1}_v, \xi^{k-1}_u,\xi^{k-1}_{\boldsymbol{\sigma}}]\Vert_{L^2}^2\nonumber\\
&& \hspace{-2.5 cm} +C_4 \Delta t \Vert [\xi^n_m, \xi^n_v, \xi^n_u,\xi^n_{\boldsymbol{\sigma}}]\Vert_{L^2}^2.
\end{eqnarray}
Therefore, if $\Delta t$ is small enough such that $\frac{1}{2}-C_4\Delta t>0$, by appying Lemma \ref{e-Diego2-1} to (\ref{EEW1}), (\ref{EEtheo1}) is concluded.
\end{proof}
As a consequence of Theorem \ref{theo1N}, the following results hold:
\begin{coro}
Under hypotheses of Theorem \ref{theo1N}, the following estimates for the total errors hold:
\begin{equation*}
\|[e^n_m,e^n_v,e^n_{u},e^n_{\boldsymbol{\sigma}}]\|_{l^{\infty}(L^2)}  \leq C(T) \Big(\Delta t +\max\{h^{r_1},h^{r_2+1},h^{r_3+1},h^{r_4+1}\}\Big),
\end{equation*}
\begin{equation*}
\|[e^n_m,e^n_{u}]\|_{l^{2}(H^1)}  \leq C(T) \Big(\Delta t +\max\{h^{r_1},h^{r_2+1},h^{r_3},h^{r_4+1}\}\Big).
\end{equation*}
\end{coro}
\begin{coro}
Under hypotheses of Theorem \ref{theo1N}. Then, $[v^n_h,\boldsymbol{\sigma}^{n}_h]$ converges to $[v,{\boldsymbol\sigma}]$ in 
	$L^\infty(L^2)$-norm and  $[m^n_h,u^n_h]$ converges to $[m, u]$ in 
	$L^\infty(L^2),L^2(H^1)$-norms, when the parameters $\Delta t$ and $h$ go to $0$.
\end{coro}

Finally, it is clear that the error estimates were derived under the inductive hypothesis (\ref{IndHyp}). Now we have to check it. 
We derive (\ref{IndHyp}) by using (\ref{EEtheo1}) recursively.  Observe that 
$$\Vert [u^{n-1},{\boldsymbol{\sigma}}^{n-1}]\Vert_{L^2\times L^4}\leq \Vert [u,{\boldsymbol{\sigma}}]\Vert_{L^\infty(H^1 \times H^1)}:=C_0 \qquad \forall n\geq 1,$$
and therefore, using the stability properties (\ref{aprox01-aNN-new})$_1$ and (\ref{aprox01-aNN}), one has \begin{equation}\label{INNa}
\Vert [u^{0}_h,{\boldsymbol{\sigma}}^0_h]\Vert_{L^2\times L^4}= \Vert [\mathbb{P}_{u} u_0,\mathbb{P}_{\boldsymbol{\sigma}} {\boldsymbol{\sigma}}_0] \Vert_{L^2\times L^4}\leq C_0\leq C_0+1:=K
\end{equation}
and
\begin{equation*}
\Vert [u^{n-1}_h,{\boldsymbol{\sigma}}^{n-1}_h]\Vert_{L^2\times L^4}\leq \Vert [\xi^{n-1}_{u},\xi_{\boldsymbol{\sigma}}^{n-1}]\Vert_{L^2\times L^4}+\Vert [\mathbb{P}_{u}u^{n-1},\mathbb{P}_{\boldsymbol{\sigma}} {\boldsymbol{\sigma}}^{n-1}]\Vert_{L^2\times L^4}\leq \Vert [\xi^{n-1}_{u},\xi_{\boldsymbol{\sigma}}^{n-1}]\Vert_{L^2\times L^4} + C_0.
\end{equation*}
Then, it is enough to show that $\Vert [\xi^{n-1}_{u},\xi_{\boldsymbol{\sigma}}^{n-1}]\Vert_{L^2\times L^4}\leq 1,$ for each $n\geq 2.$ Notice that from (\ref{EEtheo1}) and using (\ref{INNa}), one has
\begin{eqnarray}\label{DD1}
&\Vert \xi^1_{u}\Vert_{L^2}&\!\!\!\!\leq C(T,\Vert u^0_h\Vert_{L^2},\Vert {\boldsymbol{\sigma}}^0_h\Vert_{L^4})(\Delta t+\max\{h^{r_1},h^{r_2+1},h^{r_3+1},h^{r_4+1}\})\nonumber\\
&&\!\!\!\! \leq C(T,K)(\Delta t+\max\{h^{r_1},h^{r_2+1},h^{r_3+1},h^{r_4+1}\})
\end{eqnarray}
and
\begin{eqnarray}\label{j1}
&\Vert \xi^1_{\boldsymbol{\sigma}}\Vert_{L^4}&\!\!\!\!\!\leq \frac{1}{h^p}\Vert \xi^1_{\boldsymbol{\sigma}}\Vert_{L^2}\leq C(T,\Vert u^0_h\Vert_{L^2},\Vert {\boldsymbol{\sigma}}^0_h\Vert_{L^4})\frac{1}{h^p}(\Delta t+\max\{h^{r_1},h^{r_2+1},h^{r_3+1},h^{r_4+1}\})\nonumber\\
&&\!\!\!\!\! \leq C(T,K)\Big(\frac{(\Delta t)^p}{h^p}(\Delta t)^{1-p}+\max\{h^{r_1-p},h^{r_2+1-p},h^{r_3+1-p},h^{r_4+1-p}\}\Big),
\end{eqnarray}
where in (\ref{j1}) the inverse inequality  $\Vert \xi^n_{\boldsymbol{\sigma}}\Vert_{L^4}\leq h^{-p} \Vert \xi^n_{\boldsymbol{\sigma}}\Vert_{L^2}$ (with $p=1/2$ in 2D and $p=3/4$ in 3D) was used. Therefore, taking $\Delta t$ and $h$ small enough with $\Delta t\leq h$, from (\ref{DD1})-(\ref{j1}) one can conclude that $\Vert  [\xi^{1}_{u},\xi^{1}_{\boldsymbol{\sigma}}]\Vert_{L^2\times L^4}\leq 1$, which implies $\Vert [{u}^{1}_h,{\boldsymbol{\sigma}}^{1}_h]\Vert_{L^2\times L^4}\leq K$.  Analogously, by using $\Vert [{u}^{1}_h,{\boldsymbol{\sigma}}^{1}_h]\Vert_{L^2\times L^4}\leq K$, one can obtain $\Vert  [\xi^{2}_{u},\xi^{2}_{\boldsymbol{\sigma}}]\Vert_{L^2\times L^4}\leq 1$, and therefore, $\Vert [{u}^{2}_h,{\boldsymbol{\sigma}}^{2}_h]\Vert_{L^2\times L^4}\leq K$. Arguing recursively we conclude that $\Vert [{u}^{n-1}_h,{\boldsymbol{\sigma}}^{n-1}_h]\Vert_{L^2\times L^4}\leq K$, for all $n\geq 1$.

\section{Numerical simulations} 
In this section, we present some numerical experiments in order to verify the good behavior of the Schemes \textbf{UVM$\sigma$} and \textbf{UVMs}. All simulations were computed by using the software Freefem++. We have considered the discrete spaces $\mathcal{X}_m, \mathcal{X}_v,\mathcal{X}_u,\mathcal{X}_{\boldsymbol\sigma},\mathcal{X}_s$ approximated by $\mathbb{P}_1-$continuous FE, the rectangular domain $\Omega=[0,1]\times[0,1]$ and an unstructured mesh.\\

 The aim of these experiments is to see the spatio-temporal evolution of the invasion of the extracellular matrix by the cancer cells, considering two different types of extracellular matrix (homogeneous and heterogeneous), comparing the behavior when there is absence and presence of cell proliferation. These experiments are motivated by the two dimensional numerical simulations presented in \cite{Anderson}, which can be compared with experimental and clinical observations. For this reason, we have considered the values for the parameters used in \cite{Anderson}, that is, $D_m=0.001$, $\rho_m=0$, $\mu_m=0.1$, $\alpha=10$, $D_u=0.001$ and $\chi=0.005$ in (\ref{KNS}). Moreover, the discrete parameters are taken $\Delta t=10^{-2}$ and $h=1/50$; and the simulations results are showed for the times $t = 1,5,10,15$.\\
 
 \underline{Test 1. Homogeneous extracellular matriz:} The aim of this experiment is to show the behavior of the schemes \textbf{UVM$\sigma$} and \textbf{UVMs} in the context of a homogeneous extracellular matriz (see Figure \ref{fig:RBM1}(b)). In order to simulate the absence and presence of cell proliferation , we consider $\mu_u=0$ and $\mu_u=2$ respectively; and we consider the following initial conditions (see Figure \ref{fig:RBM1}):
$$
u_0= \text{exp} (-400(x-0.5)^2-400(y-0.5)^2),
$$
$$
m_0 =0.5u_0 \ \ \mbox{ and } \ \ v_0 = 1-u_0.
$$

\begin{minipage}{\textwidth}
	\begin{flushleft}
		\begin{tabular}{ccc}

		 \includegraphics[width=43mm]{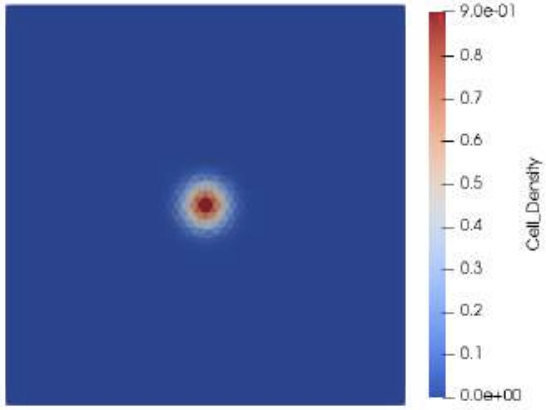} & \hspace{-2mm} \includegraphics[width=43mm]{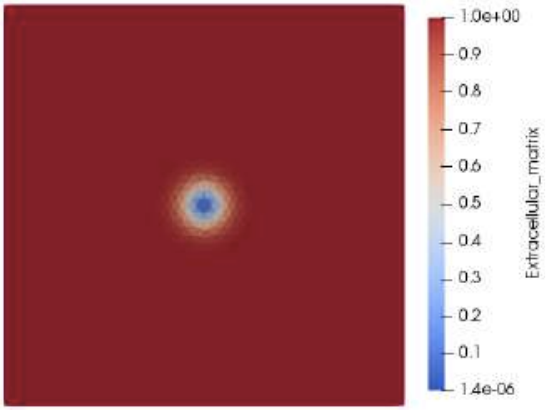} & \hspace{-2mm} \includegraphics[width=43mm]{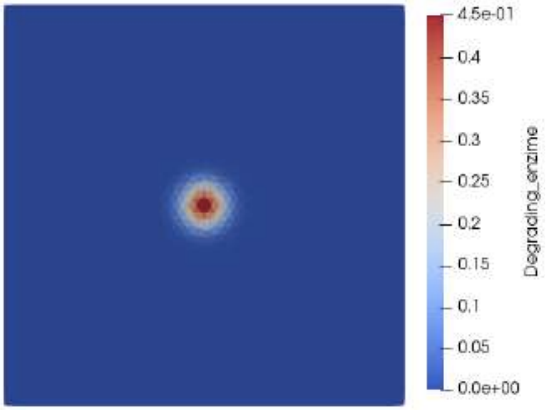}		  \\[2mm]
		  (a) {\bf Cell density} &  (b) {\bf Extracellular matrix} & (c)  {\bf Degrading enzyme}
		\end{tabular}
	\end{flushleft}
	\vspace{-0.5 cm}
\figcaption{Initial conditions in Test 1} \label{fig:RBM1}
\end{minipage}\\

The evolution results for the case  $\mu_u=0$ are showed in Figures \ref{fig:RBM2} and \ref{fig:RBM3} for the schemes \textbf{UVM$\sigma$} and \textbf{UVMs}, respectively. The behavior of the cell density reproduces the pattern observed in \cite{Anderson}. The ring of cells that makes up the tumor body at the beginning invades the extracellular matrix, while a correlated increase of the degrading enzyme occurs. In this case, we can see some of the main characteristics of the invasion of a solid
tumor in its avascular phase: diffusion, random motility, movement along the gradient of the density of adhesive components of
extracellular matrix (haptotaxis) and extracellular matrix degradation. The numerical simulations for both schemes show a very similar behavior; with the difference that the cell density computed with the scheme \textbf{UVM$\sigma$} takes negative values (very small) in some times, while in the scheme \textbf{UVMs} the cell density is always positive (see Figures \ref{fig:RBM2}, \ref{fig:RBM3} and \ref{fig:MM}). This fact is in agreement with the theoretical positivity results obtained in Subsection \ref{SSPW}.
 
\begin{minipage}{\textwidth}
	\begin{flushleft}
		\begin{tabular}{cccc}
		 {\bf Time} & {\bf Cell density} &  {\bf Extracellular matrix} & {\bf Degrading enzyme}  \\[2mm]
		 $t=1$ &		\includegraphics[width=43mm]{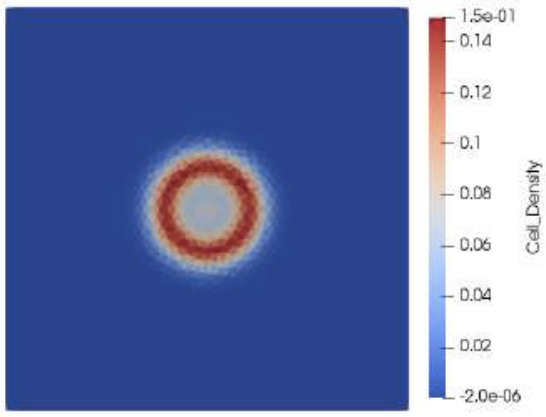} & \includegraphics[width=43mm]{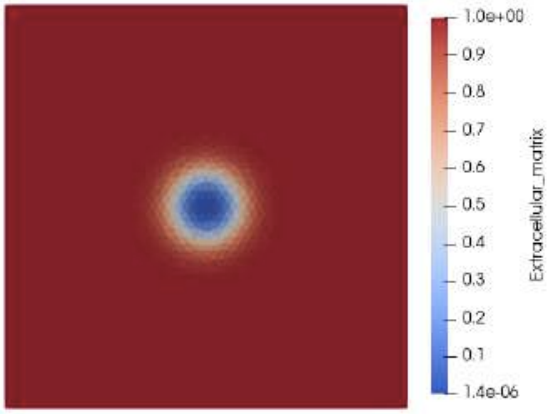} & \includegraphics[width=43mm]{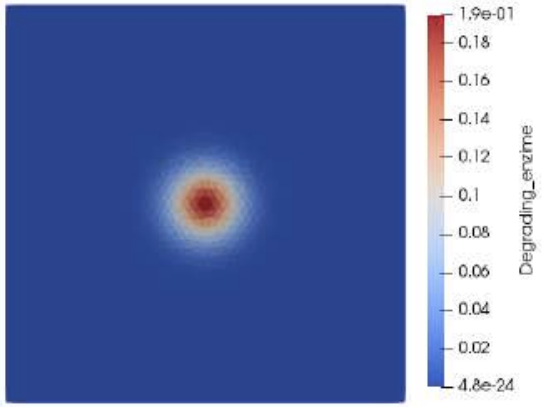}  \\[3mm]
		$t=5$ & \includegraphics[width=43mm]{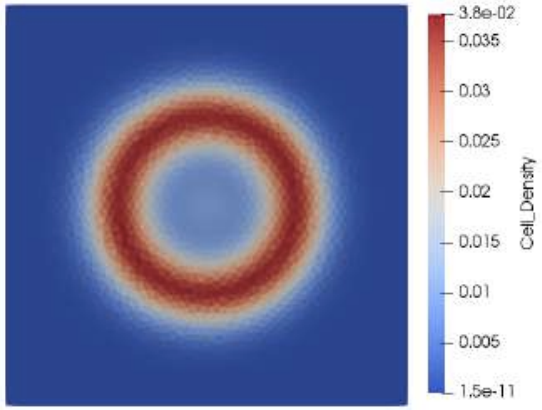} & \includegraphics[width=43mm]{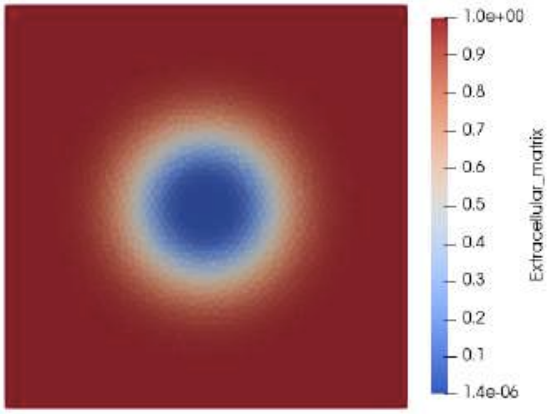} &\includegraphics[width=43mm]{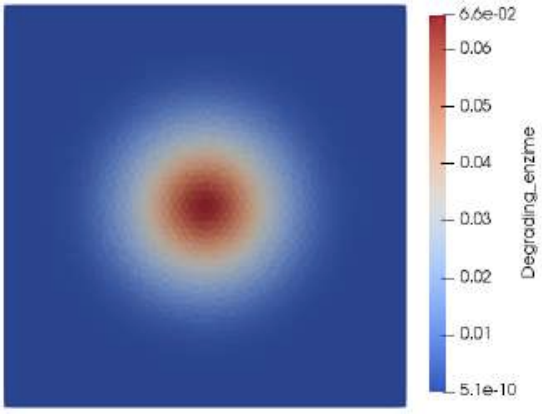} \\[3mm]
		$t=10$ & \includegraphics[width=43mm]{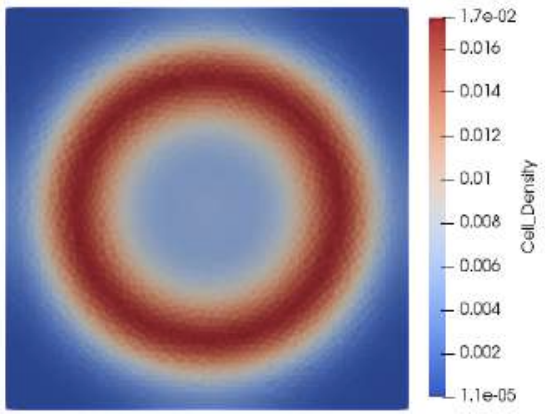} & \includegraphics[width=43mm]{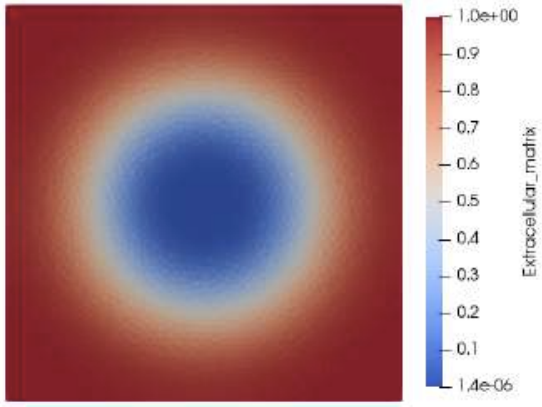} &\includegraphics[width=43mm]{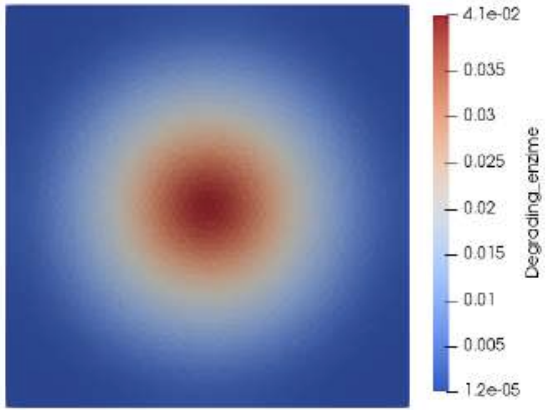} \\[3mm]
		$t=15$ & \includegraphics[width=43mm]{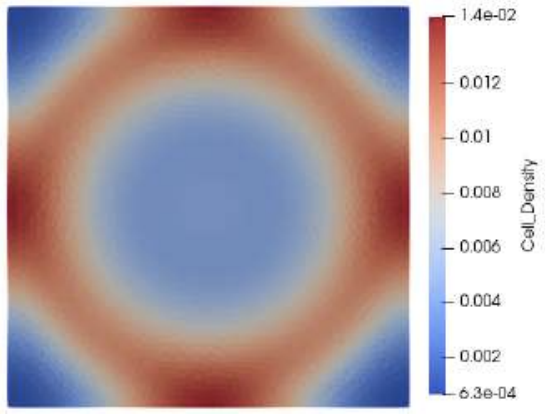} & \includegraphics[width=43mm]{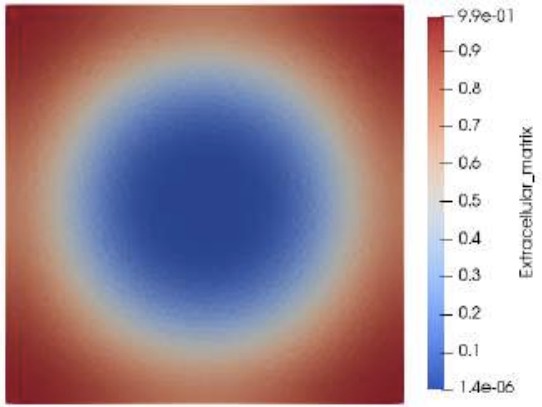} &\includegraphics[width=43mm]{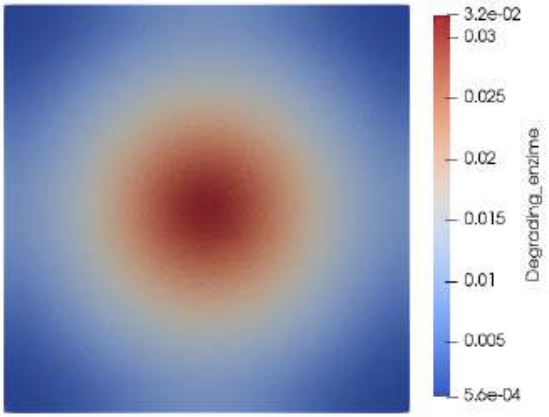}
		\end{tabular}
	\end{flushleft}
	\vspace{-0.5 cm}
\figcaption{Behavior of the scheme \textbf{UVM$\sigma$} in Test 1 for $\mu_u=0$.} \label{fig:RBM2}
\end{minipage}
\\
\\

\begin{minipage}{\textwidth}
	\begin{flushleft}
		\begin{tabular}{cccc}
		 {\bf Time} & {\bf Cell density} &  {\bf Extracellular matrix} & {\bf Degrading enzyme}  \\[2mm]
		 $t=1$ &		\includegraphics[width=43mm]{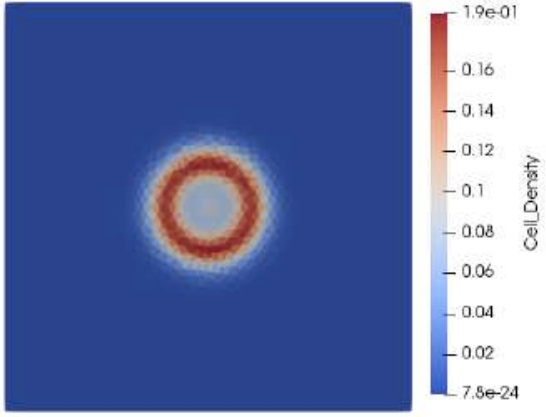} & \includegraphics[width=43mm]{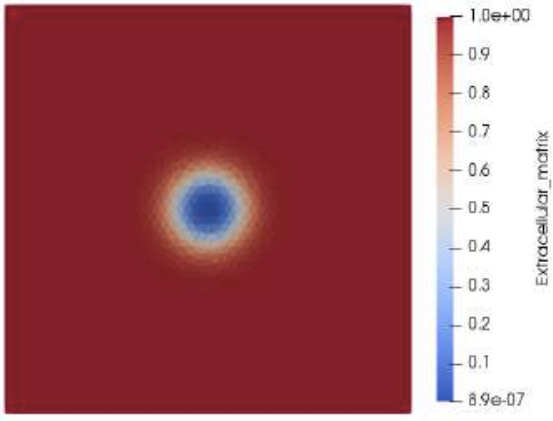} & \includegraphics[width=43mm]{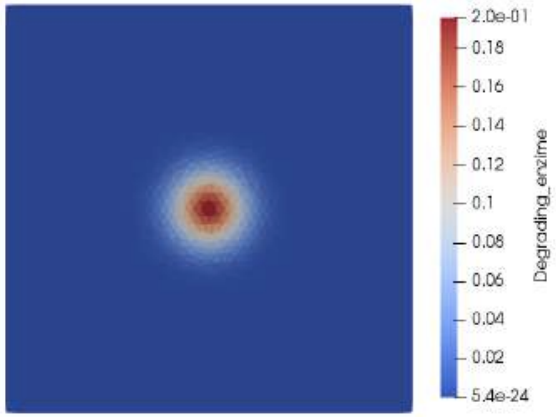}  \\[3mm]
		$t=5$ & \includegraphics[width=43mm]{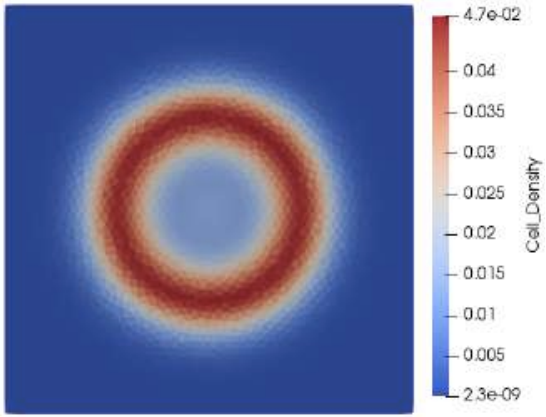} & \includegraphics[width=43mm]{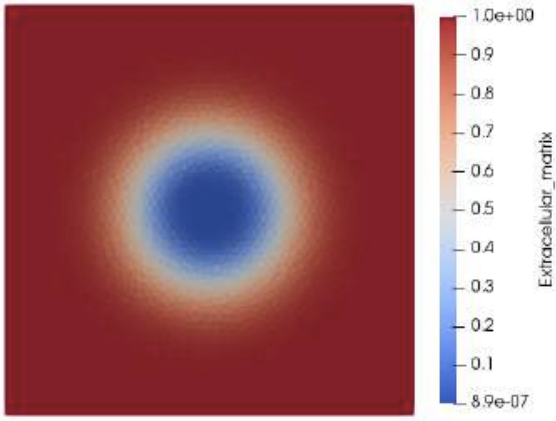} &\includegraphics[width=43mm]{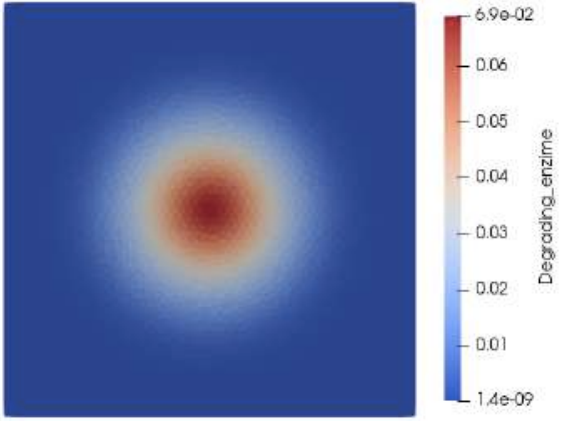} \\[3mm]
		$t=10$ & \includegraphics[width=43mm]{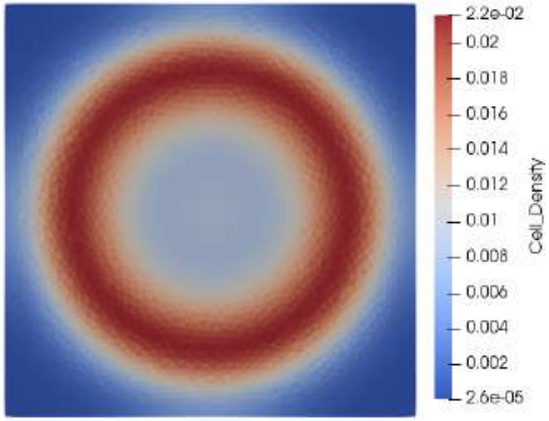} & \includegraphics[width=43mm]{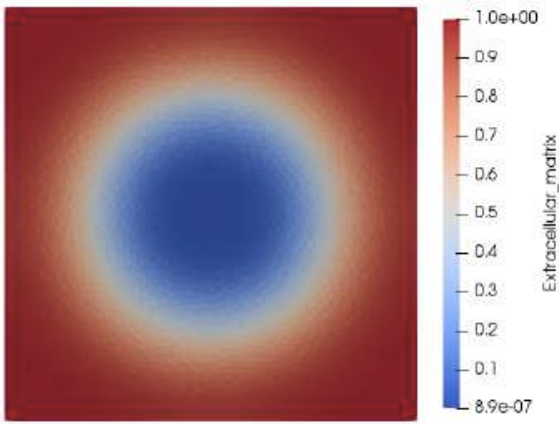} &\includegraphics[width=43mm]{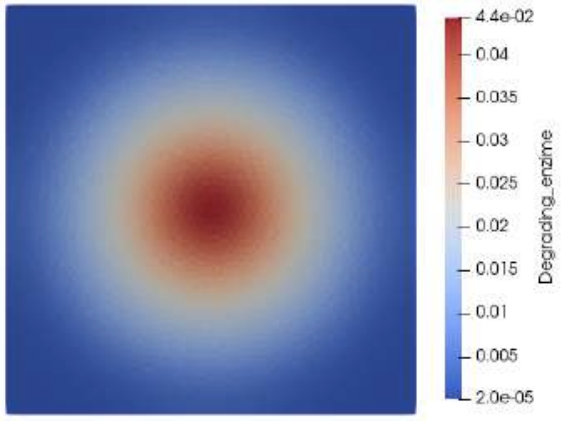} \\[3mm]
		$t=15$ & \includegraphics[width=43mm]{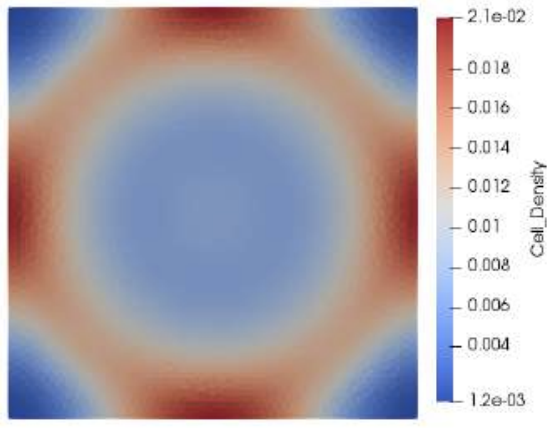} & \includegraphics[width=43mm]{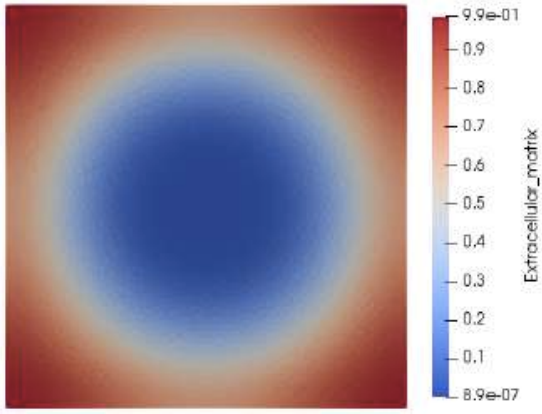} &\includegraphics[width=43mm]{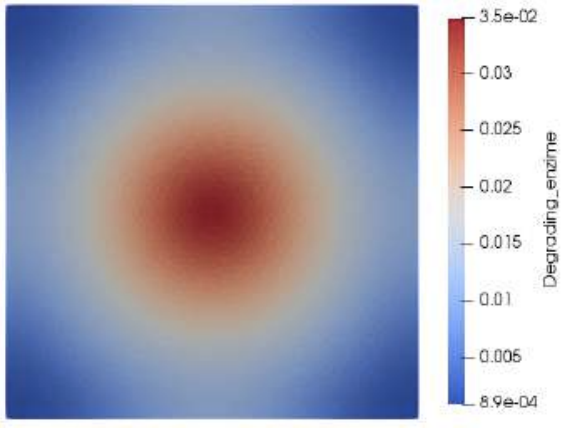}
		\end{tabular}
	\end{flushleft}
	\vspace{-0.5 cm}
\figcaption{Behavior of the scheme \textbf{UVMs} in Test 1 for $\mu_u=0$.} \label{fig:RBM3}
\end{minipage}

\vspace{1 cm}

 In Figures \ref{fig:RBM4} and \ref{fig:RBM5} we show the spatio-temporal evolution of the invasion of the extracellular matrix by the cancer cells with proliferation coefficient $\mu_u=2,$ for the schemes \textbf{UVM$\sigma$} and \textbf{UVMs}, respectively. 
The tumor growth (via proliferation) repopulates the regions where cancer cells were lacking, and it becomes more invasive. Like the case of no proliferation, the numerical simulations for both schemes show a very similar behavior; with the difference of the negative values taken of the cell density computed with the scheme \textbf{UVM$\sigma$}, contrasted with the positivity always evidenced by the scheme \textbf{UVMs} (see Figures \ref{fig:RBM4}, \ref{fig:RBM5} and \ref{fig:MM}); which is in agreement with the theoretical positivity results proved in Subsection \ref{SSPW}. We highlight that the negative values taken for the scheme \textbf{UVM$\sigma$} are very small (of order $10^{-5}$) which do not cause a significant distortion in the discrete variables obtained (for example, no spurious oscillations are evident as a result of these negative values or another strange behaviors).

 \begin{minipage}{\textwidth}
	\begin{flushleft}
		\begin{tabular}{cccc}
		 {\bf Time} & {\bf Cell density} &  {\bf Extracellular matrix} & {\bf Degrading enzyme}  \\[2mm]
		 $t=1$ &		\includegraphics[width=43mm]{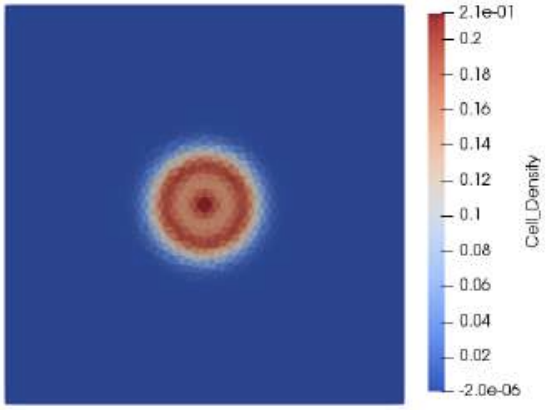} & \includegraphics[width=43mm]{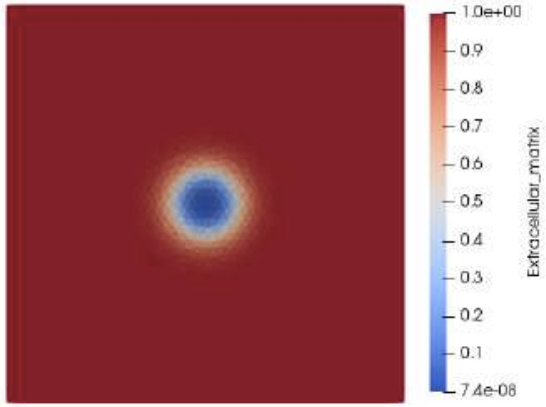} & \includegraphics[width=43mm]{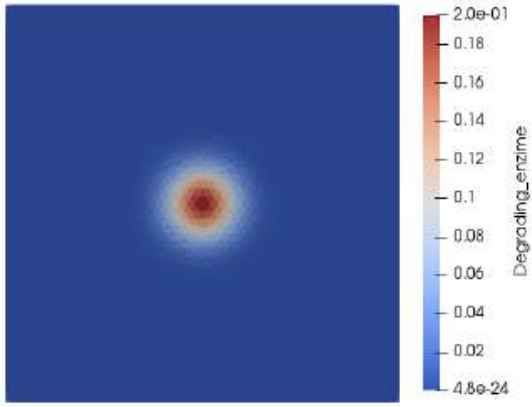}  \\[3mm]
		$t=5$ & \includegraphics[width=43mm]{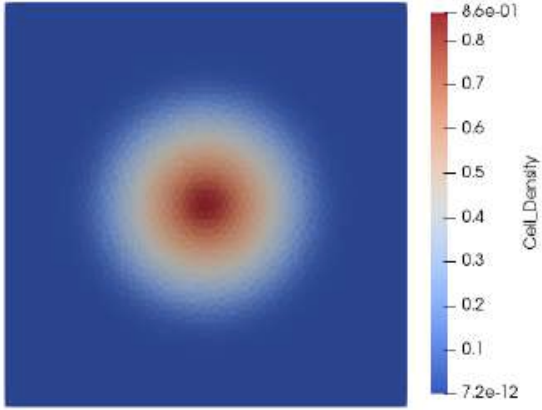} & \includegraphics[width=43mm]{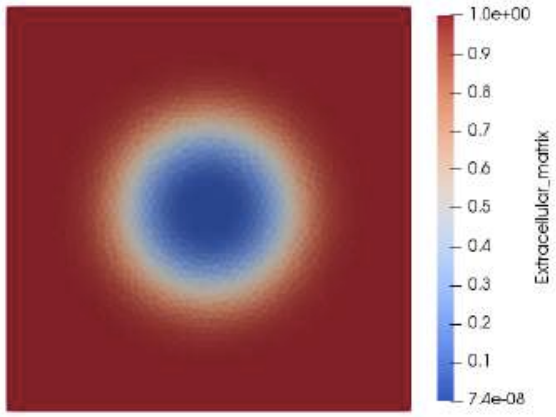} &\includegraphics[width=43mm]{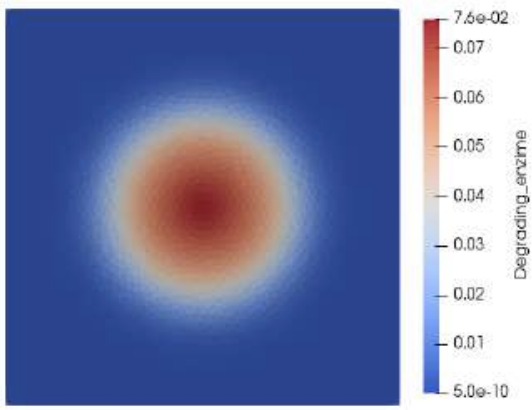} \\[3mm]
		$t=10$ & \includegraphics[width=43mm]{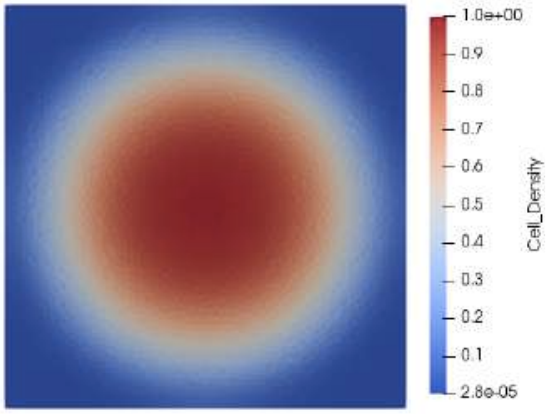} & \includegraphics[width=43mm]{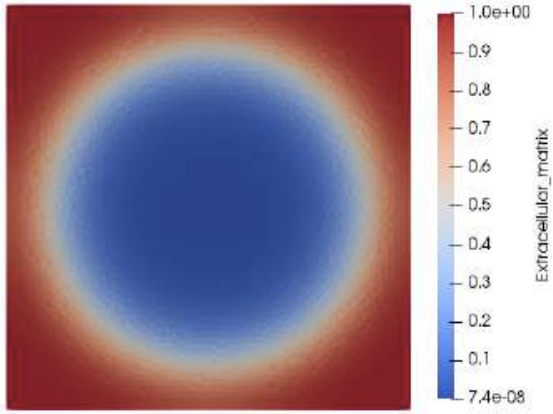} &\includegraphics[width=43mm]{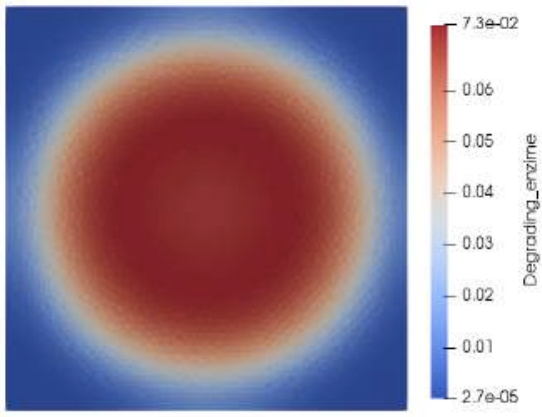} \\[3mm]
		$t=15$ & \includegraphics[width=43mm]{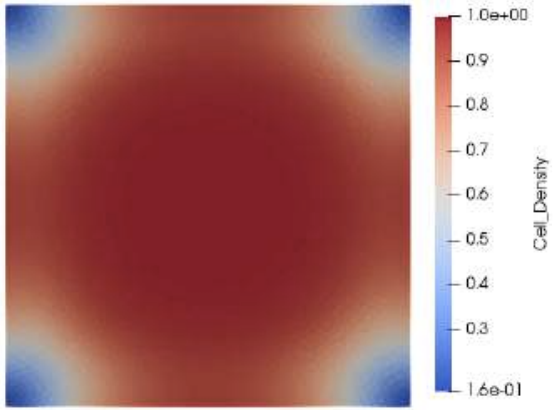} & \includegraphics[width=43mm]{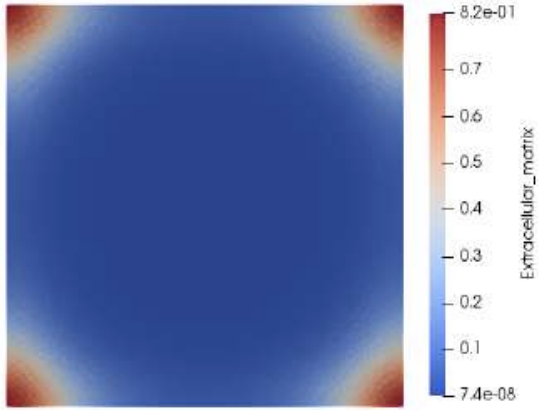} &\includegraphics[width=43mm]{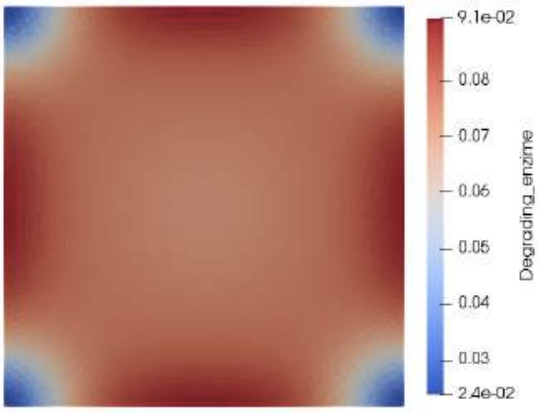} 
		\end{tabular}
	\end{flushleft}
	\vspace{-0.5 cm}
\figcaption{Behavior of the scheme \textbf{UVM$\sigma$} in Test 1 for $\mu_u=2$.} \label{fig:RBM4}
\end{minipage}

\begin{minipage}{\textwidth}
	\begin{flushleft}
		\begin{tabular}{cccc}
		 {\bf Time} & {\bf Cell density} &  {\bf Extracellular matrix} & {\bf Degrading enzyme}  \\[2mm]
		 $t=1$ &		\includegraphics[width=43mm]{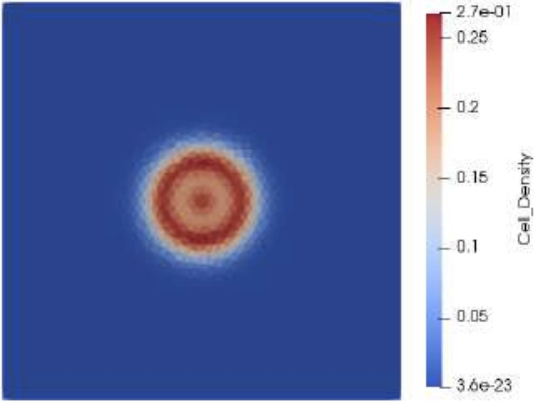} & \includegraphics[width=43mm]{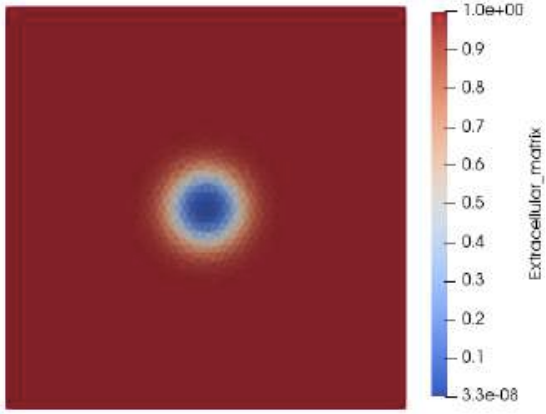} & \includegraphics[width=43mm]{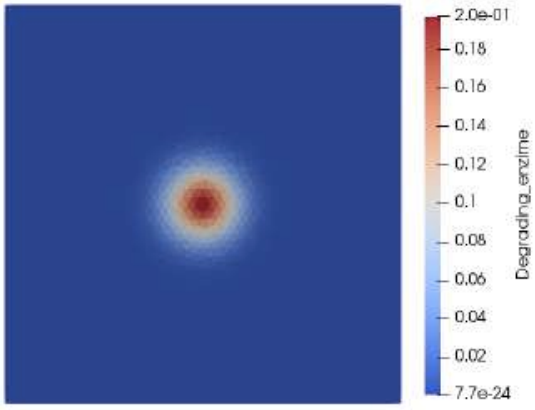}  \\[3mm]
		$t=5$ & \includegraphics[width=43mm]{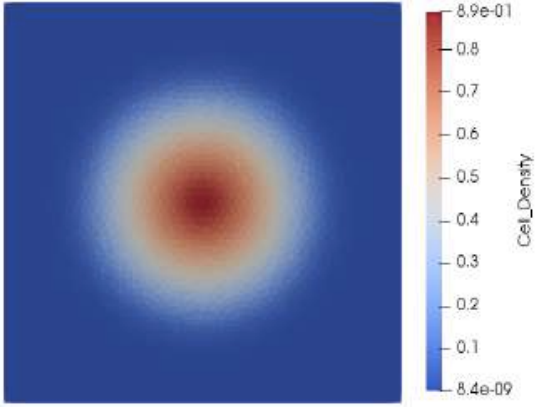} & \includegraphics[width=43mm]{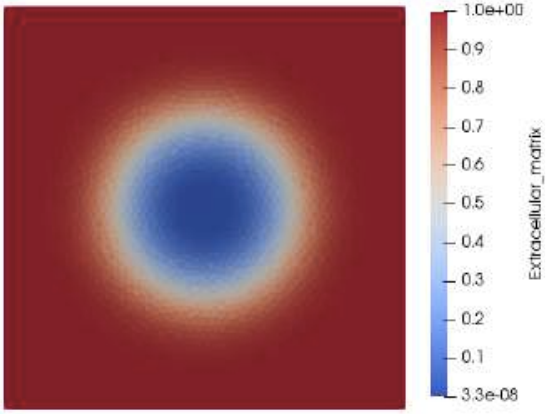} &\includegraphics[width=43mm]{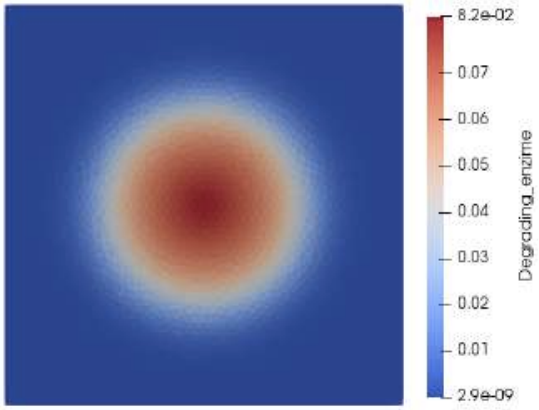} \\[3mm]
		$t=10$ & \includegraphics[width=43mm]{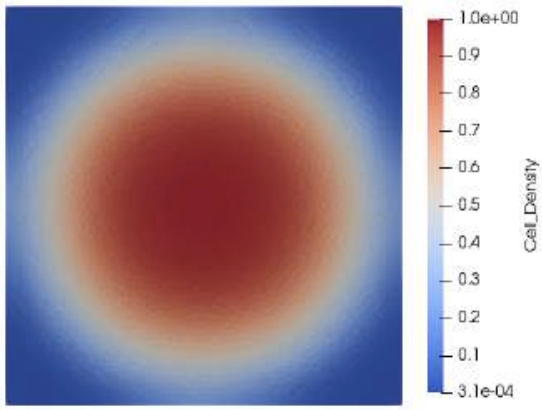} & \includegraphics[width=43mm]{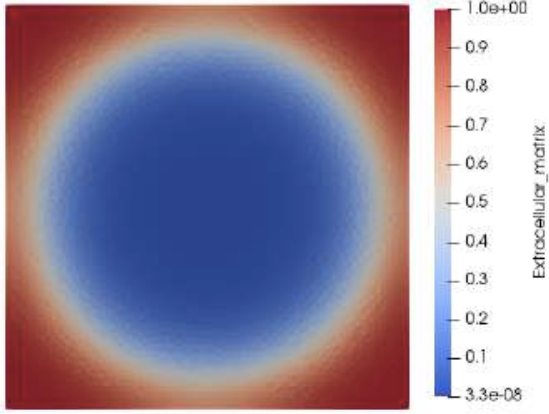} &\includegraphics[width=43mm]{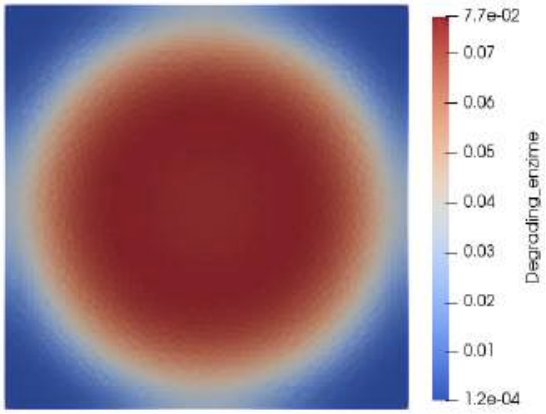} \\[3mm]
		$t=15$ & \includegraphics[width=43mm]{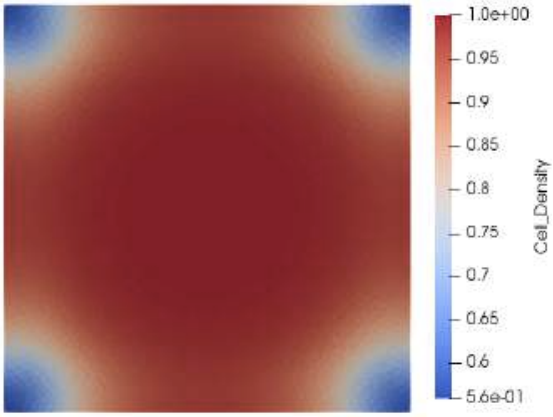} & \includegraphics[width=43mm]{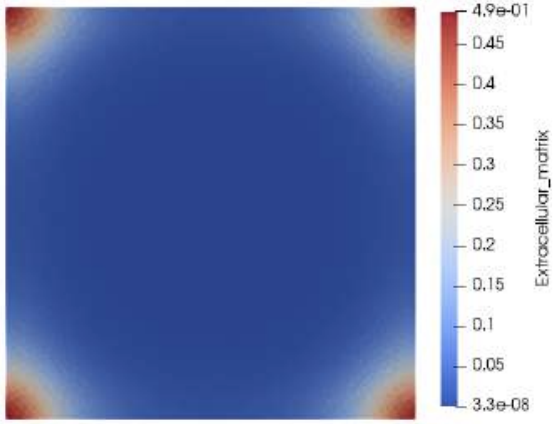} &\includegraphics[width=43mm]{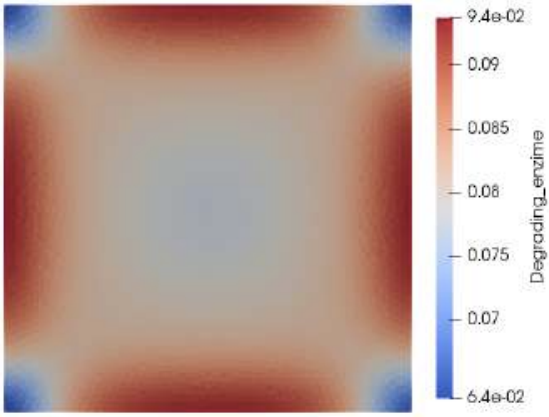}
		\end{tabular}
	\end{flushleft}
	\vspace{-0.5 cm}
\figcaption{Behavior of the scheme \textbf{UVMs} in Test 1 for $\mu_u=2$.} \label{fig:RBM5}
\end{minipage}
\\
\begin{figure}[htbp]
\centering
\subfigure[$\mu_u=0$]{\includegraphics[width=77mm]{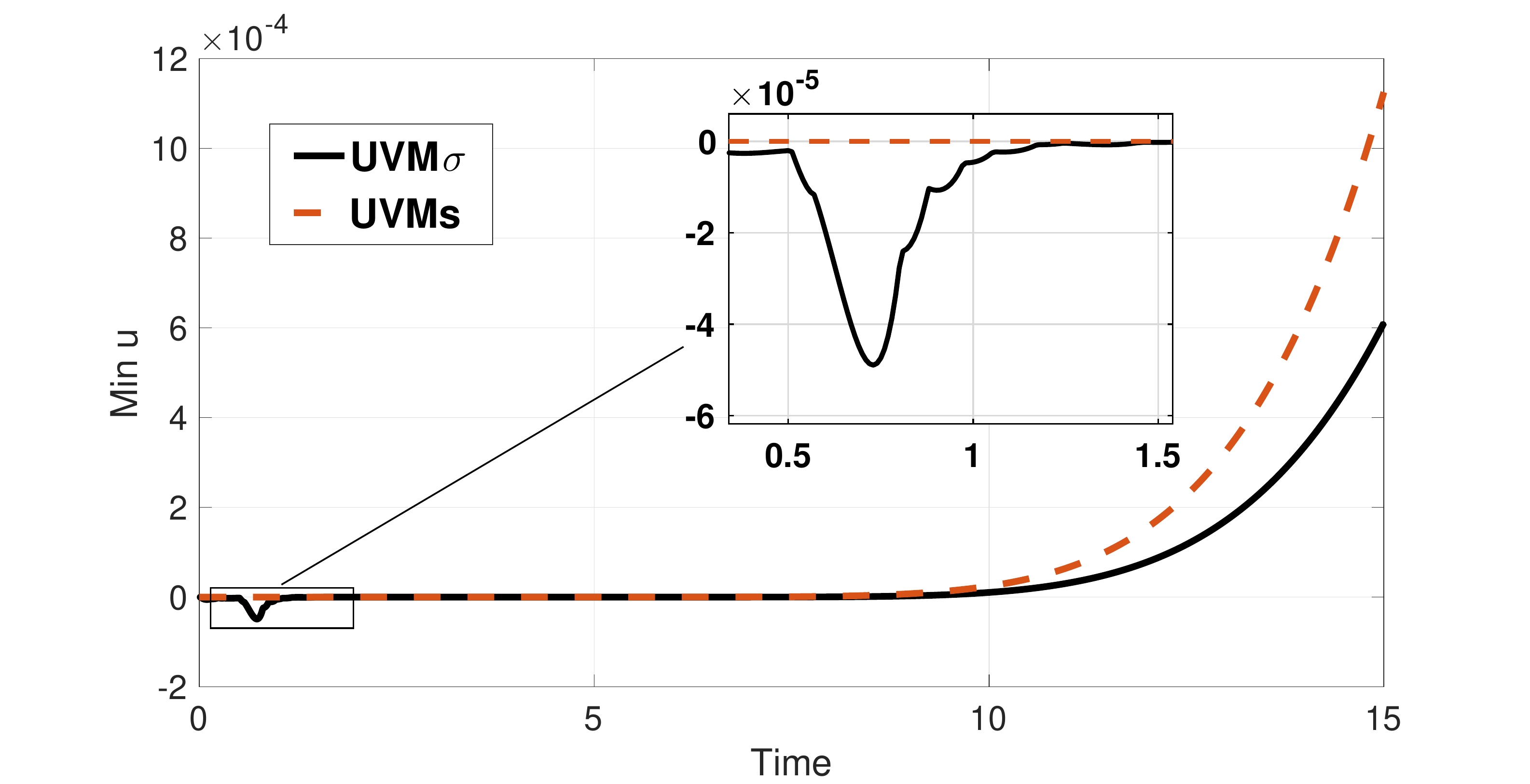}}
\subfigure[$\mu_u=2$]{\includegraphics[width=77mm]{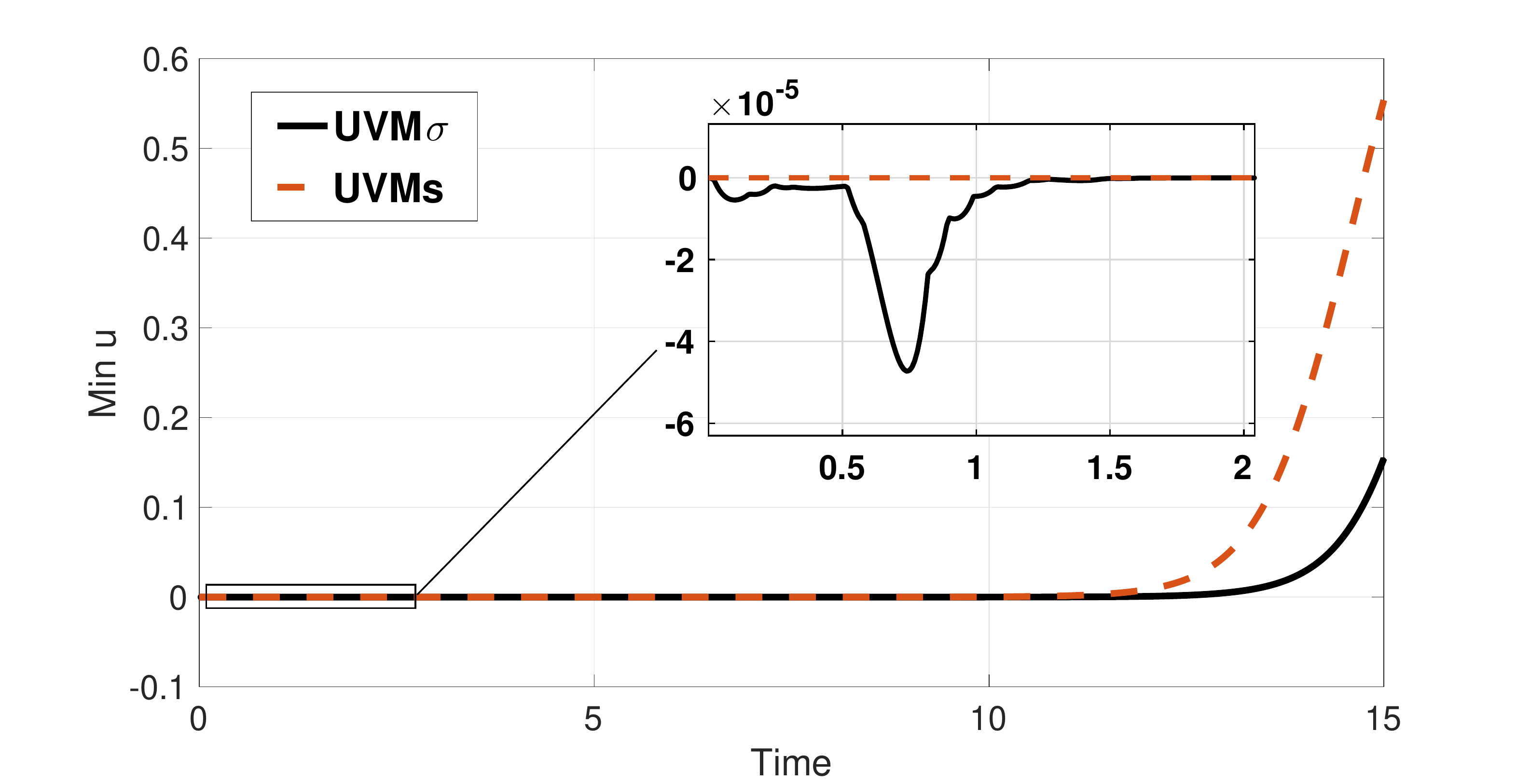}}
\caption{Minimum values of the discrete cell density $u^n_h$ computed with the schemes \textbf{UVM$\sigma$} and \textbf{UVMs} in Test 1.}\label{fig:MM}
\end{figure}

\vspace{0.6 cm}

\underline{Test 2. Heterogeneous extracellular matrix:} In this experiment we show the behavior of the schemes \textbf{UVM$\sigma$} and \textbf{UVMs} in the context of a heterogeneous extracellular matrix (see Figure \ref{fig:RBM6}). We also simulate the absence and presence of cell proliferation; for that, we consider $\mu_u=0$ and $\mu_u=2$ respectively, and the following initial conditions (see Figure \ref{fig:RBM6}):
$$
u_0= \text{exp} (-400(x-0.5)^2-400(y-0.5)^2), \ \ m_0 =0.5u_0 ,
$$
$$
v_0 = 1-\sum_{i=1}^7 \text{exp} (-b_i(x-x_i)^2-c_i(y-y_i)^2),
$$
where $b_1=b_2=800$, $b_3=b_4=b_5=600$, $b_6=400$, $b_7=100$, $c_1=c_2=100$, $c_3=c_4=c_5=200$, $c_6=300$, $c_7=50$, $x_1=y_1=y_5=0.2$, $x_2=x_6=y_3=0.5$, $y_2=0.1$, $x_3=0.3$, $x_4=0.6$, $y_4=y_7=0.7$, $x_5=x_7=0.8$ and $y_6=0.9$.\\ 
 
\begin{minipage}{\textwidth}
	\begin{flushleft}
		\begin{tabular}{ccc}
		  
		 \includegraphics[width=43mm]{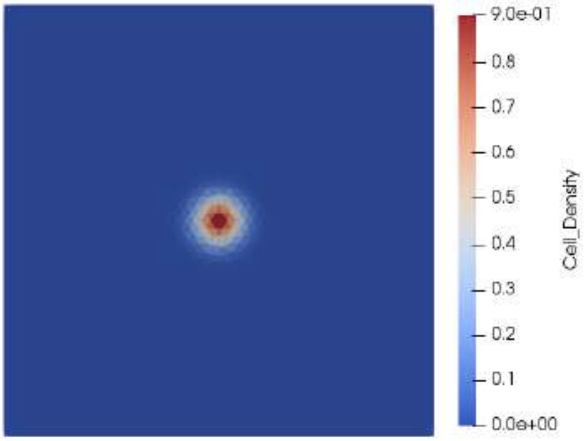} & \hspace{-2mm} \includegraphics[width=43mm]{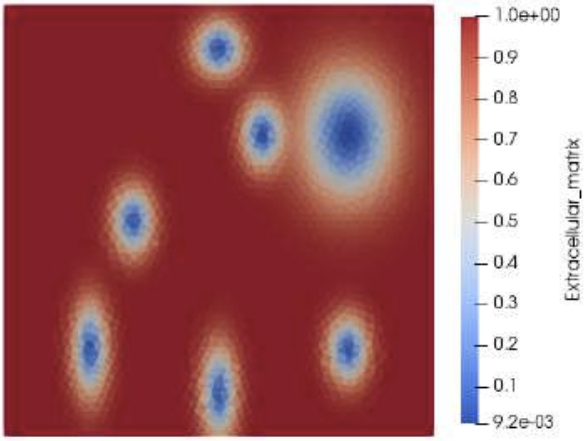} & \hspace{-2mm} \includegraphics[width=43mm]{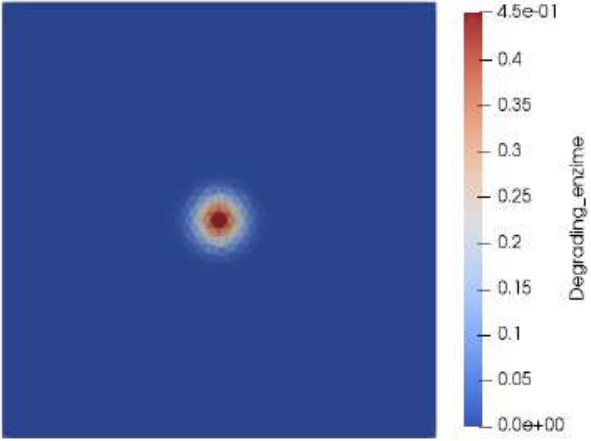}\\[2mm]
		(a) {\bf Cell density} & (b) {\bf Extracellular matrix} & (c) {\bf Degrading enzyme}  
		\end{tabular}
	\end{flushleft}
	\vspace{-0.5 cm}
\figcaption{Initial conditions in Test 2} \label{fig:RBM6}
\end{minipage}

\vspace{0.7 cm}

The evolution results for the case  $\mu_u=0$ are showed in Figures \ref{fig:RBM7} and \ref{fig:RBM8} for the schemes \textbf{UVM$\sigma$} and \textbf{UVMs}, respectively. The behavior of the cell density also reproduces the pattern reported in \cite{Anderson} in the context of heterogeneous extracellular matrix. A deterioration of the matrix is observed; sets of cancer cells emerge from the tumor body in its beginning, invading the extracellular matrix leading to possible metastasis.  The effect of the heterogeneous matrix on the cancer cell dynamics can be seen as the cancer cells approach their steady distribution. Again, the behavior of evolution of the unknowns in both schemes is similar, except in terms of the positivity of $u^n_h$ (see Figures  \ref{fig:MM2}, \ref{fig:RBM7} and \ref{fig:RBM8}). On the other hand, Figures \ref{fig:RBM9} and \ref{fig:RBM10} show the complementary effect of cell proliferation.

\begin{figure}[htbp]
\centering
\subfigure[$\mu_u=0$]{\includegraphics[width=77mm]{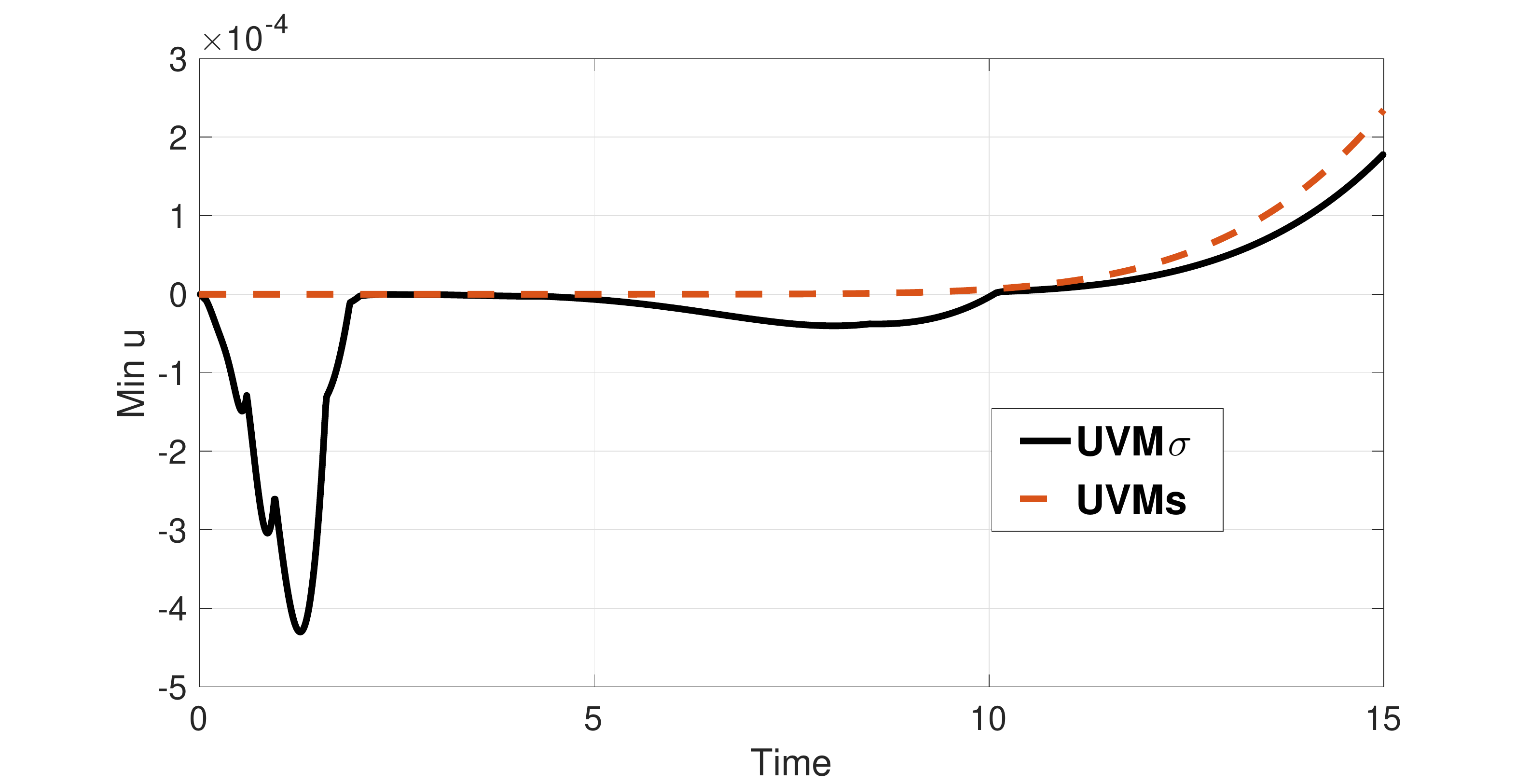}}
\subfigure[$\mu_u=2$]{\includegraphics[width=77mm]{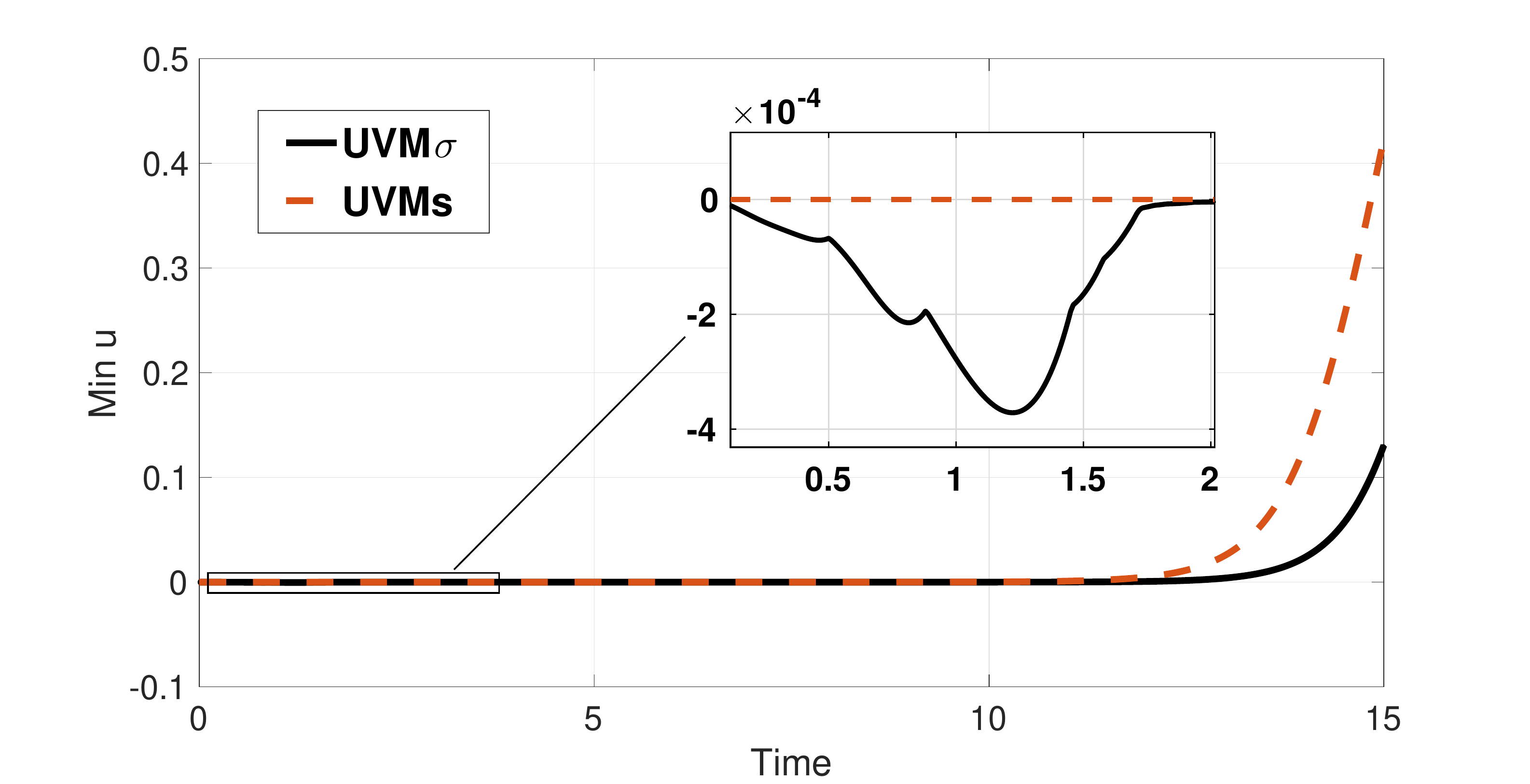}}
\caption{Minimum values of the discrete cell density $u^n_h$ computed with the schemes \textbf{UVM$\sigma$} and \textbf{UVMs} in Test 2.}\label{fig:MM2}
\end{figure}

\begin{minipage}{\textwidth}
	\begin{flushleft}
		\begin{tabular}{cccc}
		 {\bf Time} & {\bf Cell density} &  {\bf Extracellular matrix} & {\bf Degrading enzyme}  \\[2mm]
		 $t=1$ &		\includegraphics[width=43mm]{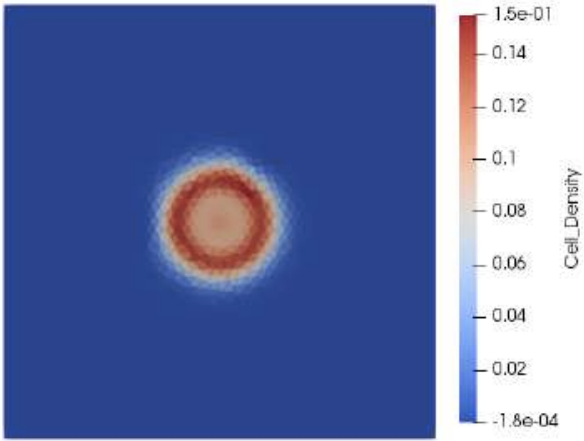} & \includegraphics[width=43mm]{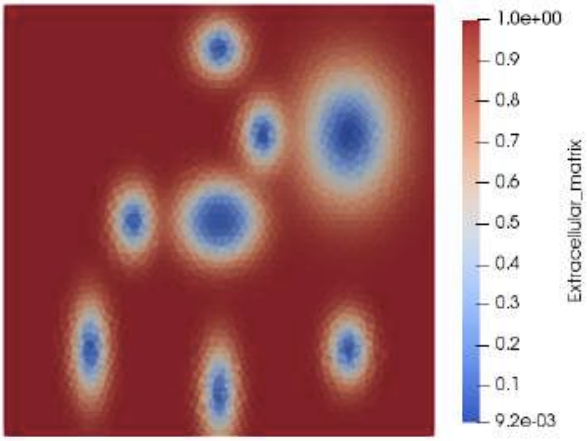} & \includegraphics[width=43mm]{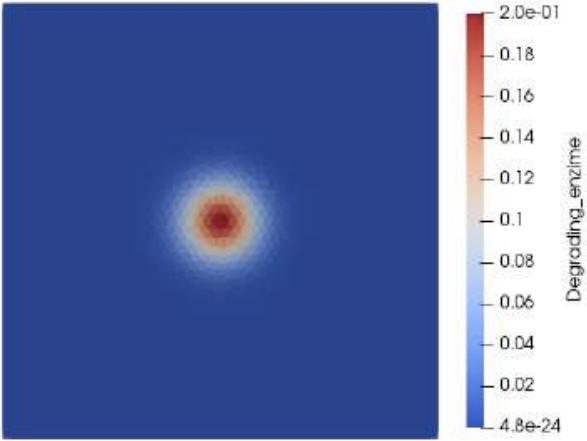}  \\[3mm]
		$t=5$ & \includegraphics[width=43mm]{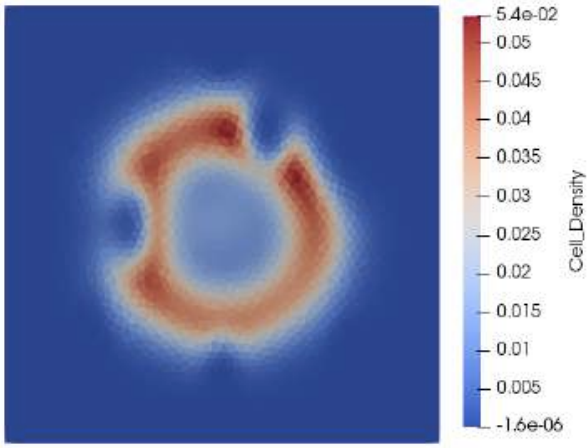} & \includegraphics[width=43mm]{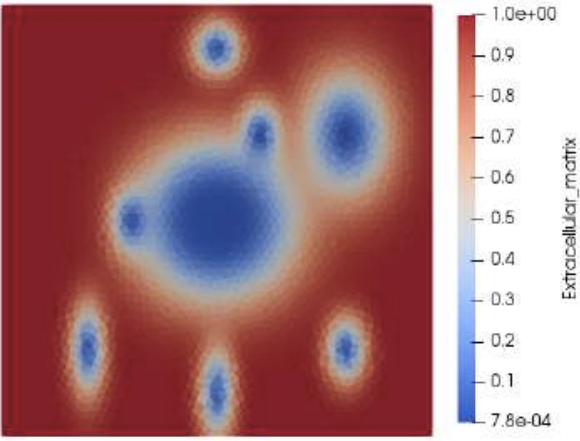} &\includegraphics[width=43mm]{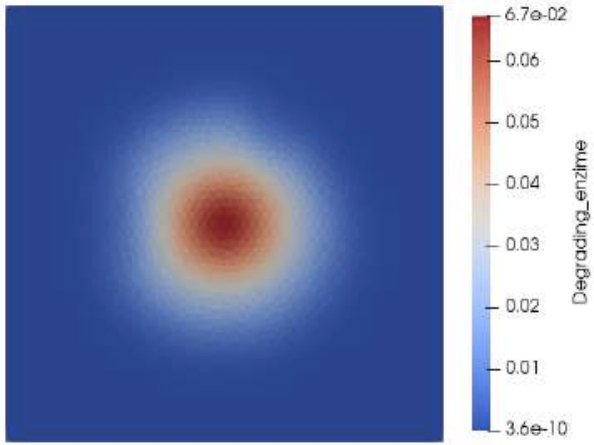} \\[3mm]
		$t=10$ & \includegraphics[width=43mm]{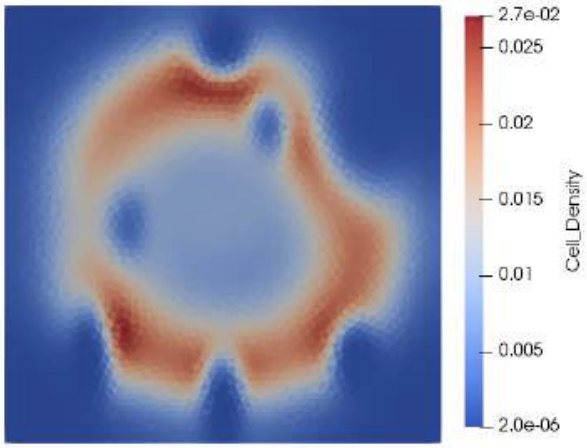} & \includegraphics[width=43mm]{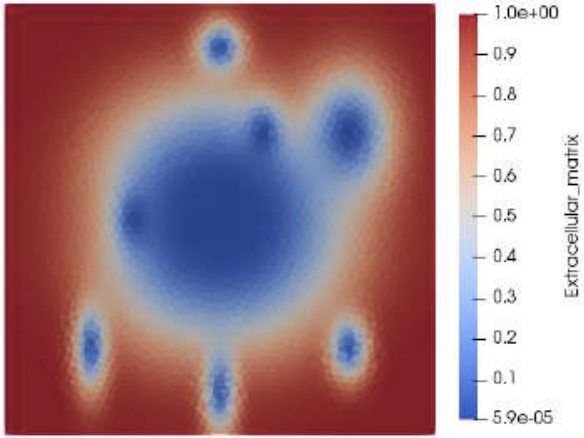} &\includegraphics[width=43mm]{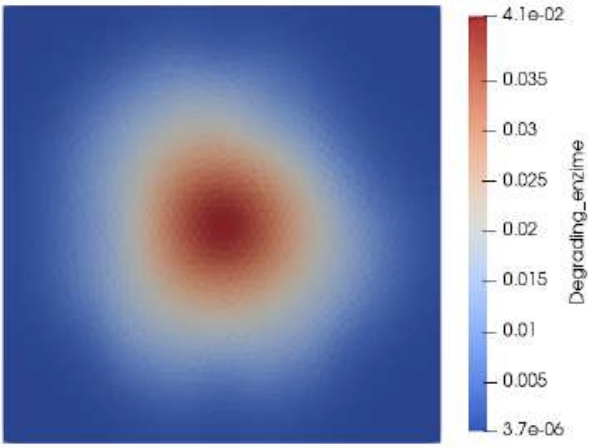} \\[3mm]
		$t=15$ & \includegraphics[width=43mm]{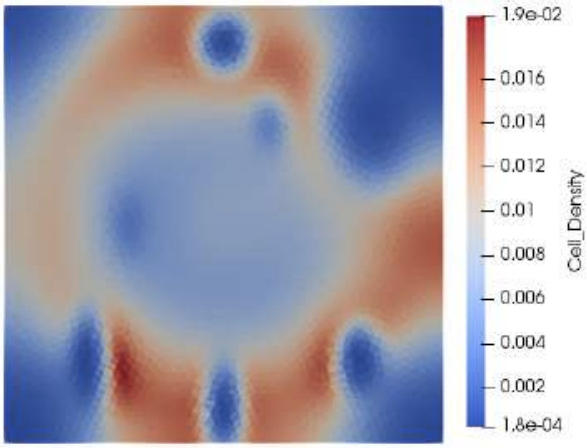} & \includegraphics[width=43mm]{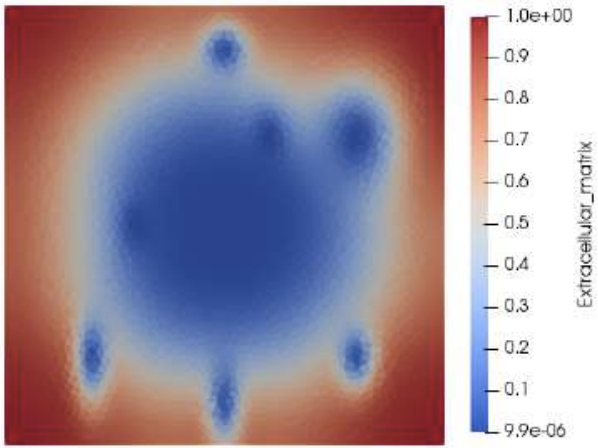} &\includegraphics[width=43mm]{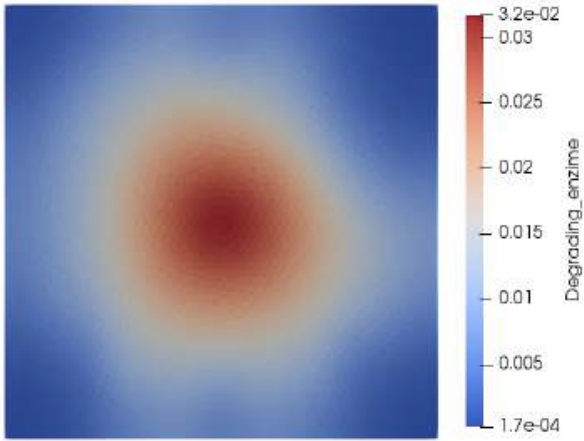} 
		\end{tabular}
	\end{flushleft}
	\vspace{-0.5 cm}
\figcaption{Behavior of the scheme \textbf{UVM$\sigma$} in Test 2 for $\mu_u=0$.} \label{fig:RBM7}
\end{minipage}

\begin{minipage}{\textwidth}
	\begin{flushleft}
		\begin{tabular}{cccc}
		 {\bf Time} & {\bf Cell density} &  {\bf Extracellular matrix} & {\bf Degrading enzyme}  \\[2mm]
		 $t=1$ &		\includegraphics[width=43mm]{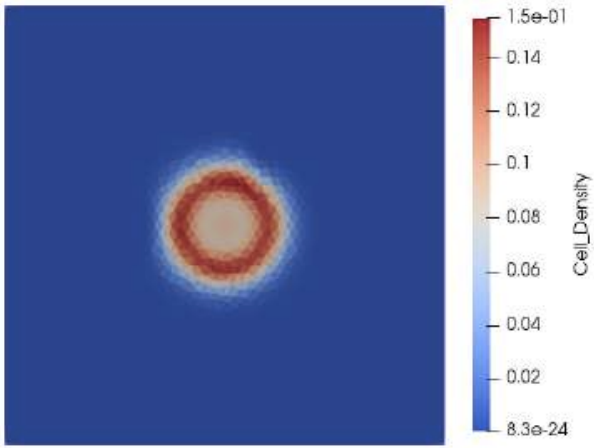} & \includegraphics[width=43mm]{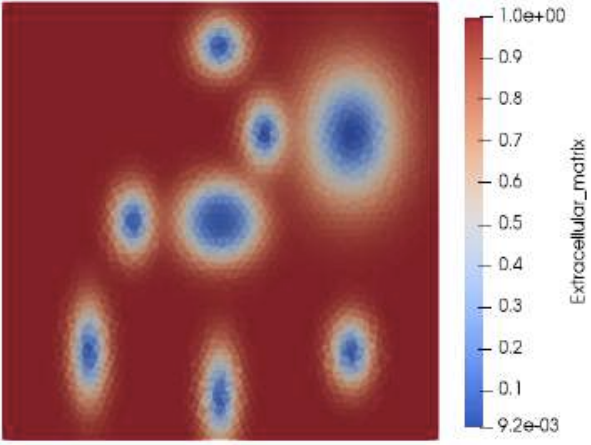} & \includegraphics[width=43mm]{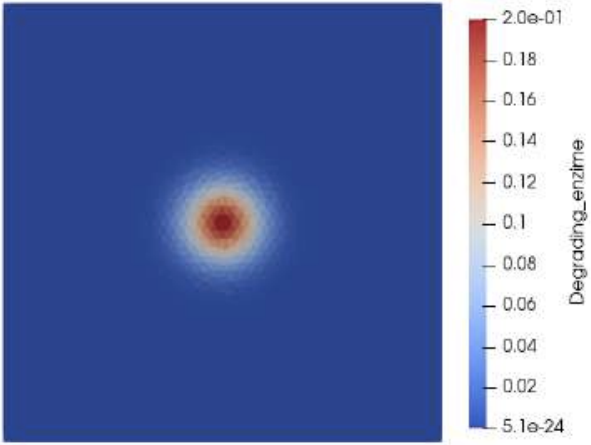}  \\[3mm]
		$t=5$ & \includegraphics[width=43mm]{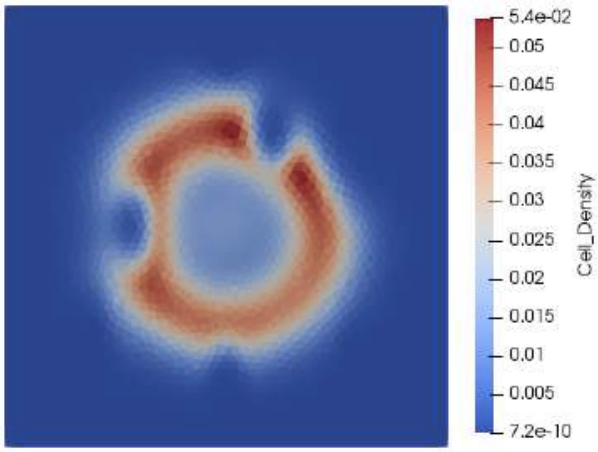} & \includegraphics[width=43mm]{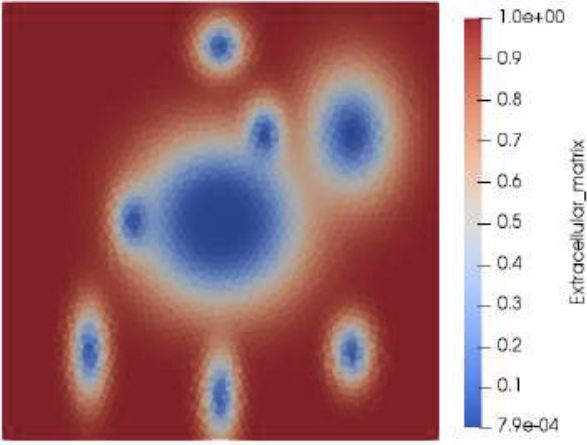} &\includegraphics[width=43mm]{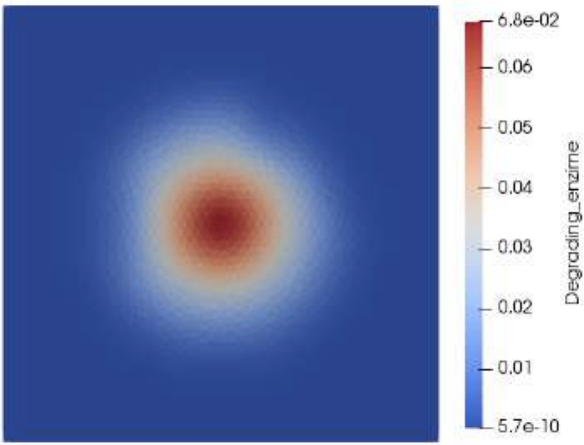} \\[3mm]
		$t=10$ & \includegraphics[width=43mm]{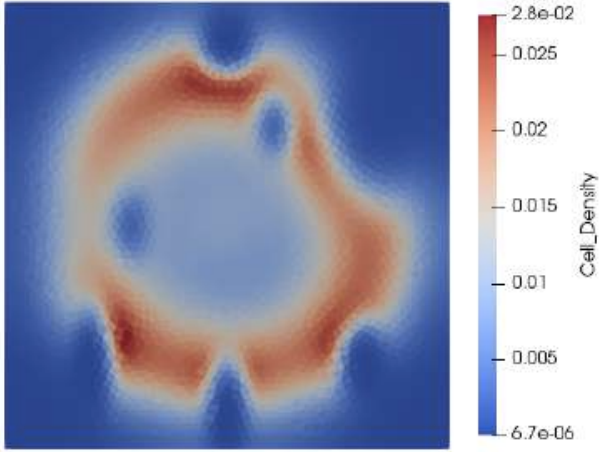} & \includegraphics[width=43mm]{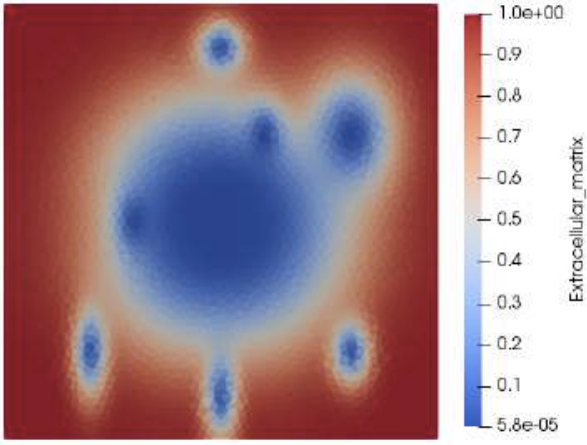} &\includegraphics[width=43mm]{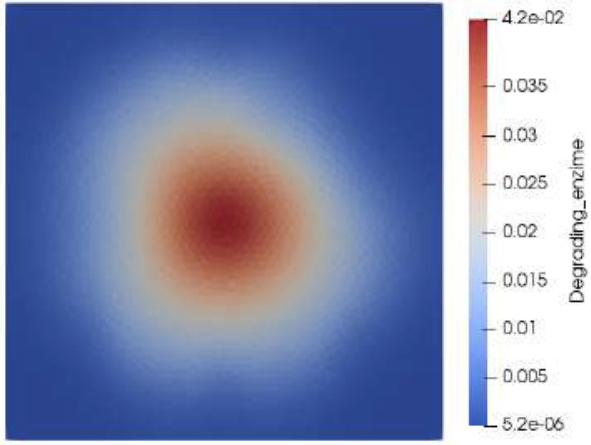} \\[3mm]
			$t=15$ & \includegraphics[width=43mm]{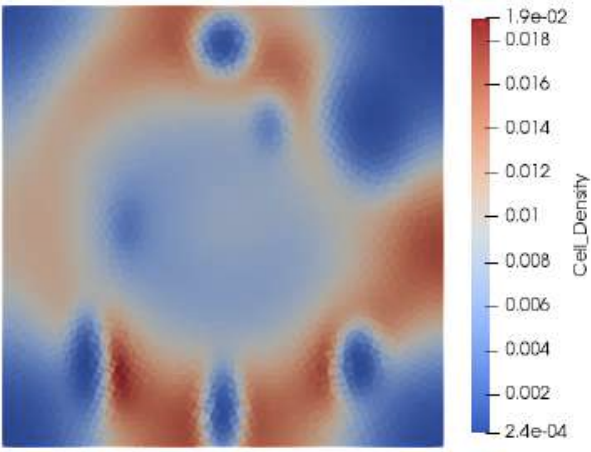} & \includegraphics[width=43mm]{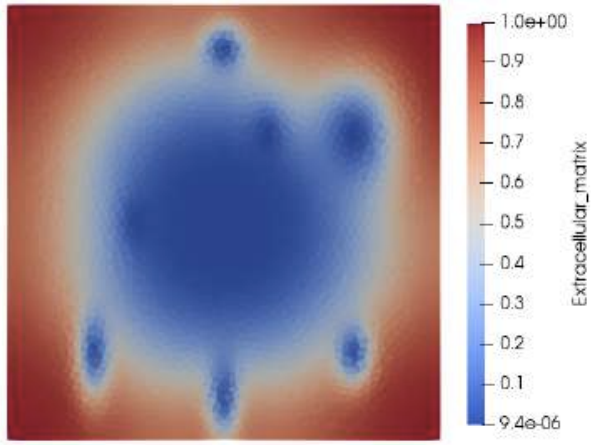} &\includegraphics[width=43mm]{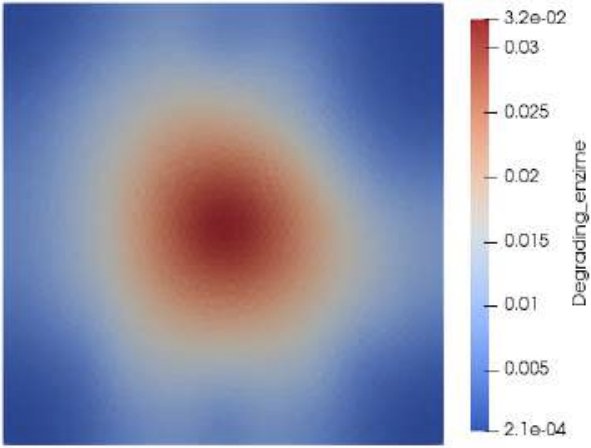}
		\end{tabular}
	\end{flushleft}
	\vspace{-0.5cm}
\figcaption{Behavior of the scheme \textbf{UVMs} in Test 2 for $\mu_u=0$.} \label{fig:RBM8}
\end{minipage}

\begin{minipage}{\textwidth}
	\begin{flushleft}
		\begin{tabular}{cccc}
		 {\bf Time} & {\bf Cell density} &  {\bf Extracellular matrix} & {\bf Degrading enzyme}  \\[2mm]
		 $t=1$ &		\includegraphics[width=43mm]{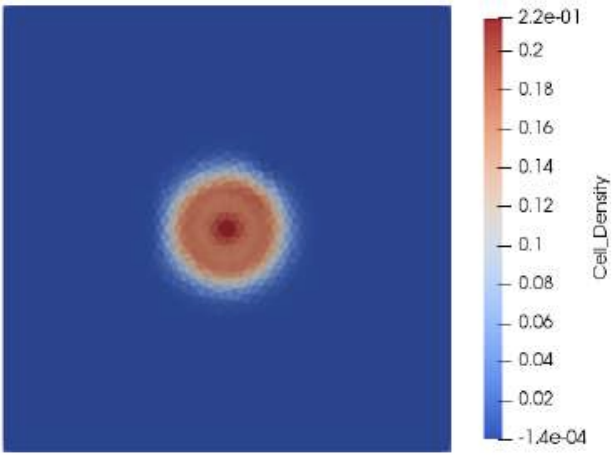} & \includegraphics[width=43mm]{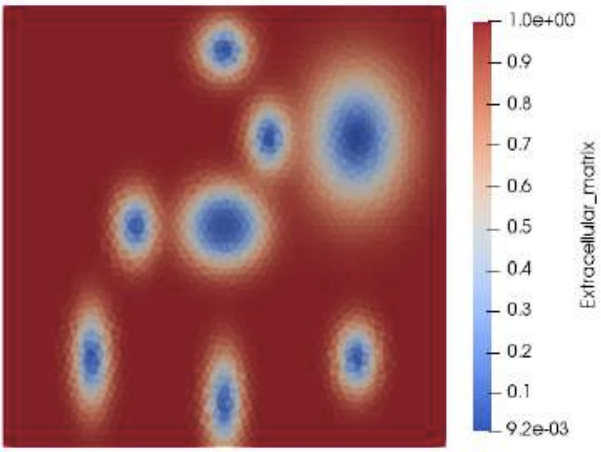} & \includegraphics[width=43mm]{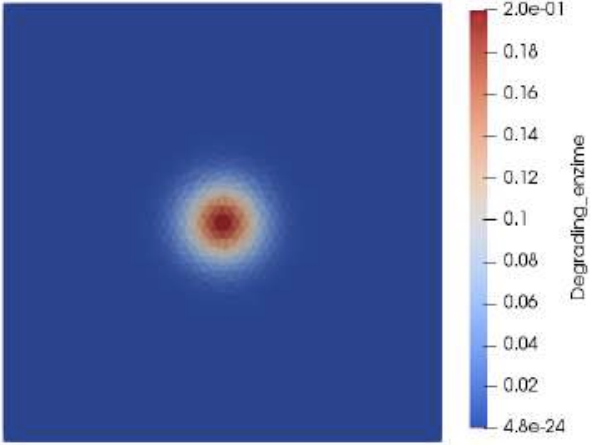}  \\[3mm]
		$t=5$ & \includegraphics[width=43mm]{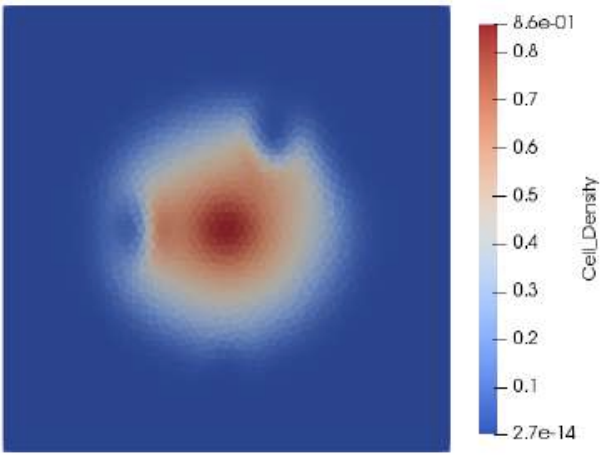} & \includegraphics[width=43mm]{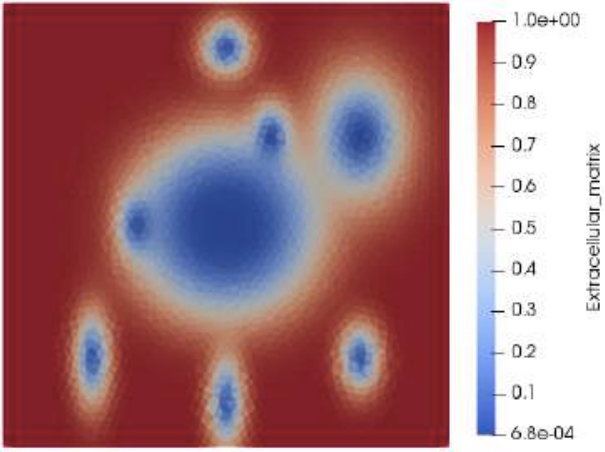} &\includegraphics[width=43mm]{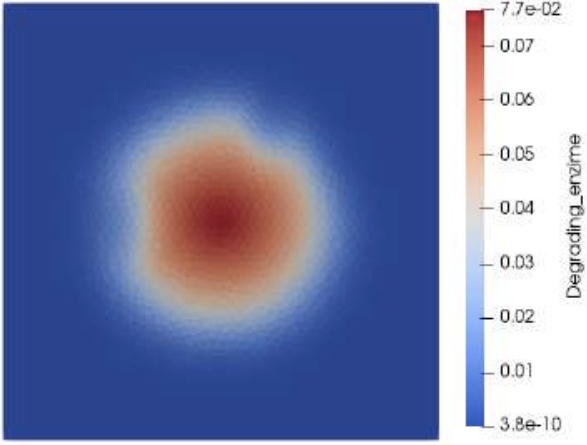} \\[3mm]
		$t=10$ & \includegraphics[width=43mm]{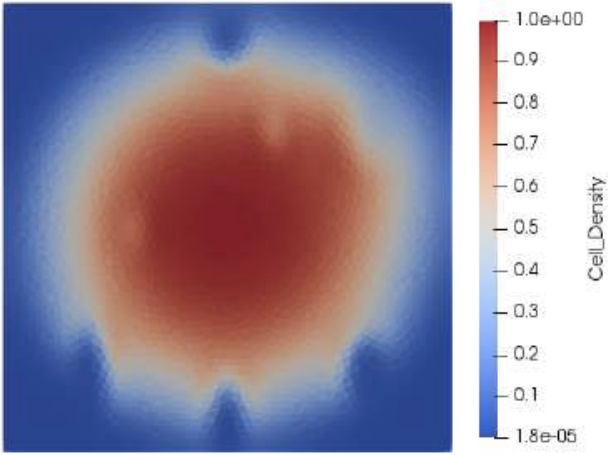} & \includegraphics[width=43mm]{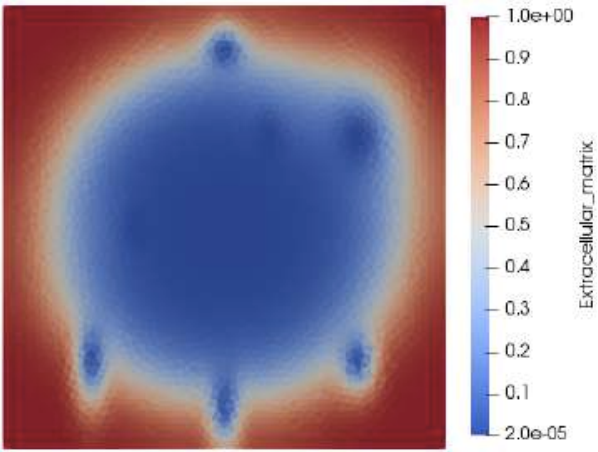} &\includegraphics[width=43mm]{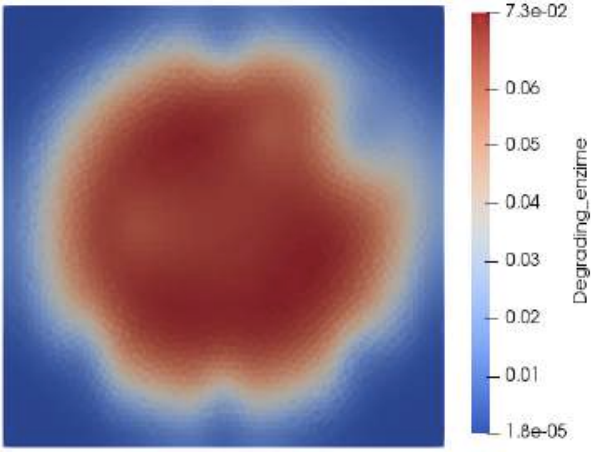} \\[3mm]
		$t=15$ & \includegraphics[width=43mm]{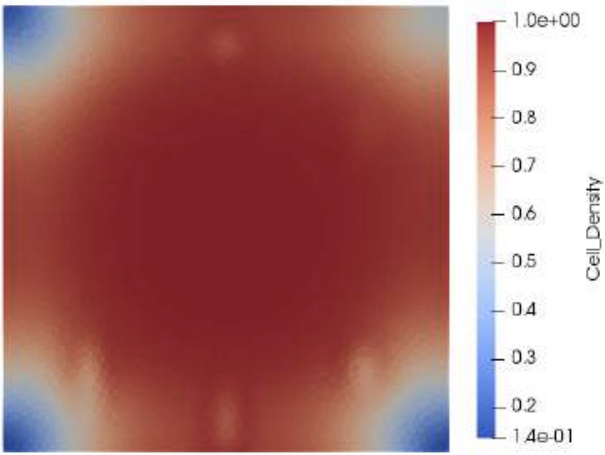} & \includegraphics[width=43mm]{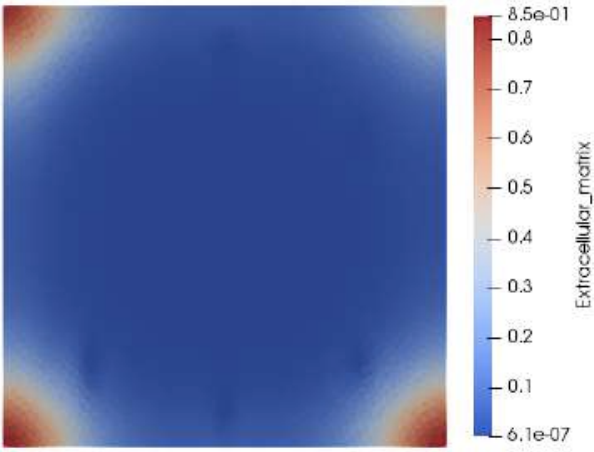} &\includegraphics[width=43mm]{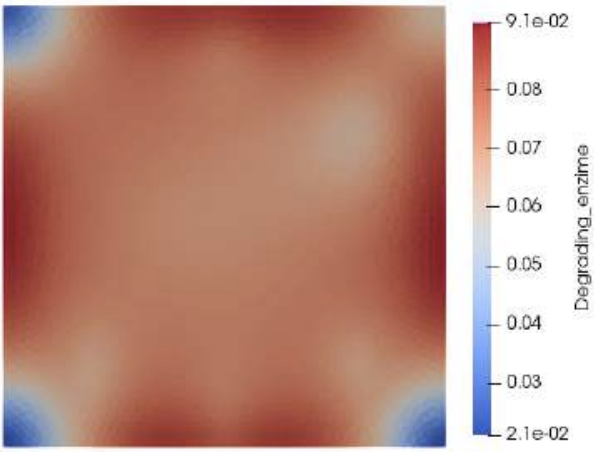} 
		\end{tabular}
	\end{flushleft}
	\vspace{-0.5 cm}
\figcaption{Behavior of the scheme \textbf{UVM$\sigma$} in Test 2 for $\mu_u=2$.} \label{fig:RBM9}
\end{minipage}

\begin{minipage}{\textwidth}
	\begin{flushleft}
		\begin{tabular}{cccc}
		 {\bf Time} & {\bf Cell density} &  {\bf Extracellular matrix} & {\bf Degrading enzyme}  \\[2mm]
		 $t=1$ &		\includegraphics[width=43mm]{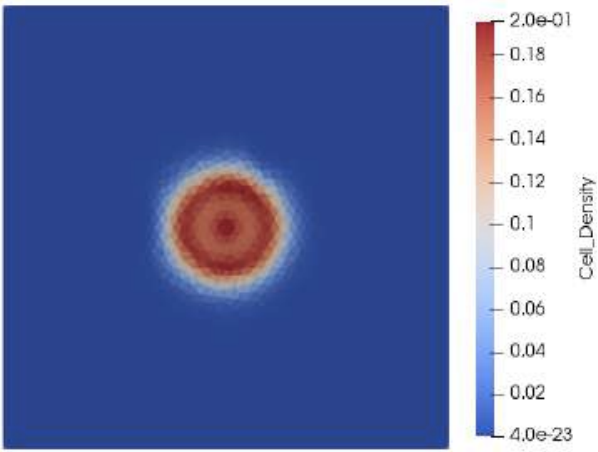} & \includegraphics[width=43mm]{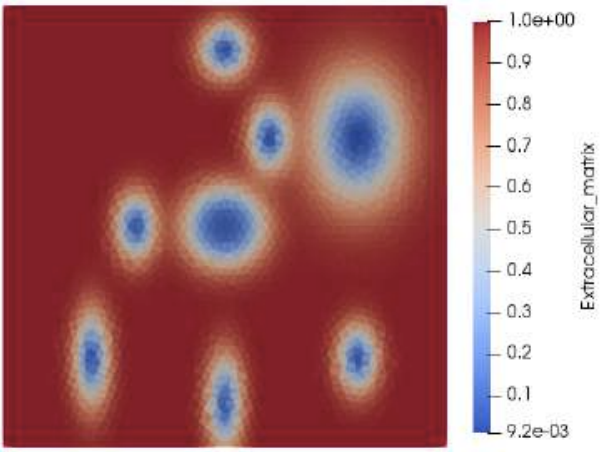} & \includegraphics[width=43mm]{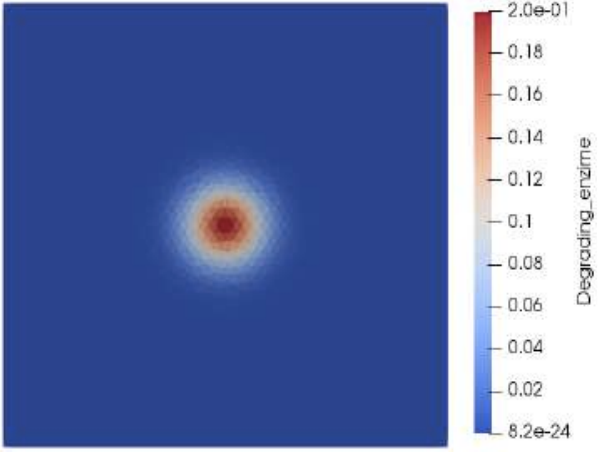}  \\[3mm]
		$t=5$ & \includegraphics[width=43mm]{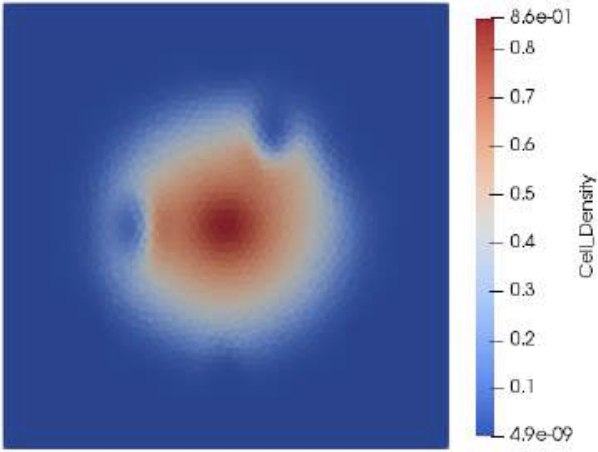} & \includegraphics[width=43mm]{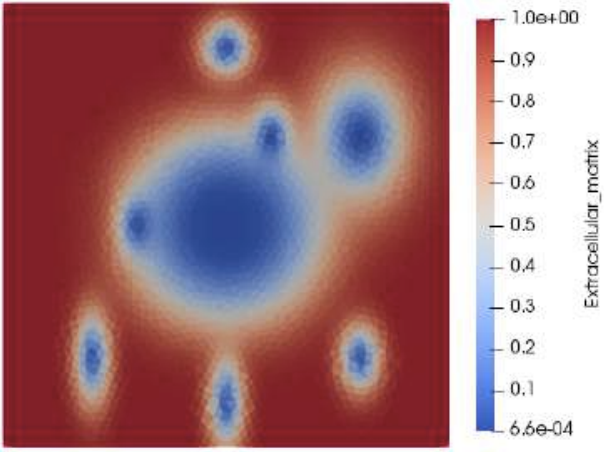} &\includegraphics[width=43mm]{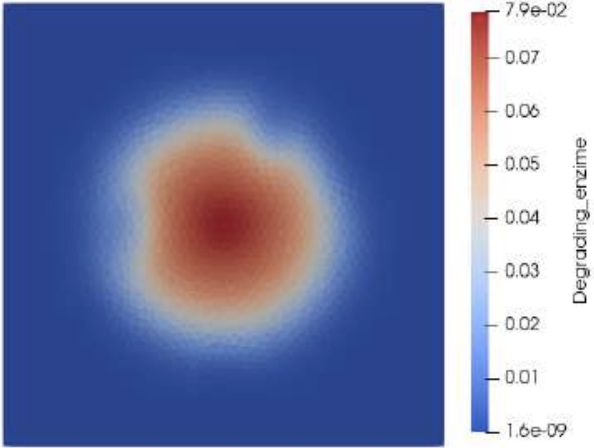} \\[3mm]
		$t=10$ & \includegraphics[width=43mm]{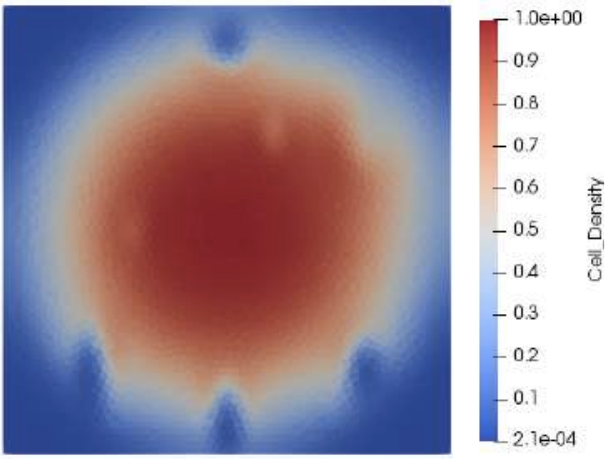} & \includegraphics[width=43mm]{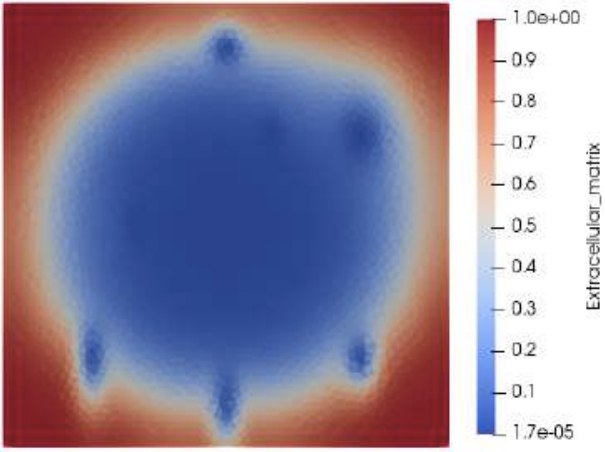} &\includegraphics[width=43mm]{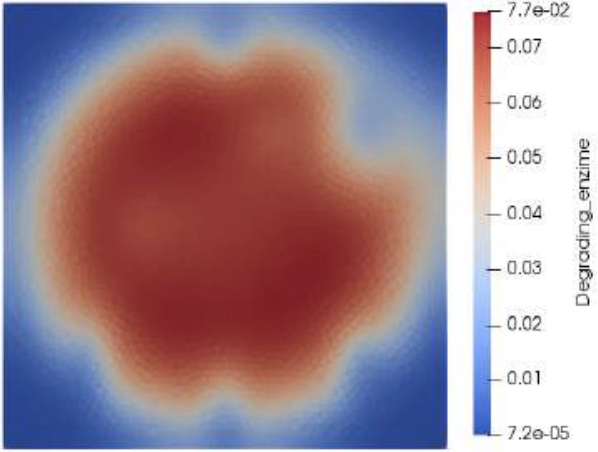} \\[3mm]
		$t=15$ & \includegraphics[width=43mm]{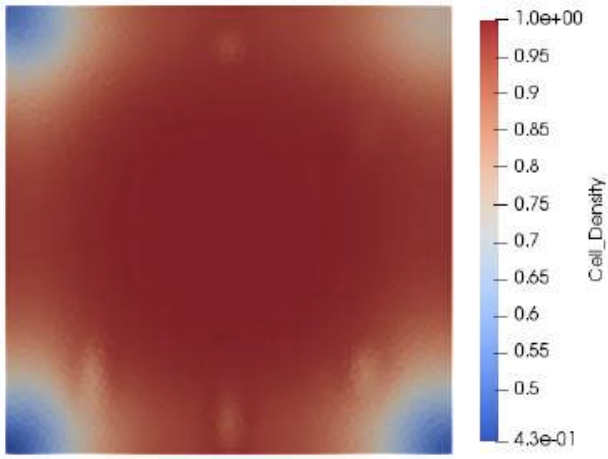} & \includegraphics[width=43mm]{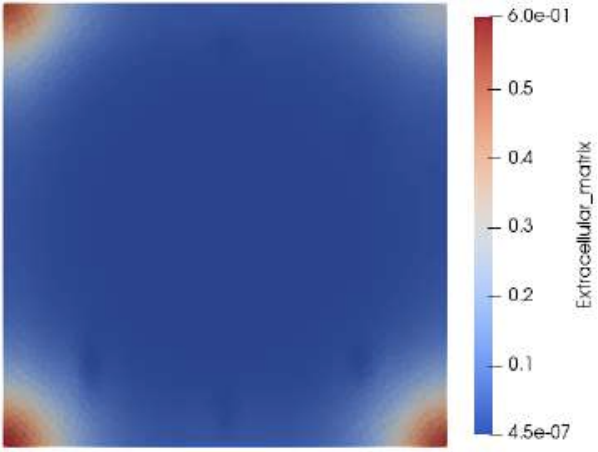} &\includegraphics[width=43mm]{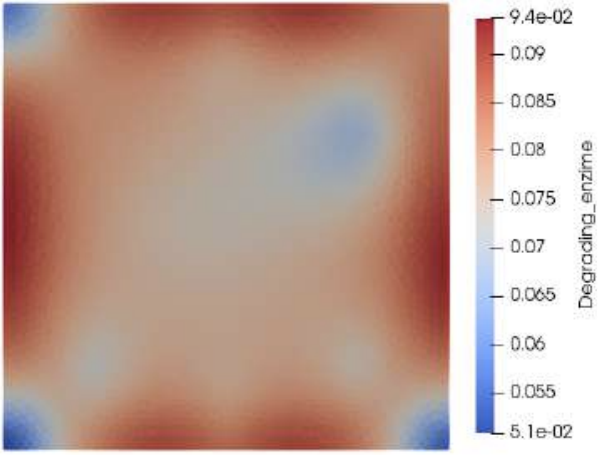}
		\end{tabular}
	\end{flushleft}
	\vspace{-0.5 cm}
\figcaption{Behavior of the scheme \textbf{UVMs} in Test 2 for $\mu_u=2$.} \label{fig:RBM10}
\end{minipage}

\bigskip

%\begin{figure}[htbp] 
%	\centering 
%	\subfigure[Discrete cell density at time t=0]{\includegraphics[width=60mm]{n0}} \hspace{1cm}
%\subfigure[Discrete chemical signal at time t=0]{\includegraphics[width=60mm]{c0}} 
%	\subfigure[Discrete cell density at time t=12e-5]{\includegraphics[width=60mm]{n1}}  \hspace{1cm}
%	\subfigure[Discrete chemical signal  at time t=12e-5]{\includegraphics[width=60mm]{c1}} 
%		\subfigure[Discrete cell density at time t=30e-5]{\includegraphics[width=60mm]{n2}} \hspace{1cm}
%	\subfigure[Discrete chemical signal  at time t=30e-5]{\includegraphics[width=60mm]{c2}} 
%	\caption{Cell density vs Chemical concentration.} \label{fig:NC1}
%\end{figure}
%\end{minipage}

\section*{Acknowledgements}
The authors have been supported by Vicerrector\'ia de Investigaci\'on y Extensi\'on of Universidad Industrial de Santander, Capital Semilla project, code 2491.

\end{document}